\let\ORIlabel\label
\let\ORIrefstepcounter\refstepcounter
\AddToHook{package/hyperref/before}{
   \let\label\ORIlabel 
   \let\refstepcounter\ORIrefstepcounter}
\documentclass[hidelinks,onefignum,onetabnum,onealgnum]{siamart250106}

\usepackage{graphicx} %
\usepackage{booktabs}
\usepackage{amsfonts}
\usepackage{amsmath}
\usepackage{amssymb}
 \usepackage{array}
\usepackage{mathtools}
\usepackage{colortbl} %
\usepackage{multicol}
\usepackage{enumitem}
\usepackage{bm}
\usepackage{comment}
\usepackage[font=small]{caption}
\usepackage{pgfplots}
\definecolor{CeruleanRef}{RGB}{12,127,172}
\hypersetup{colorlinks=true,allcolors=CeruleanRef}
\hypersetup{pdfstartview=FitB,pdfpagemode=UseNone}

\usepackage{rotating}
\RequirePackage{algorithm}[0.1]

\usepackage{algpseudocode}
\algnewcommand{\Initialize}[1]{
  \State \textbf{Initialize:}
  \Statex \hspace*{\algorithmicindent}\parbox[t]{.8\linewidth}{\raggedright #1}
}
\algnewcommand{\Indent}[2]{
  \State {#1}
  \vspace{-2mm}
  \Statex \hspace*{\algorithmicindent}\parbox[t]{.9\linewidth}{\raggedright #2}
}

\usepackage{subfig}
\usepackage{hhline}
\usepackage{multirow}
 \usepackage{wrapfig}

\newcommand{\vectt}[1]{\boldsymbol{\mathbf{#1}}}

\newcommand{\op}[1]{\operatorname{#1}}
\newcommand{\dx}{\mathrm{d}x}
\newcommand{\ds}{\mathrm{d}s}

\newcommand{\supp}{\mathrm{supp}}

\newcommand{\bu}{\vectt{u}}

\newcommand{\bphi}{\vectt{\varphi}}
\newcommand{\divv}{\operatorname{div}}
\newcommand{\colto}{\colon\!}

\newcommand{\lsb}{[\![}
\newcommand{\rsb}{]\!]}

\newtheorem{remark}[theorem]{Remark}

\newtheorem{problem}{Problem}
\crefname{problem}{Problem}{Problems}

\author{
Ioannis P.~A.~Papadopoulos\thanks{%
Mathematical Institute, University of Oxford, UK}
(\email{ioannis.papadopoulos@maths.ox.ac.uk})
\and
Michael Hinterm\"uller\thanks{%
Weierstrass Institute and Humboldt-Universit\"at zu Berlin, Berlin, Germany}
(\email{hintermueller@wias-berlin.de})
}

\headers{Mesh-dependence in active set strategies}{I.~P.~A.~Papadopoulos and M.~Hinterm\"uller}

\title{Mesh-dependent iteration count growth in primal-dual active set strategies}
\begin{document}

\maketitle

\begin{abstract}
Primal-dual active set strategies (PDAS) are popular iterative solvers for mixed complementarity problems such as constrained optimization problems with pointwise inequality constraints. Examples include the reduced-space active set algorithm \texttt{vinewtonrsls} found in PETSc. When applied to discretized infinite-dimensional problems, PDAS exhibit local superlinear convergence thanks to their equivalence to a semismooth Newton method (SSN). However, for many problem classes the number of iterations, to reach convergence, grows without bound under mesh refinement. In this paper we numerically study PDAS iteration counts on uniformly refined meshes for obstacle problems, Signorini problems, and related models. As the mesh size tends to zero, PDAS applied to Signorini-type problems lose their local superlinear convergence, resulting in linear growth of the iteration count (adding some iterations with each refinement).  For obstacle problems, PDAS stagnates, leading to exponential iteration growth (asymptotically doubling with each refinement). We explain these phenomena by (i) proving that, for obstacle problems, nodal degrees of freedom only peel away from the obstacle layer-by-layer during the deactivation phase, (ii) deriving a general global convergence rate for PDAS that depends on the magnitude of dual feasibility violation, and (iii) demonstrating why, in the infinite-dimensional setting, this leads to a well-defined solver, but without local superlinear convergence, for some problems yet divergence for others.
\end{abstract}

\vspace*{-2mm}

\vspace*{-2mm}
\vspace*{-2mm}

\section{Introduction}
\label{sec:introduction}

Optimization problems involving both partial differential operators and pointwise inequality constraints arise in numerous applications such as continuum mechanics, finance, the life sciences, and imaging. For a broad class of such models, discretization gives rise to the following finite-dimensional quadratic programming problem:
\begin{align}
\min_{\bu \in \mathbb{R}^n} \frac{1}{2} \bu^\top A \bu - \vectt{b}^\top \bu \;\; \text{subject to} \;\; \bphi^l_i \leq \bu_i \leq \bphi^u_i, \;\; i \in \{1 \colto n\},
\label{eq:quad-program}
\end{align}
where $n \in \mathbb{N}$, $\{1\colto n\} \coloneqq \{1,\dots,n\}$, $A \in \mathbb{R}^{n \times n}$, $\vectt{b} \in \mathbb{R}^n$, and $\bphi^l, \bphi^u \in \mathbb{R}^n$ denote given lower and upper bounds, respectively. The set of indices corresponding to the active constraints in the solution vector is called the \emph{active set} (cf.~\Cref{sec:pdas}). %

Primal-dual active set strategies (PDAS) are a popular class of algorithms for solving mixed complementarity problems such as \cref{eq:quad-program}. In many cases they can be related to Newton's method for iteratively solving these problems. Their success is unsurprising. They are easy to implement, feature local superlinear convergence, and deliver feasible solutions unlike penalty method counterparts.  This has led to their ubiquitous adoption throughout the optimization community.\footnote{The name ``primal-dual active set" for a solver of the type considered here appears in \cite{bergounioux1999}. However, closely related algorithms appeared well before 1999 (e.g.~\cite{hoppe1987}), and the broader idea of exploiting active set identification to accelerate convergence has a long history.}

The analysis of their favourable convergence properties was primarily conducted in the early 2000's. In particular, a 2003 paper by Hinterm\"uller, Ito, and Kunisch \cite{hintermuller2002} studied a particular instance of a PDAS (hereon referred to as HIK) applied to \cref{eq:quad-program}. Under certain conditions (e.g.~$A$ is symmetric positive-definite), they showed that HIK is equivalent to a \emph{semismooth} Newton method (SSN). Hence it automatically inherits the local superlinear convergence properties that make Newton methods so popular.\footnote{Developing SSNs for MCPs was a core research focus in the 90's and early 2000's \cite{benson2006,chen1996,de1996,dirkse1997,harker1990,munson2001} and extremely efficient implementations can be found in a number of packages e.g.~the PATH solver \cite{dirkse1995} and the reduced space active-set strategy \cite[Alg.~3.2]{benson2006} implemented as \texttt{vinewtonrsls} in PETSc \cite{petsc}.}  In the same paper, the analysis was extended to a number of infinite-dimensional problems including optimal control problems with box constraints on the \emph{control}. Further infinite-dimensional results for SSN were then established in \cite{hintermuller2004mesh}. However, there are several classical infinite-dimensional problems that do not satisfy the necessary conditions for justifying the method in function space. Examples include
\begin{enumerate}
\item obstacle problems, \label{item:problem1}
\item optimal control with box constraints on the state,\label{item:problem2}
\item Signorini problems (i.e.~contact problems).\label{item:problem3}
\end{enumerate} %
For these problems, the HIK algorithm does not define a Newton-type method in infinite dimensions. In fact, for the obstacle problem, all iterates after the second HIK iteration are ill-defined (see \Cref{sec:ssn}). Nevertheless, in practice, problems \labelcref{item:problem1}--\labelcref{item:problem3} are first discretized and then solved. On the resulting finite-dimensional problems, HIK converges with a local superlinear rate owing to its possible interpretation as an SSN. But here lies the surprise. The lack of validity of HIK (and all other PDAS) on the infinite-dimensional level manifests as iteration counts that grow with each refinement of the discretization. This observation is a side-effect  of solver \emph{mesh dependence} where the convergence properties of a solver degrade as the discretization is refined cf.~\cite[Sec.~1]{ito2003}, \cite[Sec.~4]{hintermuller2002}, \cite[Sec.~5.2.2]{bergounioux2002}, \cite[Sec.~7.5]{graser2009}, \cite[Tab.~1,3,4]{hintermuller2003}, \cite[Sec.~4.1]{farrell2020}, \cite[Fig.~3a]{dokken2025}. The aim of this paper is to publicize this phenomenon, collect numerical experiments, and provide theoretical insight. 

The potential for mesh-dependent convergence in PDAS does not appear to be widely known beyond the community working on infinite-dimensional nonsmooth optimization. Although warnings can be found in some texts, a thorough investigation has not been previously circulated. It may surprise even readers familiar with this phenomenon that the iteration count growth rate can vary from mildly linear to exponentially fast depending on the problem at hand.  Since PDAS are often applied to problems where the pointwise constraints are just one of many complications, it may be overlooked that growing iteration counts are due to the PDAS solver rather than other causes.  With each mesh refinement, a single PDAS iteration carries a higher computational cost as the linear systems become larger and potentially ill-conditioned. Thus increasing iteration counts may quickly cause resolution bottlenecks.

This paper focuses on ``vanilla" HIK applied to discretized quadratic programming problems of the form \cref{eq:quad-program}. However, PDAS solvers are often used for nonconvex problems and combined with continuation schemes, linesearches or other globalization schemes, and inexact solves. We hope that understanding the fundamental mesh-dependent properties of the vanilla form  applied to quadratic problems with simple constraints should help identify and remedy mesh-dependent iteration counts in modified PDAS applied to harder problems. We remark that mesh-dependent convergence comes in various flavours and the mesh dependence of PDAS is different from other solvers analysed in the past \cite{schwedes2017mesh}. Some  modifications of PDAS do significantly hinder iteration count growth. We briefly mention two such techniques:
\begin{itemize}
\item \textbf{Multigrid.} Multigrid strategies that allow for the fine-level active set to be altered by coarse-level corrections often can lead to much slower growing iteration counts cf.~\cite{graser2009,brandt1983,hackbusch1983,hoppe1987,hoppe1994,kornhuber1994, kornhuber1996, bueler2024}.
\item \textbf{Regularization.} Regularizing may allow for the PDAS to be valid on the infinite-dimensional level and admit local superlinear convergence cf.~\cite[Ch.~9]{ito2008}, \cite{hintermuller2006feasible,ito2003}.
\end{itemize}

The structure of the paper is as follows.  In \Cref{sec:problem} we introduce the obstacle, thin obstacle, and Signorini problems, respectively, and fix specific examples in \cref{prob:obstacle,prob:thin,prob:signorini}. These examples will serve as the primary target of our analytical and numerical investigations. The HIK algorithm is then described in \Cref{sec:pdas}. Numerical experiments, completed with iteration count tables for \cref{prob:obstacle,prob:thin,prob:signorini}, are provided in \Cref{sec:experiments} and a novel global convergence result for the finite-dimensional problem is derived in \cref{th:convergence-rate} of \Cref{sec:convergence}. Our quest to analyse HIK's mesh-dependent convergence leads to two differing discoveries, one for the obstacle problem and the other for the thin obstacle and Signorini problems:

\noindent \textbf{Obstacle problem.} The numerical experiments reveal that the mesh-dependent convergence manifests as a \emph{layer-by-layer peeling} of the degrees of freedom (dofs) from the obstacle with each iteration during the deactivation phase\footnote{The deactivation phase is where the dofs are switching from the active to the inactive set defined in \Cref{sec:pdas}. It is also sometimes called the inactivation phase \cite[Sec.~5]{graser2009}.} cf.~\cite[Sec.~1]{ito2003}, \cite[Sec.~4.3]{bergounioux2002}. We illustrate this remarkable observation in \cref{fig:peeling} and subsequently analyse it in \cref{th:peeling} of \Cref{sec:peeling}. Peeling leads to exponentially growing iteration counts such that HIK diverges as the mesh size $h \to 0$. By introducing the infinite-dimensional generalization of the HIK algorithm in \Cref{sec:ssn}, we show that the regularity of the Lagrange multiplier iterates is insufficient to define the active and inactive sets after the second iteration causing the divergence to occur. This happens even for obstacle problems where the Lagrange multiplier solution has higher, i.e., $L^2$-regularity, than the usual $H^{-1}$-regularity.

\smallskip
\noindent \textbf{Thin obstacle and Signorini problems.} In stark contrast, peeling does not occur for the thin obstacle and Signorini problems. Nevertheless, the iteration counts exhibit a mild linear growth, but the convergence plots in \Cref{sec:experiments} suggest an initial regime of mesh-independent linear convergence. Superlinear convergence occurs once the correct active set is identified. However, this identification occurs later with each uniform refinement.  We remark why peeling does not occur in \cref{rem:no-peeling} of \Cref{sec:peeling}. Subsequently, in \Cref{sec:ssn} we show that HIK is, in fact, well-defined for the two-dimensional thin obstacle problem in infinite dimensions thanks to the better regularity of the Lagrange multiplier iterates. We then prove that the global convergence rate of \cref{th:convergence-rate} extends to the infinite-dimensional setting resulting in guaranteed sublinear convergence in \cref{th:convergence-rate-inf}. Although HIK still converges, it is no longer equivalent to an SSN in the limit and, hence, the iteration count growth is due to the loss of local superlinear convergence as the mesh size $h \to 0$. 

Our conclusions are given in \Cref{sec:conclusions} where we discuss potential outcomes of the analysis. In the appendix, we provide further numerical experiments on HIK iteration counts. In particular, we study the ineffectiveness of lifting of the multiplier into the discrete primal space, grid sequencing, biactive obstacle problems, and optimal control problems.

We end the introduction by summarizing the various growth rates of different problems explored in this paper in \cref{tab:growth-rates}. We solely consider simple problems with smooth data on convex polygonal domains. Growth rates for harder problems may vary. For instance, grid-sequencing will likely have iteration count growth if the obstacle is not well represented on coarse meshes.
\begin{table}[h!]
\renewcommand{\arraystretch}{1.1}
\centering
\small
\scalebox{1.0}{
\begin{tabular}{|l|l|}
\hhline{-|-|}
\rowcolor{lightgray!10} \multicolumn{1}{|l|}{Problem} &Growth Rate\\ 
\hline
Optimal Control with Control Constraints & No Growth\\ 
Obstacle Problem with Grid-Sequencing & No Growth*\\ \hline
Thin Obstacle Problem (Obstacle of codim 1)  &  Linear (additive)\\
Scalar Signorini & Linear (additive)\\
Signorini (Linear elasticity) & Linear (additive)\\ \hline
Obstacle Problem & Exponential (multiplicative)\\
Optimal Control with State Constraints & Exponential (multiplicative)\\
Optimal Control with Constraints on an $H^1$-Control & Exponential (multiplicative)\\ \hline
\end{tabular}}
\caption{A summary of the observed iteration count growth rates for problems considered in \Cref{sec:experiments} and \Cref{sec:further-experiments}. *Grid-sequencing necessitates solves on all the parent meshes.} 
\label{tab:growth-rates}
\end{table}

\smallskip
\noindent \textbf{Data availability.} Software to generate all the tables and figures in this article may be found at PrimalDualActiveSet.jl \cite{obstacle.jl}. The specific version used in this paper is archived on Zenodo \cite{zenodo}. The implementation is written in the open-source language Julia \cite{Bezanson2017}. Visualizations are made in Plots.jl or ParaView \cite{paraview}.

\section{Problem formulation}
\label{sec:problem}

In this work we consider open, bounded and Lipschitz domains	 $\Omega \subset \mathbb{R}^d$ for $d \in \mathbb{N}$. We denote the usual Lebesgue spaces by $L^p(\Omega)$, $p \in [1,\infty]$, and the Sobolev spaces by $W^{s,p}(\Omega)$, $s \geq0$, $p \in [1,\infty]$ \cite{Adams2003}. We define the Hilbert spaces $H^s(\Omega) \coloneqq W^{s, 2}(\Omega)$ and $H^1_0(\Omega) \coloneqq \{ u \in H^1(\Omega) : u|_{\partial \Omega} = 0\}$ where $|_{\partial \Omega} : W^{1,p}(\Omega) \to W^{1-1/p,p}(\partial \Omega)$ denotes the trace operator. We use the notation $H^s(\Omega)^d$ for a vector-valued Sobolev space with $d$ copies and the inner product $(\cdot,\cdot)$ without a subscript always denotes the $L^2(\Omega)$-inner product. We use $| \cdot |$ to denote the size of a finite-dimensional set, the Lebesgue measure in $\mathbb{R}^d$, or the Euclidean $\ell^2$-norm depending on the context. For Banach spaces $X$ and $Y$, we define $\mathcal{L}(X,Y)$ as the space of bounded linear operators $X \to Y$.

\cref{prob:obstacle,prob:thin,prob:signorini} introduced later in this section will be discretized with the finite element method (FEM). We assume that the domain $\Omega$ is triangulated with a quasi-uniform sequence of meshes $\{\mathcal{T}_h\}_h$ \cite[Def.~4.4.13]{Brenner2008} containing simplicial or tetrahedral cells in two and three dimensions, respectively. Unless otherwise stated, we employ a continuous piecewise linear finite element discretization (P1-FEM) and denote the finite element space, with any Dirichlet boundary conditions already incorporated, by $U_h$. We denote the dual space of $U_h$ by $U_h^*$. When stating estimates, we will often denote a generic constant by $C>0$. The value of $C$ may change from line-to-line, without relabelling, but is always uniformly bounded as the mesh size $h \to 0$.

We now introduce the obstacle and Signorini problem \cite{fichera1964, stampacchia1964}, respectively.

\subsection{Obstacle problem}
\label{sec:problem:obstacle}
The obstacle problem describes the vertical displacement of an elastic membrane constrained to remain below (or potentially above) a specified obstacle.  One seeks a minimizer in $K \coloneqq \{v \in H^1_0(\Omega) : u \leq \varphi \; \text{a.e.~in} \; \Omega \}$, that satisfies
\begin{align}
\min_{u \in K} \int_\Omega \frac{1}{2} |\nabla u|^2 - f u \, \mathrm{d}x,
\label{eq:ob1}
\end{align}
for a given upper bound $\varphi \in H^1(\Omega)$, with $\varphi|_{\partial\Omega}\geq 0$, and forcing term $f \in L^2(\Omega)$. %
The unique minimizer $u \in K$ satisfies the first-order optimality conditions: Find $(u, \lambda) \in H^1_0(\Omega) \times H^{-1}(\Omega)$ that satisfies, for all $v \in H^1_0(\Omega)$ \cite[Sec.~1.4]{ito2008}
\begin{align}
\begin{cases}
(\nabla u, \nabla v) + \langle \lambda, v \rangle = (f,v),\\
u \leq \varphi, \; \langle \lambda, u-\varphi \rangle = 0, \; \text{and} \; \langle \lambda, w \rangle \geq 0 \; \text{for all } 0 \leq w \in H^1_0(\Omega),
\end{cases}
\label{eq:ob3}
\end{align}
where $H^{-1}(\Omega)$ is the dual space of $H^1_0(\Omega)$, and $\langle \cdot, \cdot \rangle$ denotes the duality pairing between $H^{-1}(\Omega)$ and $H^1_0(\Omega)$.  Here, $\lambda$ denotes the Lagrange multiplier associated with the unilateral bound constraint. The bottom three conditions in \cref{eq:ob3} are known as primal feasibility, complementarity, and dual feasibility, respectively.

If $f \in W^{1,\infty}(\Omega)$ and $\varphi \in H^{5/2}(\Omega)$ with $\varphi > 0$ a.e.~in a neighbourhood around $\partial \Omega$, then $u \in H^{s}(\Omega) \cap H^1_0(\Omega)$, for any $s < 5/2$ \cite[Th.~3]{brezis1971} and, therefore, $\lambda \in H^{s-2}(\Omega) \subset L^p(\Omega)$ with $p=\infty$ for $d \in \{1,2\}$ and any $p<\infty$ when $d=3$. This extra regularity is sharp in the sense that $u \not\in H^{5/2}(\Omega)$ for many choices of smooth domain, data, and obstacle cf.~\cref{prob:obstacle}.

Given this extra regularity on $u$ and $\lambda$, we notice that $\lambda \in L^2(\Omega)$ and thus the dual feasibility condition $\langle \lambda, w \rangle \geq0$ for all $0\leq w\in H_0^1(\Omega)$ may be reinterpreted as $\lambda \geq 0$ a.e.~in $\Omega$. In turn this allows the first-order optimality conditions \cref{eq:ob3} to be rewritten as the nonsmooth system, for all $v \in H^1_0(\Omega)$,
\begin{align}
\begin{cases}
(\nabla u, \nabla v) + \langle \lambda, v \rangle = (f,v),\\
\lambda - (\lambda + c(u-\varphi))_+ = 0 \; \text{a.e.~in} \; \Omega,
\end{cases}
\label{eq:ob4}
\end{align}
for any fixed $c>0$, where $(v)_+ \coloneqq \max(0, v)$. %

We now fix two specific obstacle problems that will be the focus of our numerical and analytical investigations in later sections. The first will be simply referred to as the obstacle problem throughout, and the second as the thin obstacle problem.

\begin{problem}[Obstacle problem with a constant obstacle]
\label{prob:obstacle}
For $d\in\{1,2,3\}$, we choose the domain, obstacle, and forcing term as
\begin{align}
\Omega = (0,1)^d, \quad \varphi \equiv 1, \quad f \equiv 20,
\label{eq:setup}
\end{align}
respectively. This constitutes the simplest setup one could envisage and was chosen to alleviate any concerns that mesh dependence occurs due to a complicated choice of domain, obstacle, or forcing term. When $d=1$, the solution is the piecewise quadratic function:
\begin{align}
u(x) =
\begin{cases}
-10x^2 + 20 \alpha x & \text{if} \; x \in [0, \alpha),\\
1 & \text{if} \; x \in [\alpha, 1-\alpha),\\
-10x^2 + 20(1-\alpha)x + 10-20(1-\alpha) & \text{if} \; x \in [1-\alpha, 1],
\end{cases}
\end{align}
where $\alpha = 1/\sqrt{10}$. For reference, we plot the solutions for $d \in \{1,2\}$ in \cref{fig:solns} (left, middle). 
\end{problem}

\smallskip
\begin{problem}[Thin obstacle problem]
\label{prob:thin}
Here we consider an obstacle problem with a \emph{thin} obstacle of codimension one. For $(x_1,\dots,x_d) = x \in \mathbb{R}^d$,
 \begin{align}
 \Omega=(0,1)^d, \quad f \equiv 20, \quad
\text{and}\quad K:=\{v\in H_0^1(\Omega)~:~v|_{ \Gamma_{1/2}}\leq 1\},
\label{eq:setup:scalar}
\end{align}
where $|_{\Gamma_{1/2}}$ denotes the trace operator from $\Omega$ to $\Gamma_{1/2}\coloneqq \{x\in\overline{\Omega}~:~x_1=1/2\}$.
In \cref{fig:solns} (right) we plot the solution for $d =2$.
\end{problem}

\begin{figure}[h!]
\centering
\begin{minipage}{0.32\textwidth}
\includegraphics[width = \textwidth]{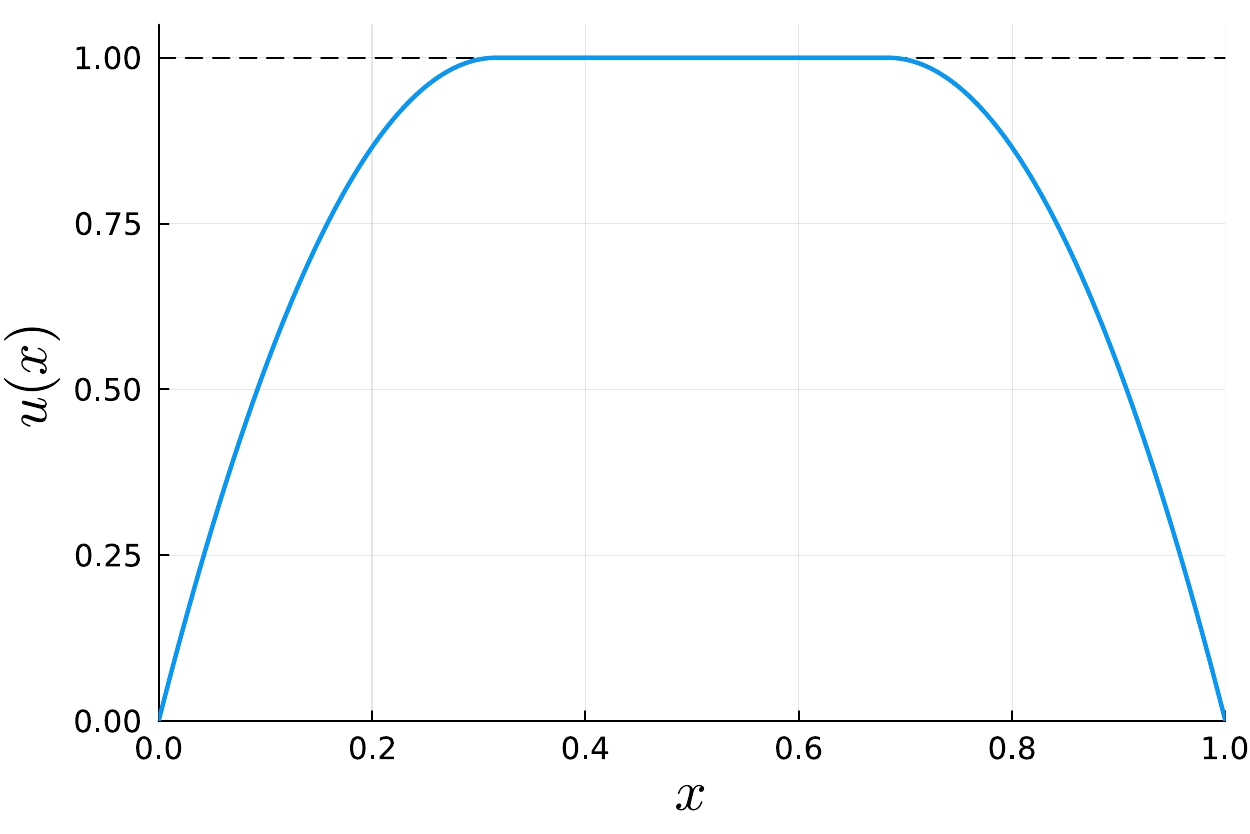}
\end{minipage}\;\;
\begin{minipage}{0.32\textwidth}
\vspace{-4mm}
\hspace*{5mm}
\includegraphics[width =0.75 \textwidth]{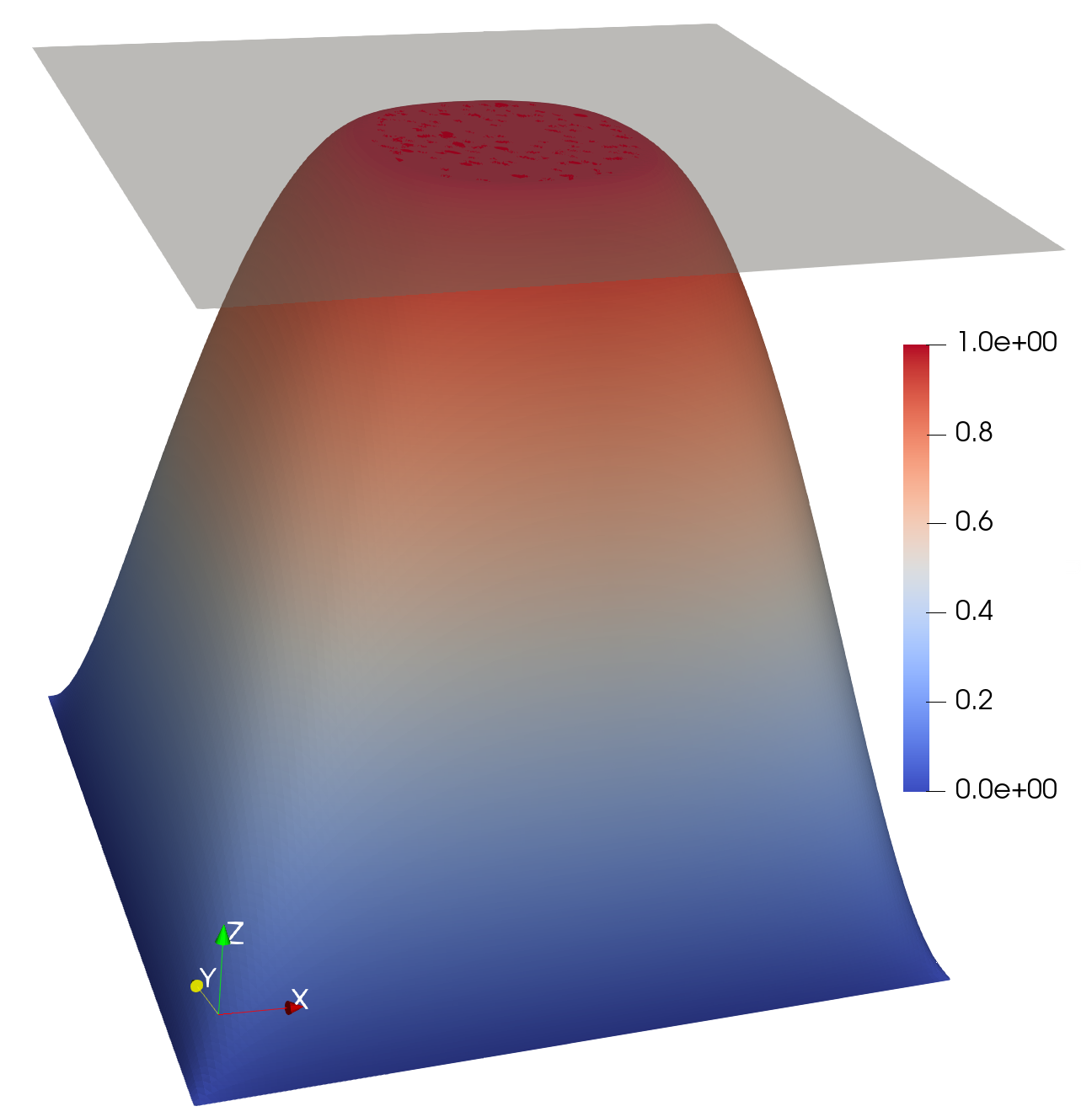}
\end{minipage}
\begin{minipage}{0.32\textwidth}
\vspace{-2mm}
\includegraphics[width =0.68 \textwidth]{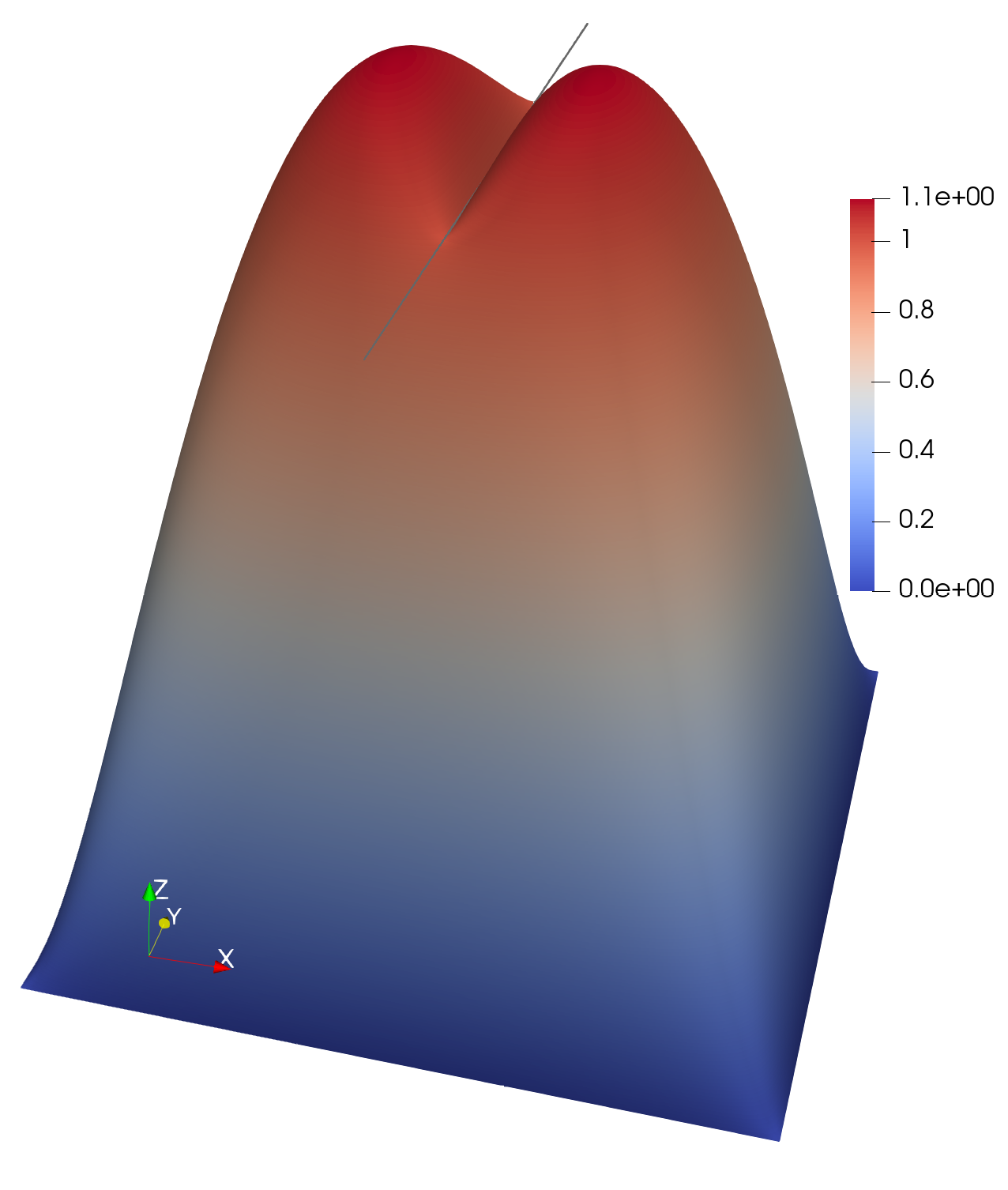}
\end{minipage}
\caption{(Left and middle) The solutions to the obstacle problem of \cref{prob:obstacle} for $d \in \{1,2\}$. The dashed line and grey surface in the 1D and 2D solutions, respectively, represent the obstacle. (Right) The solution of the thin obstacle problem of \cref{prob:thin}. The grey line at $x_1=1/2$ represents the thin (codimension one) obstacle.}
\label{fig:solns}
\end{figure}
Note that while $\Gamma_{1/2}$ has zero ($d$-dimensional) Lebesgue measure, it has positive capacity such that $K$ manifests a genuine constraint in $H_0^1(\Omega)$. For more on capacity theory see, e.g., \cite{kinderlehrer2000}.

\subsection{Signorini problem}
The Signorini problem models the deformation of a linearly elastic material in contact with a rigid body. Pointwise constraints are only enforced on a contact boundary. As such, the feasible set is $K \coloneqq \{ v \in H^1_{\Gamma_D}(\Omega)^d : v|_{\Gamma_C} \cdot n \leq \varphi \; \text{on} \; \Gamma_C\}$ where $d \in \{2,3\}$, $\Gamma_D \subset \partial \Omega$ is nonempty, $H^1_{\Gamma_D}(\Omega)^d \coloneqq \{ v \in H^1(\Omega)^d : v|_{\Gamma_D} = 0 \;
\}$, $\varphi \in H^{1/2}(\partial \Omega)$ is a prescribed gap function that defines how far a point on the contact boundary $\Gamma_C \subset \partial \Omega$ may deflect, and $n$ is the outward unit normal to $\partial\Omega$. We shall assume that $\bar \Gamma_C$ is compactly embedded in $\partial \Omega \backslash \bar \Gamma_D$ and that $\Gamma_F:=\partial \Omega\setminus (\bar{\Gamma}_D\cup\bar{\Gamma}_C)$ is the {\it free} boundary yielding natural boundary conditions on $u$. Solving the Signorini problem requires to find $u \in K$ satisfying
\begin{align}
    \min_{u \in K} \int_\Omega  \mu |\varepsilon(u)|^2 + \frac{\ell}{2} |\divv(u)|^2 - f \cdot u \, \dx,
    \label{eq:signorini}
\end{align}
where $\varepsilon(u) = (\nabla u + (\nabla u)^\top)/2$ is the symmetrized gradient, $\mu$ and $\ell$ are the Lam\'e parameters, and $f \in L^2(\Omega)^d$ is a given body force. Let $M_+ \coloneqq \{ \lambda \in H^{-1/2}(\Gamma_C) : \langle \lambda, w \rangle_{\Gamma_C} \geq 0 \; \text{for all} \; 0\leq w \in H^{1/2}(\Gamma_C)\}$, where $\langle \cdot, \cdot \rangle_{\Gamma_C}$ denotes the duality pairing of $H^{-1/2}(\Gamma_C)$ and $H^{1/2}(\Gamma_C)$. Then there exists a unique Lagrange multiplier $\lambda \in M_+$ such that the minimizer $u$ also satisfies, for all $v \in H^1_{\Gamma_D}(\Omega)^d$ and $0 \leq  w \in H^{1/2}(\Gamma_C)$,
\begin{align}
\begin{cases}
2 \mu (\varepsilon(u), \varepsilon(v)) + \ell (\divv(u), \divv(v)) + \langle \lambda, v|_{\Gamma_C} \cdot n \rangle_{\Gamma_C} = (f, v),\\
u|_{\Gamma_C} \cdot n \leq \varphi, \; \langle \lambda, u|_{\Gamma_C} \cdot n -\varphi \rangle_{\Gamma_C} = 0, \; \text{and} \; \langle \lambda, w \rangle_{\Gamma_C} \geq 0,
\end{cases}
\label{eq:signorini-mcp}
\end{align}
with $\varepsilon\cdot n=0$ on $\Gamma_F$ (assuming zero traction) and $\varepsilon\cdot n=-\lambda$ on $\Gamma_C$ in the sense of the trace.

Under suitable smoothness of the domain and problem data (that we omit here), we have that $\lambda \in L^2(\Gamma_C)$ and, therefore, \cref{eq:signorini-mcp} may be reformulated as the nonsmooth equation, for all $v \in H^1_{\Gamma_D}(\Omega)^d$ and any $c>0$:
\begin{align}
\begin{cases}
2 \mu (\varepsilon(u), \varepsilon(v)) + \ell (\divv(u), \divv(v)) + \langle \lambda, v \cdot n \rangle_{\Gamma_C} = (f, v),\\
\lambda - (\lambda + c(u|_{\Gamma_C} \cdot n-\varphi))_+ = 0 \; \text{a.e.~on} \; \Gamma_C.
\end{cases}
\label{eq:signorini-nonsmooth}
\end{align}

We now fix a specific instance of a Signorini problem that will be the focus of our numerical investigations in \Cref{sec:experiments}.

\begin{problem}[Signorini]
\label{prob:signorini}
For $d \in \{2,3\}$ and $(x_1,\dots,x_d) = x \in \mathbb{R}^d$, we choose the parameters
\begin{align}
\begin{gathered}
\Omega = (0,5) \times (0,1)^{d-1}, 
\;\; \Gamma_D = \{x\in\bar\Omega:x_1 = 0 \; \text{or} \; x_1=5\}, \\
\Gamma_C = \{x\in\bar\Omega: x\in (0,5)\times (0,1)^{d-2}\times \{1\}\}, \;\;
\quad n=e_d,
\quad \varphi \equiv 1/2,\\  
E=200, 
\quad \nu = 0.3,
\quad
 \ell = \frac{E \nu}{(1+\nu)(1-2\nu)},
\quad \mu = \frac{E}{2(1+\nu)},
\quad f =10e_d,
\end{gathered}
\label{eq:setup:signorini}
\end{align}
where $e_d \in \mathbb{R}^d$ is the unit vector $(0,\dots,0,1)^\top$. This models a two- and three-dimensional linearly elastic beam clamped on its left- and right-hand side edges in 2D or faces in 3D. A body force deforms the beam upwards but the beam cannot penetrate through an obstacle positioned at the plane $x_d = 1/2$.  We plot both the two- and three-dimensional solutions in \cref{fig:signorini}.
\end{problem}
\begin{figure}[h!]
\centering
\includegraphics[height =0.2 \textwidth]{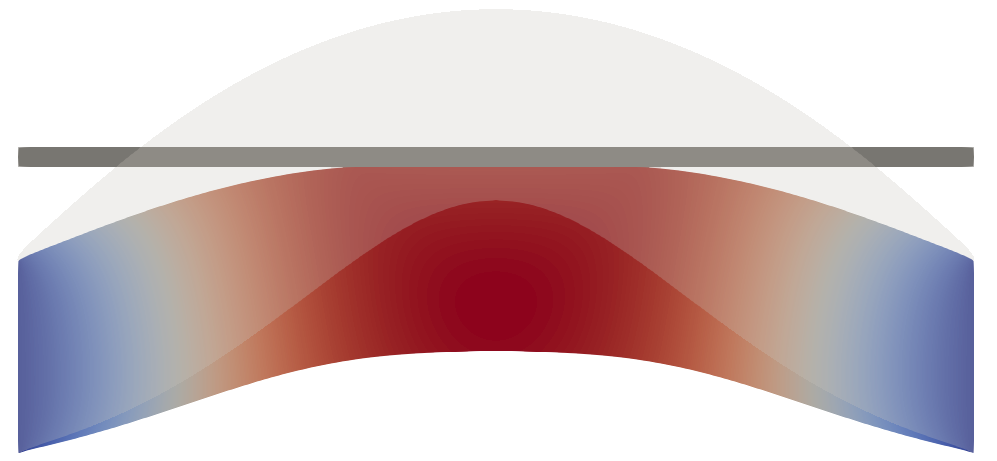}
\includegraphics[height =0.18 \textwidth]{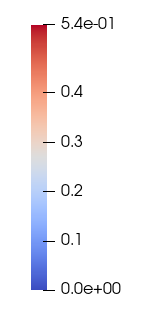}
\includegraphics[height =0.2 \textwidth]{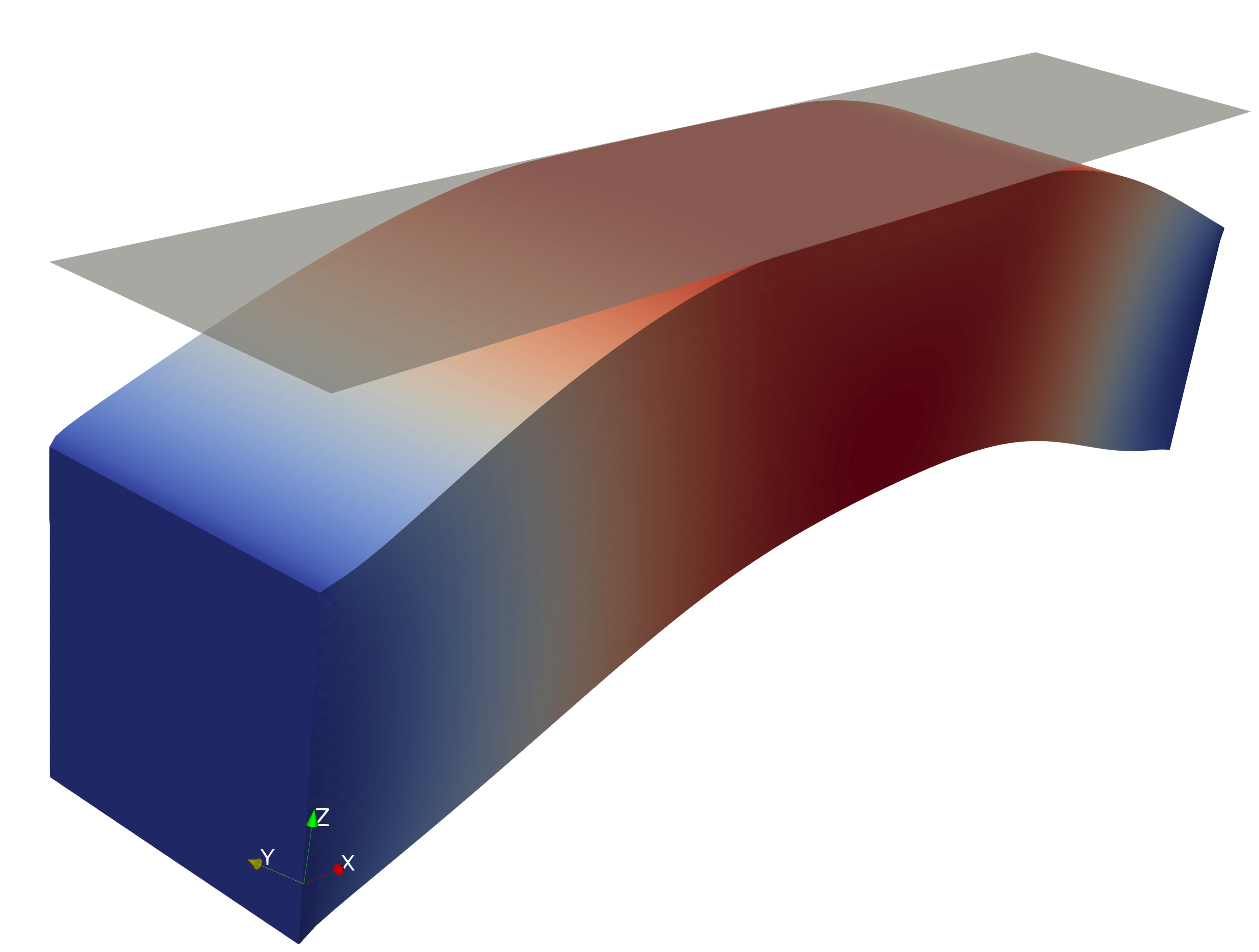}
\caption{The solutions to the Signorini problem of \cref{prob:signorini} in two dimensions (left) and three dimensions (right). The color denotes the magnitude of the displacement $u$ as indicated on the colorbar and the thin horizontal beams represent the obstacle. In the 2D solution the transparent beam is the solution without the pointwise constraint.}
\label{fig:signorini}
\end{figure}

\section{Primal-dual active-set strategies}
\label{sec:pdas}

In this section we introduce the PDAS described in \cite{hintermuller2002}, called HIK.\footnote{Other PDAS and SSNs often result in equivalent iterates. For instance, in exact arithmetic, the reduced active set strategy \cite[Alg.~3.2]{benson2006} is almost equivalent to HIK \cite[App.~A]{papadopoulos2021}.} %

A key aspect of the low-order P1-FEM nodal basis discretization considered in this paper is that entrywise box constraints on the dofs imply pointwise box constraints on the finite element function itself. In other words: If $u_h$ denotes the finite element discretization of $u$ and $\vectt{u}_h \in \mathbb{R}^N$ denotes the corresponding discrete coefficient vector, then $[\vectt{u}_h]_i \leq 1$ for $i \in \{1 \colto N\}$ implies $u_h(x) \leq 1$ for all $x \in \Omega$. Making a similar simplification of box constraints for higher-order finite element spaces is possible cf.~\cite{banz2015, kirby2024}. For the remainder of this section we fix a mesh and thus drop the subscript $_h$ in the coefficient vector for brevity. 

Below we focus on deriving HIK for the discretized obstacle problem. The extension to the Signorini problem is considered afterwards. The first step to implement a PDAS is to discretize the nonlinear system \cref{eq:ob4}:
\begin{subequations}
\label{pdas:1}
\begin{align}
A\vectt{u} + B \vectt{\lambda} &= \vectt{b}, \label{pdas:1a}\\
\vectt{\lambda} - (\vectt{\lambda} + c(\vectt{u} - \vectt{\varphi}))_+ & = \vectt{0}, \label{pdas:1b}
\end{align}
\end{subequations}
Here $\vectt{b} \coloneqq M \vectt{f}$ is the load vector and $\vectt{f}$ and $\vectt{\varphi} \in \mathbb{R}^N$ are the discrete coefficient vectors for $f_h$ and $\varphi_h$, respectively, and $c>0$. The max function of a vector is understood component-wise, i.e.~$[(\vectt{v})_+]_i = \max(0,\vectt{v}_i)$. The matrix $A$ is the stiffness matrix $A_{ij} = (\nabla \phi_i, \nabla \phi_j)$ and $M$ is the mass matrix $M_{ij} = (\phi_i,\phi_j)$ for the P1-FEM basis functions $\phi_i \in U_h$, $i \in \{1 \colto N\}$. The matrix $B$ is determined by the discretization of the Lagrange multiplier $\lambda$. To implement a PDAS, $B$ must be diagonal in order to perform a static condensation of the arising linear systems. Two popular choices are:
\begin{enumerate}
\item Discretize $\lambda$ with the dual P1-FEM basis $\{\psi_i\}_{i=1}^N$ such that $\psi_i(\phi_j) = \delta_{ij}$.\footnote{It is also common to choose the normalization $\psi_i(\phi_j) = \delta_{ij} \int_\Omega \phi_j \, \dx$.} In a sense, $\vectt{\lambda}_i$ represents the coefficient for a Dirac delta concentrated at the node $x_i \in \mathbb{R}^d$. This results in $B=I$. Moreover, \cref{pdas:1b} is still valid on the finite-dimensional level since $\langle \vectt{\lambda}_i \psi_i, w_h \rangle \geq 0$ for any $0 \leq w_h \in V_h$, if and only if $\vectt{\lambda}_i \geq 0$. In fact, this is the implicit choice for a discretize-then-optimize approach, i.e.~discretize \cref{eq:ob1} and derive the KKT conditions of the constrained finite-dimensional optimization problem \cite[Eq.~(2.3)]{hintermuller2002}.
\item Discretize $\lambda$ with the primal P1-FEM basis $\{\phi_i\}_{i=1}^N$. Then $\vectt{\lambda}$ represents the coefficient vector for a continuous P1-FEM function $\lambda_h$. In this case $B=M$. However, by choosing an inexact quadrature scheme and picking the quadrature points at the dofs, a mass-lumping occurs and $B$ becomes diagonal.
\end{enumerate}
For the remainder of this paper, we opt for the first choice of discretization for $\lambda$ such that $B=I$.

Our goal is to solve the discretized nonlinear system \cref{pdas:1}, with $B=I$. HIK is an iterative method that leverages \cref{pdas:1b} as a predictive strategy. We first introduce notation that extracts submatrices and vectors.

\begin{definition}[Slice notation]
\label{def:slice}
Consider the two sets $\mathcal{U}, \mathcal{Y} \subseteq \{1 \colto N\}$. Given a matrix $B \in \mathbb{R}^{N \times N}$ and a vector $\vectt{x} \in \mathbb{R}^N$, then
\begin{align}
B_{\mathcal{U}, \mathcal{Y}} \in \mathbb{R}^{|\mathcal{U}|\times|\mathcal{Y}|} \;\; \text{and} \;\; \vectt{x}_{\mathcal{U}} \in \mathbb{R}^{|\mathcal{U}|}
\end{align}
where $B_{\mathcal{U}, \mathcal{Y}}$ denotes the submatrix arising after deleting all the rows of $B$ with the indices $\{1 \colto N\} \backslash \mathcal{U}$ and all the columns with the indices $\{1 \colto N\} \backslash \mathcal{Y}$. Similarly $\vectt{x}_{\mathcal{U}}$ denotes the subvector arising after deleting all the rows of $\vectt{x}$ with the indices $\{1 \colto N\} \backslash \mathcal{U}$.
\end{definition}

At iteration $k$, given the current iterate $(\vectt{u}^k,\vectt{\lambda}^k)$ let $\mathcal{A}_h^k$ and $\mathcal{I}_h^k$ denote the active and inactive set predictions, respectively, defined by
\begin{align}
\mathcal{A}_h^k \coloneqq \{i:\vectt{\lambda}^k_i - c(\vectt{\varphi}_i - \vectt{u}^k_i) > 0\} \quad \text{and} \quad
\mathcal{I}_h^k \coloneqq\{i: \vectt{\lambda}^k_i - c(\vectt{\varphi}_i - \vectt{u}^k_i) \leq 0\}. \label{eq:inactive-active}
\end{align}
Then, the HIK iteration seeks $(\vectt{u}^{k+1}, \vectt{\lambda}^{k+1})$ satisfying
\begin{align}
\label{eq:pdas-update}
A \vectt{u}^{k+1} + \vectt{\lambda}^{k+1} = \vectt{b}, \quad
\vectt{u}^{k+1}_{\mathcal{A}_h^k} = \vectt{\varphi}_{\mathcal{A}_h^k}, \quad \text{and} \quad
\vectt{\lambda}^{k+1}_{\mathcal{I}_h^k} = \vectt{0}.
\end{align}

\subsection{HIK as an SSN}
\label{sec:pdas:semi}

The operator $(\cdot)_+ : \mathbb{R}^N \to \mathbb{R}^N$, $N \in \mathbb{N}$ in \cref{pdas:1b} is not Fr\'echet differentiable but does possess a so-called Newton derivative \cite[Def.~2.1]{brokate2022} and can be shown to be semismooth on $\mathbb{R}^N$ \cite[Lem.~3.1]{hintermuller2002}. One choice of a Newton derivative $G_+$ for $(\vectt{v})_+$ is
\begin{align}
[G_+(\vectt{v}) \vectt{w}]_i = 
\begin{cases}
\vectt{w}_i, & \text{if } \vectt{v}_i > 0,\\
0, & \text{otherwise}.
\end{cases}
\end{align}
Leveraging this Newton derivative for $(\cdot)_+$, an SSN for \cref{pdas:1} seeks the update $(\vectt{\delta}^k,\vectt{\varepsilon}^k)$ satisfying the linearized system 
\begin{align}
\begin{pmatrix}
A & I_N \\
-cE^{\mathcal{A}_h^k} & E^{\mathcal{I}_h^k}
\end{pmatrix}
\begin{pmatrix}
\vectt{\delta}^k\\
\vectt{\varepsilon}^k
\end{pmatrix}
=
-
\begin{pmatrix}
A \vectt{u}^k + \vectt{\lambda}^k - \vectt{b}\\
\vectt{g}^k
\end{pmatrix},
\label{eq:HIKstep}
\end{align}
where $I_{N}$ is the $N \times N$ identity matrix
and $E^{\mathcal{K}} \in \mathbb{R}^{N \times N}$ denotes the diagonal matrix with a one for diagonal indices $i \in \mathcal{K} \subseteq \{1 \colto N\}$ and zero otherwise. Moreover,
\begin{align}
\vectt{g}^k_i &= 
\begin{cases}
-c(\vectt{u}^k_i  - \vectt{\varphi}_i)  & \; \text{if} \; i \in \mathcal{A}_h^k,\\
 \vectt{\lambda}^k_i & \; \text{if} \; i \in \mathcal{I}_h^k.
\end{cases}
\end{align}
The subsequent SSN iterates are $\vectt{u}^{k+1} = \vectt{u}^{k} + \vectt{\delta}^k$ and $\vectt{\lambda}^{k+1} = \vectt{\lambda}^{k} + \vectt{\varepsilon}^k$. Via static condensation, the linear system solve in \cref{eq:HIKstep} may be reduced to \cite[Sec.~2]{hintermuller2002}
\begin{align}
A_{\mathcal{I}_h^k,\mathcal{I}_h^k} \vectt{\delta}^k_{\mathcal{I}_h^k}
=- A_{\mathcal{I}_h^k, \mathcal{A}_h^k} (\vectt{\varphi} - \vectt{u}^k)_{\mathcal{A}_h^k} - (A\vectt{u}^k - \vectt{b})_{\mathcal{I}_h^k}, \label{eq:reducedHIK}
\end{align}
condensing the linear system solve of size $2N$ to one of size $|\mathcal{I}_h^k| \leq N$  with the remaining unknowns set to
\begin{align}
\vectt{u}^{k+1}_{\mathcal{A}_h^k} = \vectt{\varphi}_{\mathcal{A}_h^k}, \quad \vectt{\lambda}^{k+1}_{\mathcal{A}_h^k} = -(A \vectt{u}^{k+1} -  \vectt{b})_{\mathcal{A}_h^k} \quad \text{and} \quad  \vectt{\lambda}^{k+1}_{\mathcal{I}_h^k} = \vectt{0}.
\label{eq:static}
\end{align}
A calculation reveals that $(\vectt{u}^{k+1}, \vectt{\lambda}^{k+1})$ satisfies the HIK iteration in \cref{eq:pdas-update} \cite[Sec.~2]{hintermuller2002}. Hence, the HIK and SSN applied to \cref{pdas:1} result in identical iterates. Moreover, an SSN for a semismooth equation possessing a uniformly bounded and invertible Newton derivative in a neighborhood of a solution is sufficient to extract local superlinear convergence properties cf.~\cite{hintermuller2002}, \cite[Th.~6.5]{ulbrich2002semismooth}. We observe this phenomenon in the numerical experiments in the next section.

We now summarize the HIK algorithm in \cref{alg:hik}.
\begin{algorithm}[h!]
\caption{HIK solver for \cref{pdas:1}}
\label{alg:hik}
\begin{algorithmic}[1]
\State{\textbf{input:} $A$, $\vectt{b} \coloneqq M \vectt{f}$, $\vectt{\varphi}$, $\vectt{u}^0$, $\mathrm{tol}$.}
\State{$\vectt{\lambda}^0 = \vectt{0}$, $\mathcal{A}^0_h = \emptyset$ and $\mathcal{I}^0_h = \{1 \colto N\}$.}
\State{Initialize $k=0$.}
\State{\textbf{repeat}}
\State{\quad Solve $A_{\mathcal{I}^k_h,\mathcal{I}^k_h} \vectt{\delta}^k_{\mathcal{I}^k_h}
=- A_{\mathcal{I}^k_h, \mathcal{A}^k_h} (\vectt{\varphi} - \vectt{u}^k)_{\mathcal{A}^k_h} - (A\vectt{u}^k - \vectt{b})_{\mathcal{I}^k_h}$.}
\State{\quad Assign $\vectt{u}^{k+1}_{\mathcal{I}^k_h} = \vectt{u}^k_{\mathcal{I}^k_h} + \vectt{\delta}^k_{\mathcal{I}^k_h}$ and $\vectt{u}^{k+1}_{\mathcal{A}^k_h} = \vectt{\varphi}_{\mathcal{A}^k_h}$.}
\State{\quad Assign $\vectt{\lambda}^{k+1}_{\mathcal{A}^k_h} = -(A \vectt{u}^{k+1} - \vectt{b})_{\mathcal{A}^k_h}$ and $\vectt{\lambda}^{k+1}_{\mathcal{I}^k_h} = \vectt{0}$.}
\State{\quad Assign $\mathcal{A}_h^{k+1} = \{i : \vectt{\lambda}^{k+1}_i + \vectt{u}^{k+1}_i - \vectt{\varphi}_i > 0\}$ and $\mathcal{I}^{k+1}_h = \{1 \colto N\} \backslash \mathcal{A}_h^{k+1}$.}
\State{\quad Assign $k \leftarrow k+1$.}
\State{\textbf{until} $\| \vectt{\lambda}^k - (\vectt{\lambda}^k + \vectt{u}^k - \vectt{\varphi})_+\|_{\ell^2} + \| A \vectt{u}^k + \vectt{\lambda}^k - \vectt{b}\|_{\ell^2} < \mathrm{tol}$}.
\end{algorithmic}
\end{algorithm}
Let us note that we chose $c=1$ in \cref{alg:hik}. Due to the steps 6 and 7 of the algorithm, this choice has only an effect on the first iteration, i.e., $k=0$.
Indeed, for $k\geq 1$ one has for $i\in\mathcal{I}^k_h$ that $\vectt{\lambda}_i^{k+1}=0$. Hence, only the sign of $\vectt{u}_i^{k+1}-\vectt{\varphi}_i$ decides on whether $i\in\mathcal{A}^{k+1}_h$ or not. Similarly, for $i\in\mathcal{A}^k_h$ we have $\vectt{u}_i^{k+1}=\vectt{\varphi}_i$ such that only the sign of $\vectt{\lambda}_i^{k+1}$ decides on whether $i\in\mathcal{I}^{k+1}_h$ or not.

\subsection{HIK for Signorini}
To extend HIK to the solution of the Signorini problem, we now rewrite the corresponding KKT conditions. Here we solely consider a Signorini problem where the normal at the contact boundary is $n \in \{e_1, \dots, e_d \}$, with a constant gap function $\varphi$, where $e_j$ is the unit vector with a one in the $j^\text{th}$-entry.\footnote{For more complex domains and gap functions, the normal on the deformed boundary depends on $u$ and the physics become nonlinear \cite[Sec.~7.1]{alphonse2025}.} The nonsmooth reformulation of the discretized KKT conditions of the Signorini problem is
\begin{subequations}
\label{pdas:3}
\begin{align}
A\vectt{u} + E_0\tilde{\vectt{\lambda}} &= \vectt{b}, \label{pdas:3a}\\
E_0 \tilde{\vectt{\lambda}} - (E_0\tilde{\vectt{\lambda}} + c(\vectt{u} - E_\infty \tilde{\vectt{\varphi}}))_+ & = \vectt{0}, \label{pdas:3b}
\end{align}
\end{subequations}
where $\vectt{u} \in \mathbb{R}^N$, and $\tilde{\vectt{\lambda}}, \tilde{\vectt{\phi}} \in \mathbb{R}^{N_C}$ such that $N_C<N$ is the number of dofs along the contact boundary. The matrix $A$ now denotes the stiffness matrix corresponding to the linear elasticity equation. Further, $E_0, E_\infty \in \mathbb{R}^{N \times N_C}$ denote the extension-by-zero and extension-by-infinity matrices, respectively. After choosing $\vectt{\lambda} = E_0 \tilde{\vectt{\lambda}}$ and $\vectt{\varphi} = E_\infty \tilde{\vectt{\varphi}}$ we recover the nonsmooth system \cref{pdas:1}. We could instead restrict $\vectt{u}$ to the boundary dofs and pose \cref{pdas:3b} over the boundary dofs only (ultimately reducing the size of the nonlinear system). However, after a static condensation, the linear systems to be solved are of the same size. %

\section{A numerical investigation}
\label{sec:experiments}
In this section we investigate the growth of the HIK iteration counts for the obstacle, thin obstacle, and Signorini problems, introduced in \cref{prob:obstacle,prob:thin,prob:signorini}, as we uniformly refine the mesh.

To summarize, we shall observe that the iteration counts grow exponentially fast for the obstacle problem, and mildly linearly for the thin obstacle and Signorini problems, respectively. The difference in behaviour is due to the codimensionality of the obstacle. %

\smallskip
\noindent \textbf{Experiment.} In \cref{tab:mesh-dependence} we report the number of HIK iterations (\cref{alg:hik}) to solve \cref{prob:obstacle,prob:thin,prob:signorini} and provide convergence plots in \cref{fig:convergence}. We use $n$ to parametrize the number of cells along an axis. On every mesh, we initialize HIK with a zero initial guess for $u$. The linear systems arising in \cref{alg:hik} are solved via a Cholesky factorization.

\begin{table}[h!]
\renewcommand{\arraystretch}{1.1}
\centering
\small
\begin{tabular}{|l|c|c|c|c|c|c|c|}
\hhline{-|-|-|-|-|-|-|-|}
\rowcolor{lightgray!10} \multicolumn{1}{|c|}{$n$} & $2^{4}$ & $2^{5}$ & $2^{6}$  &$2^{7}$ &$2^{8}$ &$2^{9}$ &$2^{10}$\\
\hline
Obstacle 1D ($n$)  & 5 & 8 & 14 &  27& 54 & 106 & 210 \\
Obstacle 2D  ($n \times n$)& 5 & 7 & 12 &  22& 43 &  83& 165 \\
Obstacle 3D  ($n \times n \times n$)& 5 & 7 & 10 &  15 & 27 &- & - \\
Thin Obstacle 2D ($n \times n$) & 3  &5  &  5&6 &8 & 8&9  \\
Thin Obstacle 3D ($n \times n \times n$) & 3  &4  & 6 & 7 & 8 & -&-  \\
\hline
\hline
\rowcolor{lightgray!10} \multicolumn{1}{|c|}{$n$} & $5$ & $10$ & $20$  &$40$ &$80$ &$160$ & $320$\\
\hline
Signorini 2D ($5n \times n$) &6  &7  &8  & 9 & 10 &12 & 12 \\
Signorini 3D ($5n \times n \times n$) & 7 &8  &9  & 11 &  12 &- & - \\ %
\hline
\end{tabular}
\caption{Mesh-dependent HIK iteration growth, when applied to the obstacle problem in \cref{prob:obstacle} with $d \in \{1,2,3\}$, the thin obstacle problem in \cref{prob:thin} with $d \in \{2,3\}$, and the Signorini problem in \cref{prob:signorini} with $d \in \{2,3\}$.} %
\label{tab:mesh-dependence}
\end{table}

\begin{figure}[h!]
\centering
\includegraphics[width =0.32 \textwidth]{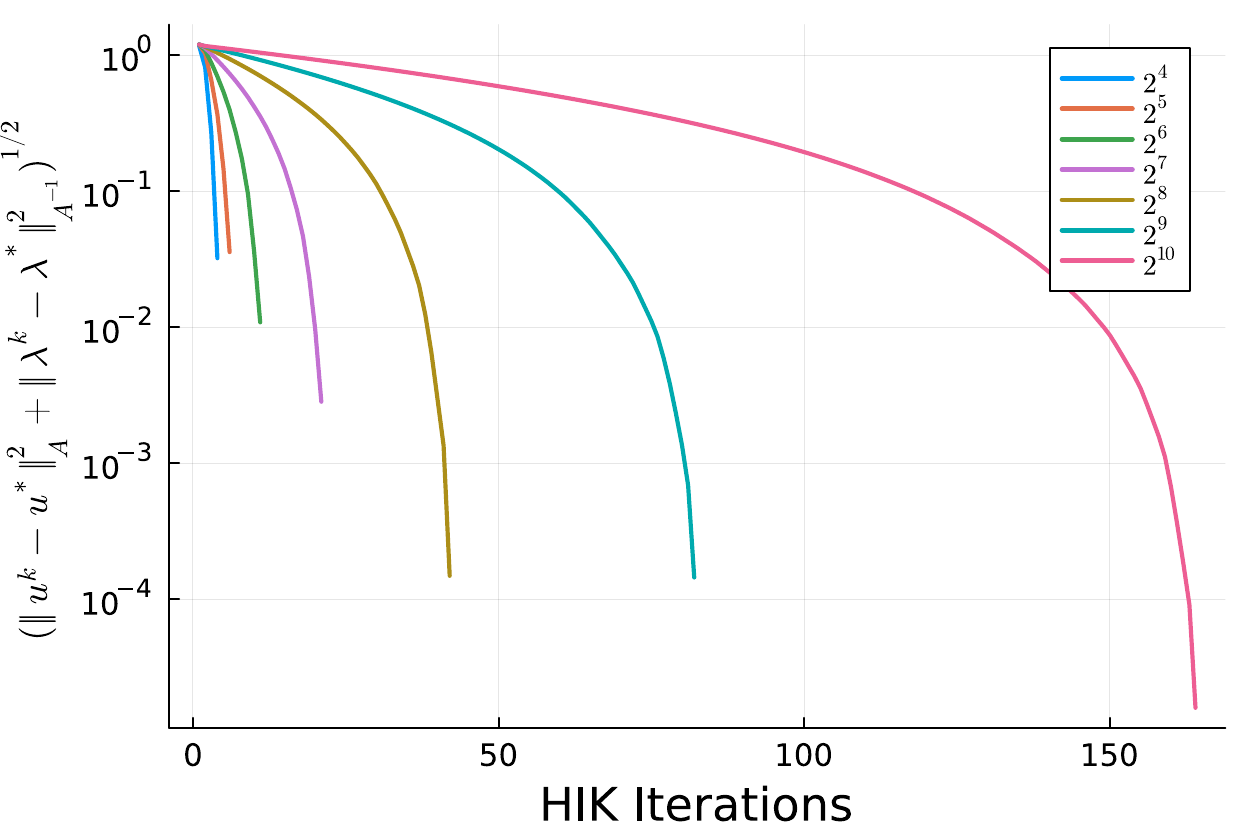}
\includegraphics[width =0.32 \textwidth]{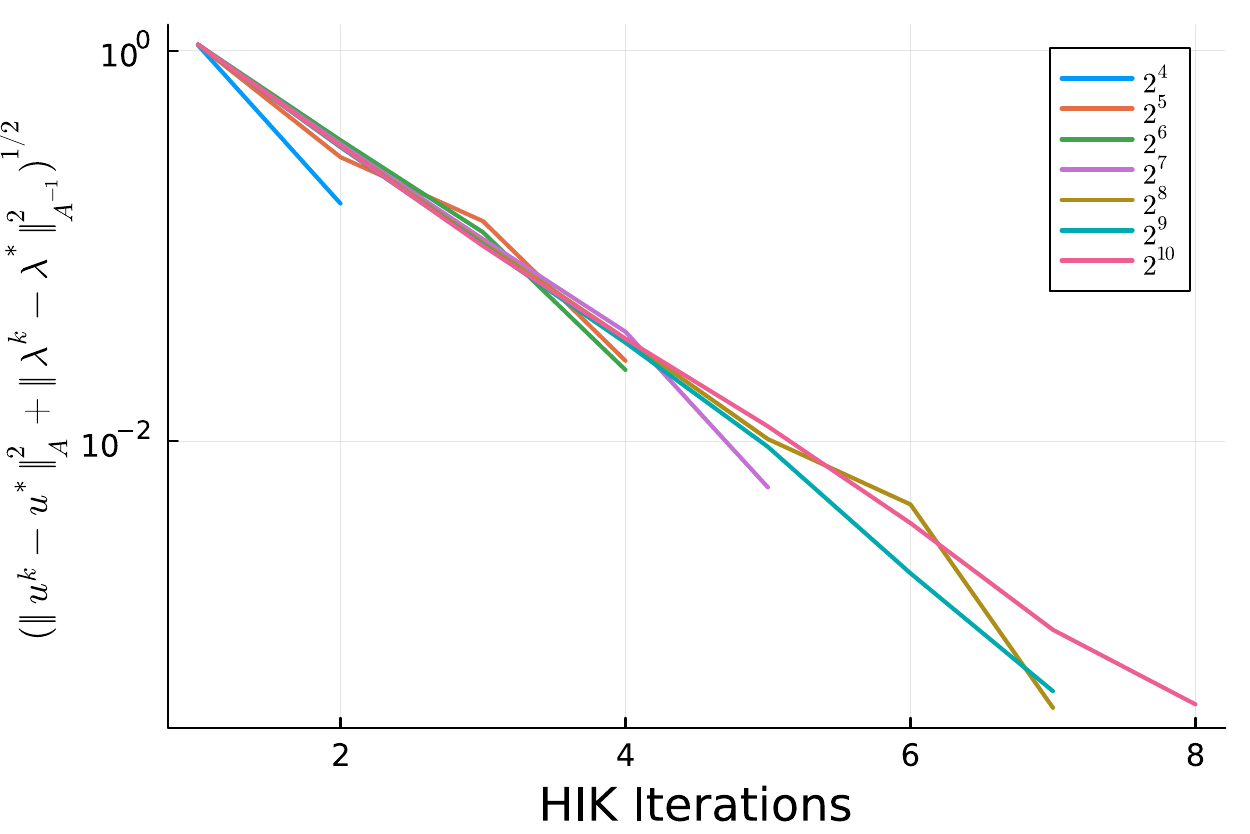}
\includegraphics[width =0.32 \textwidth]{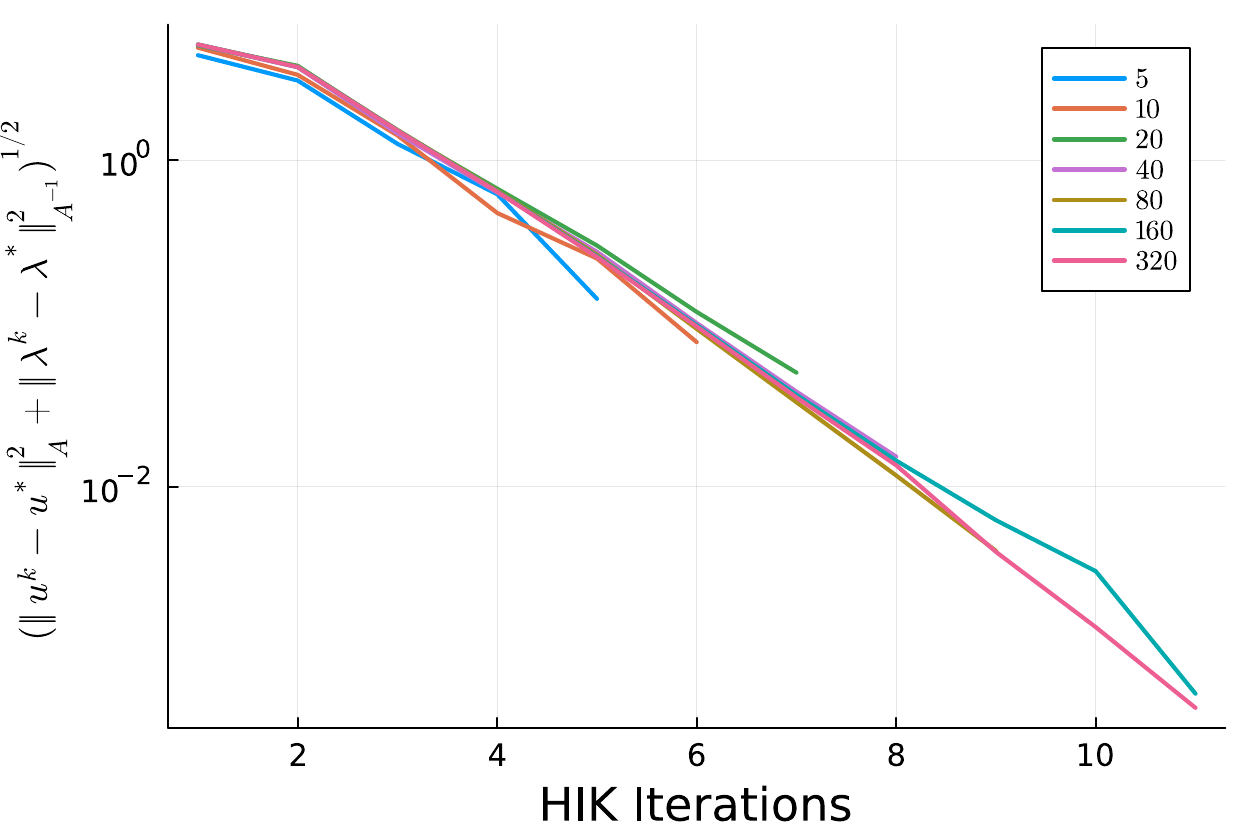}
\caption{Convergence plots of the HIK iterates, with increasing mesh resolution $n$ as labelled in the legend, for \cref{prob:obstacle,prob:thin,prob:signorini} from left to right, respectively, when $d=2$. \cref{prob:obstacle} features superlinear convergence of the iterates. However, the convergence is mesh dependent and for each uniform refinement we observe that the convergence rate halves. For \cref{prob:thin,prob:signorini}, the convergence exhibits an initial regime of mesh-independent linear convergence. Once the correct active set is identified, the subsequent error drops to zero. This correct identification occurs later with each uniform refinement.}
\label{fig:convergence}
\end{figure}

\smallskip
\noindent \textbf{\cref{prob:obstacle} results.}  The iteration counts for obstacle problem solves experience exponential growth. For $n \geq 2^7$, the number of iterations almost doubles with each uniform refinement in both one and two dimensions. The convergence plot in \cref{fig:convergence} reveals a superlinear convergence of the iterates that deteriorates as the mesh resolution $n$ increases. If $(\vectt{u}^*_h, \vectt{\lambda}^*_h)$ denotes the solution on a mesh with mesh size $h$, then we observe that after a uniform refinement (denoted by a subscript of $h/2$)
\begin{align}
\| \vectt{u}^{2k}_{h/2} - \vectt{u}^{*}_{h/2} \|^2_A + \| \vectt{\lambda}^{2k}_{h/2} - \vectt{\lambda}^{*}_{h/2}
 \|^2_{A^{-1}} = \| \vectt{u}^{k}_h - \vectt{u}^{*}_h \|^2_A + \| \vectt{\lambda}^{k}_h - \vectt{\lambda}^{*}_h \|^2_{A^{-1}}.
\end{align}
A remarkable observation can be made in \cref{fig:peeling}: Only the dofs on the boundary of active set are switching to the inactive set from one HIK iteration to the next. Qualitatively, after the first HIK iteration, the solver has ``overshot" and a sizeable number of dofs are glued to the obstacle, i.e., are active concerning the bound constraint. However, the true discretized solution has a smaller active set, and each subsequent HIK iteration removes dofs from the active set. This is visually observed as the HIK iterates \emph{peel} away from the obstacle.

We call this the layer-by-layer peeling effect and it is a useful heuristic for determining the growth in the number of HIK iterations. After a uniform mesh refinement, HIK must peel away twice as many dofs from the active set. This correction requires roughly twice as many iterations when compared to the scheme on the next coarser mesh. The peeling effect, however, gets mitigated by grid-sequencing or full approximation multigrid schemes. Indeed, in those regimes, first the problem is solved on a coarse mesh where fewer dofs are incorrectly assigned to the active set and thus requiring a relatively small number of HIK iterations. This coarse scale solution is then prolongated to the next finer mesh to generate the initial guess for running HIK on the next finest grid. As the initial guess is then usually close to the solution, the HIK solver has fewer dofs to peel away from the (wrong) active set estimate.  We further analyse the peeling effect in \Cref{sec:peeling}.
\begin{figure}[h!]
\centering
\subfloat[1D: 16 cells]{\includegraphics[height =0.22\textwidth]{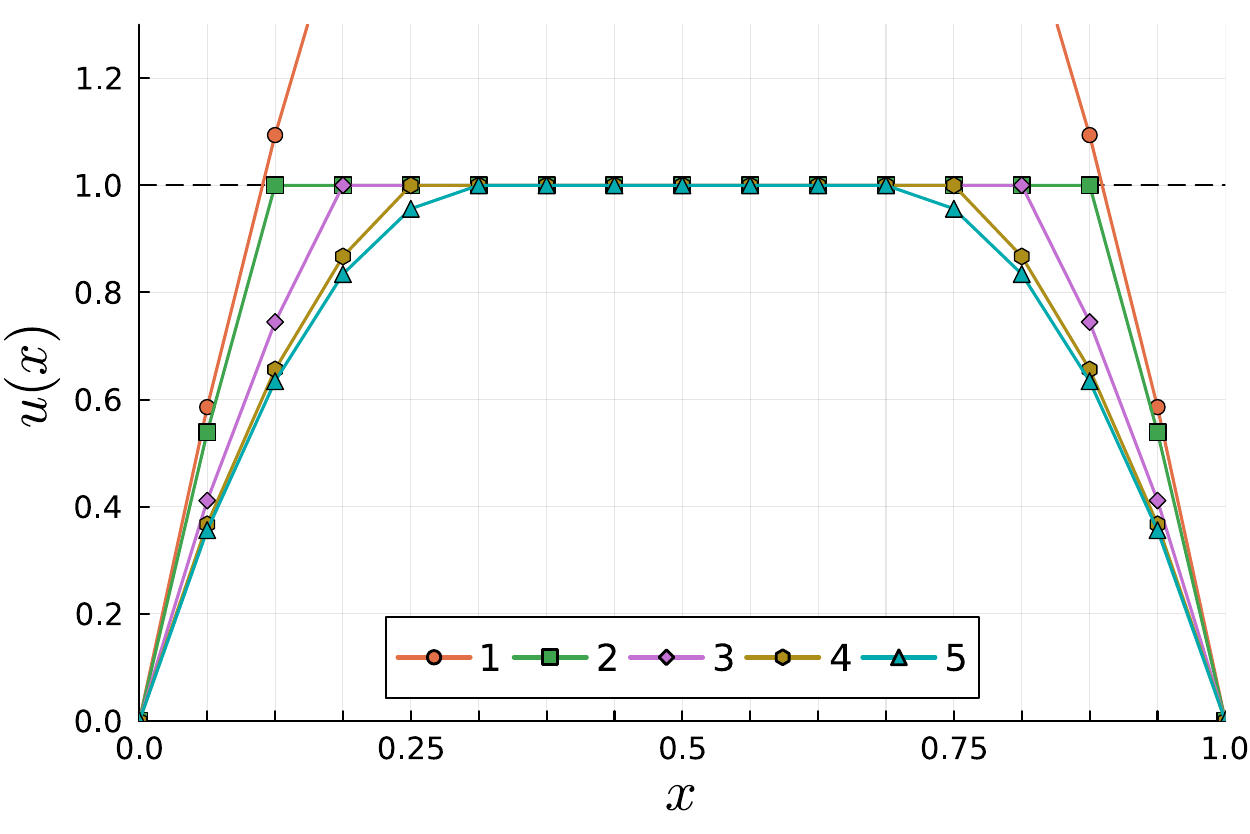}}
\subfloat[1D: 32 cells]{\includegraphics[height =0.22\textwidth]{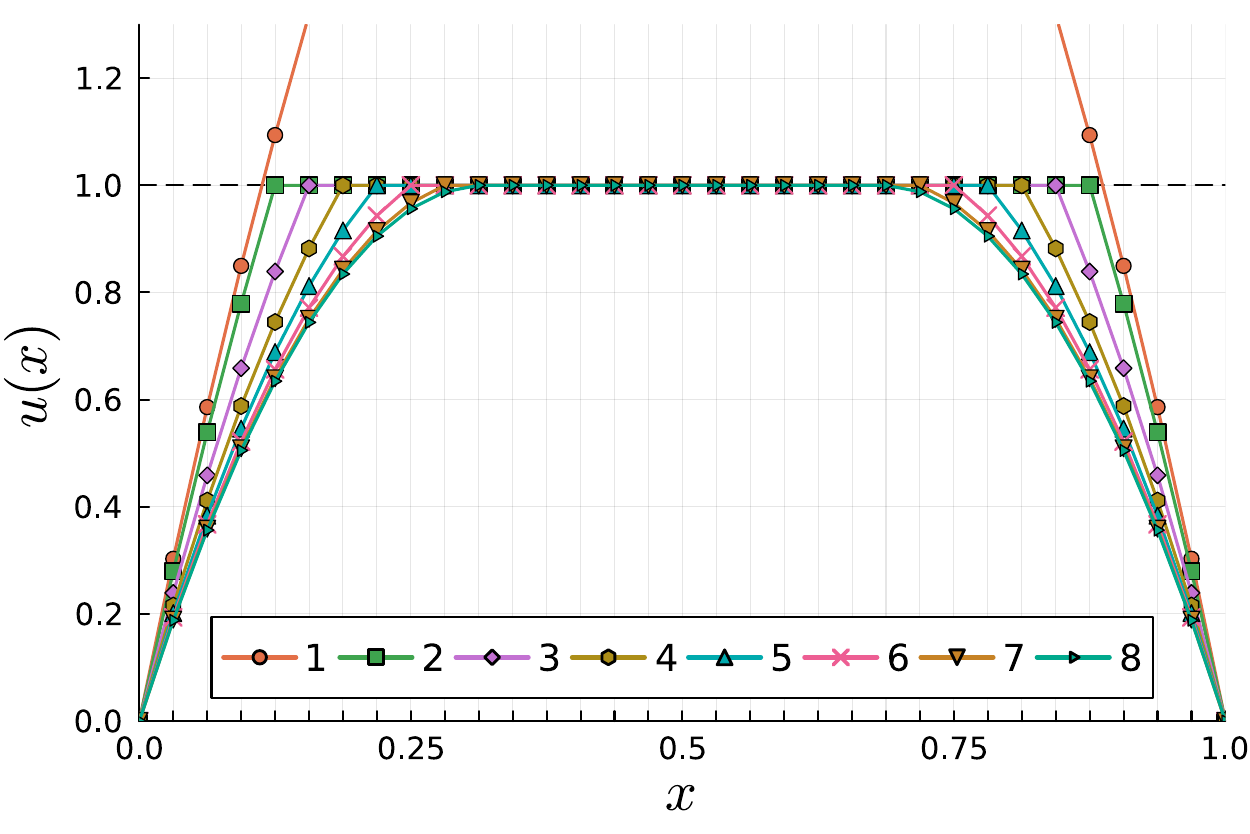}} \;
\subfloat[2D: $16 \times 16$ mesh]{\includegraphics[height =0.22 \textwidth]{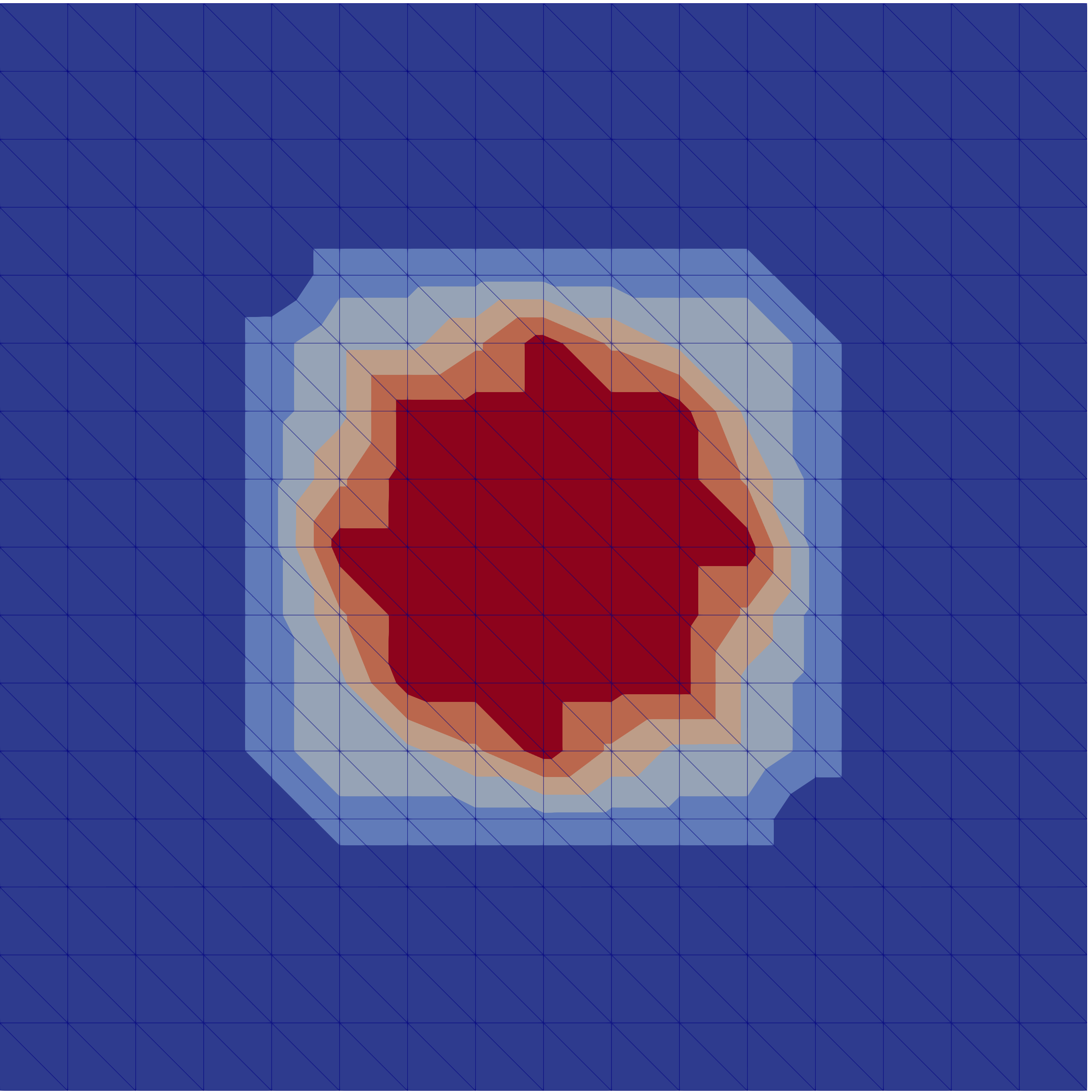}}
\includegraphics[height =0.22\textwidth]{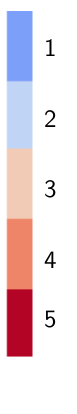}
\caption{Evolution of the HIK iterates, in 1D (left and middle), and the active set, in 2D (right), when the HIK solver is applied to the obstacle problem of \cref{prob:obstacle} on various meshes. The legend numbers are the iteration value $k$. In 2D, the active sets are nested and shrink as the iteration count increases.  We observe that a dof (or equivalently a vertex of a cell in the mesh) only switches from the active to the inactive set if it shares an edge with a dof in the inactive set. Visually, the iterates appear to be peeling away from the obstacle one layer of dofs at a time.}
\label{fig:peeling}
\end{figure}

\smallskip
\noindent \textbf{Results for \cref{prob:thin,prob:signorini}.} The HIK solver produces mesh-dependent iteration counts when applied to the thin obstacle and Signorini problems of \cref{prob:thin,prob:signorini}. However, in stark contrast to the obstacle problem of \cref{prob:obstacle}, peeling does not occur layer-by-layer and dofs can switch significantly faster from the active to the inactive set within one iteration. This is demonstrated in \cref{fig:signorini-iterations} and was also observed in \cite{hintermuller2003} and \cite[Tab.~1--6]{hueber2005}. 

The iteration counts only grew by 1--3 iterations with each refinement. \cref{fig:convergence} reveals a seemingly mesh-independent \emph{linear} convergence of the iterates. A similar behaviour has been proven for projection methods accelerated via domain decomposition \cite{schoberl1998}. In the penultimate HIK iteration, the active set is correctly identified and in the subsequent iteration the error drops to zero as the primal equation is linear. This active set identification occurs later for each mesh refinement which causes the mild growth in the number of iterations.

\begin{figure}[h!]
\centering
\subfloat[It.~2: 69 active dofs]{\includegraphics[width=0.3\textwidth]{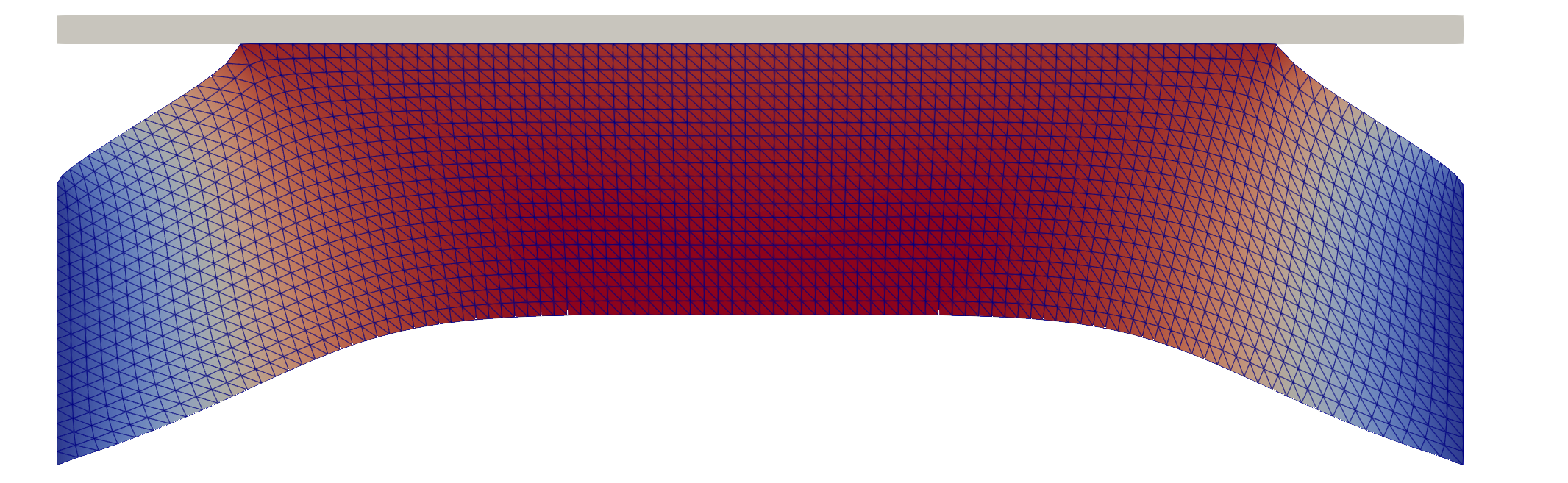}}\;\;
\subfloat[It.~3: 51 active dofs]{\includegraphics[width=0.3\textwidth]{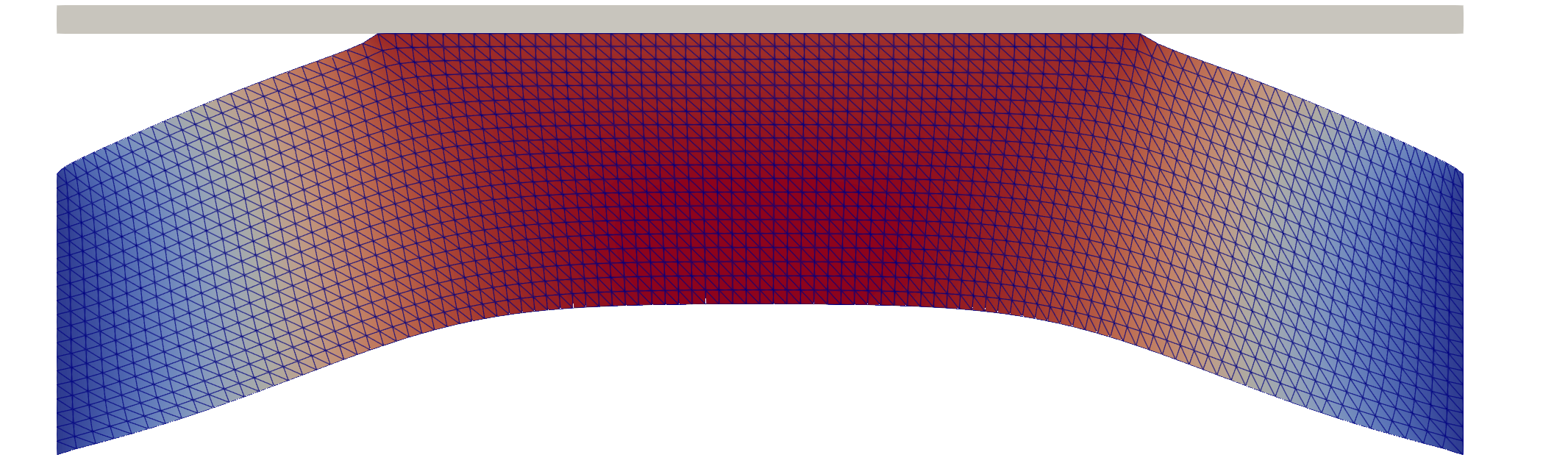}}\;\;
\subfloat[It.~4: 41 active dofs]{\includegraphics[width=0.3\textwidth]{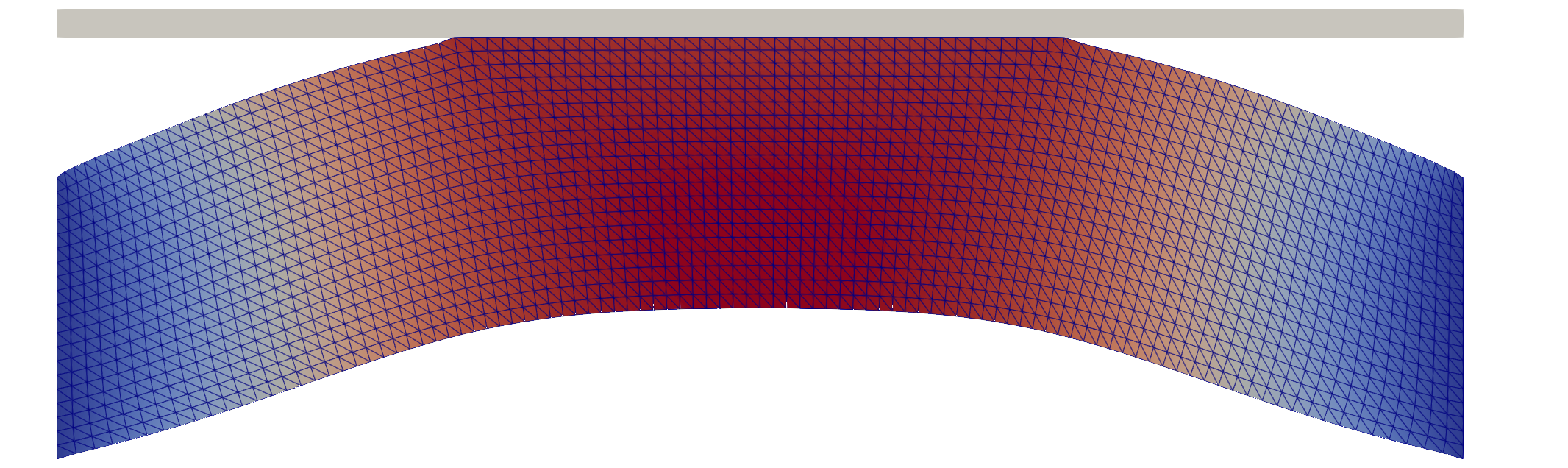}}\\
\subfloat[It.~2: 1092 active dofs]{\includegraphics[width=0.3\textwidth]{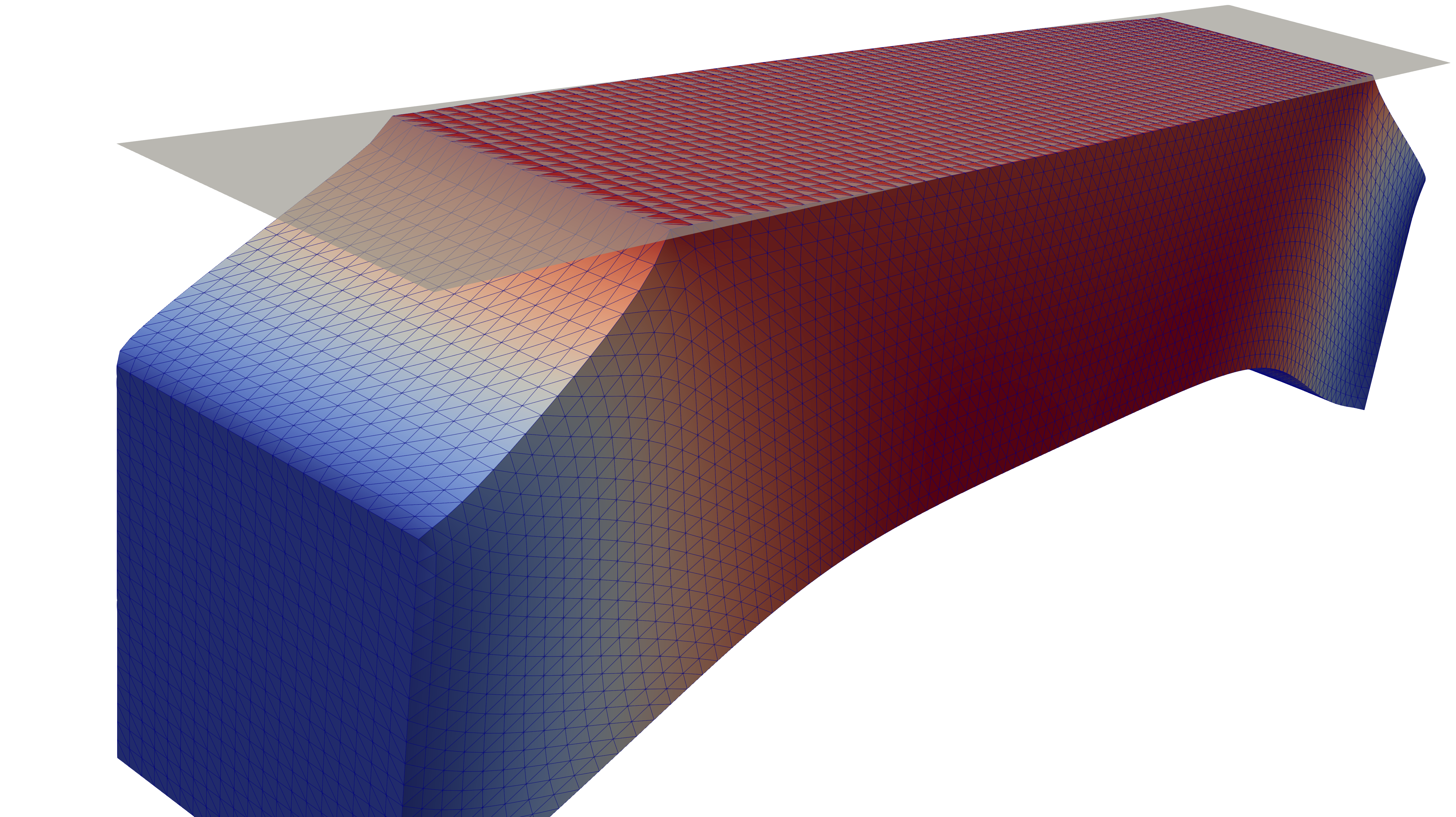}}\;\;
\subfloat[It.~3: 880 active dofs]{\includegraphics[width=0.3\textwidth]{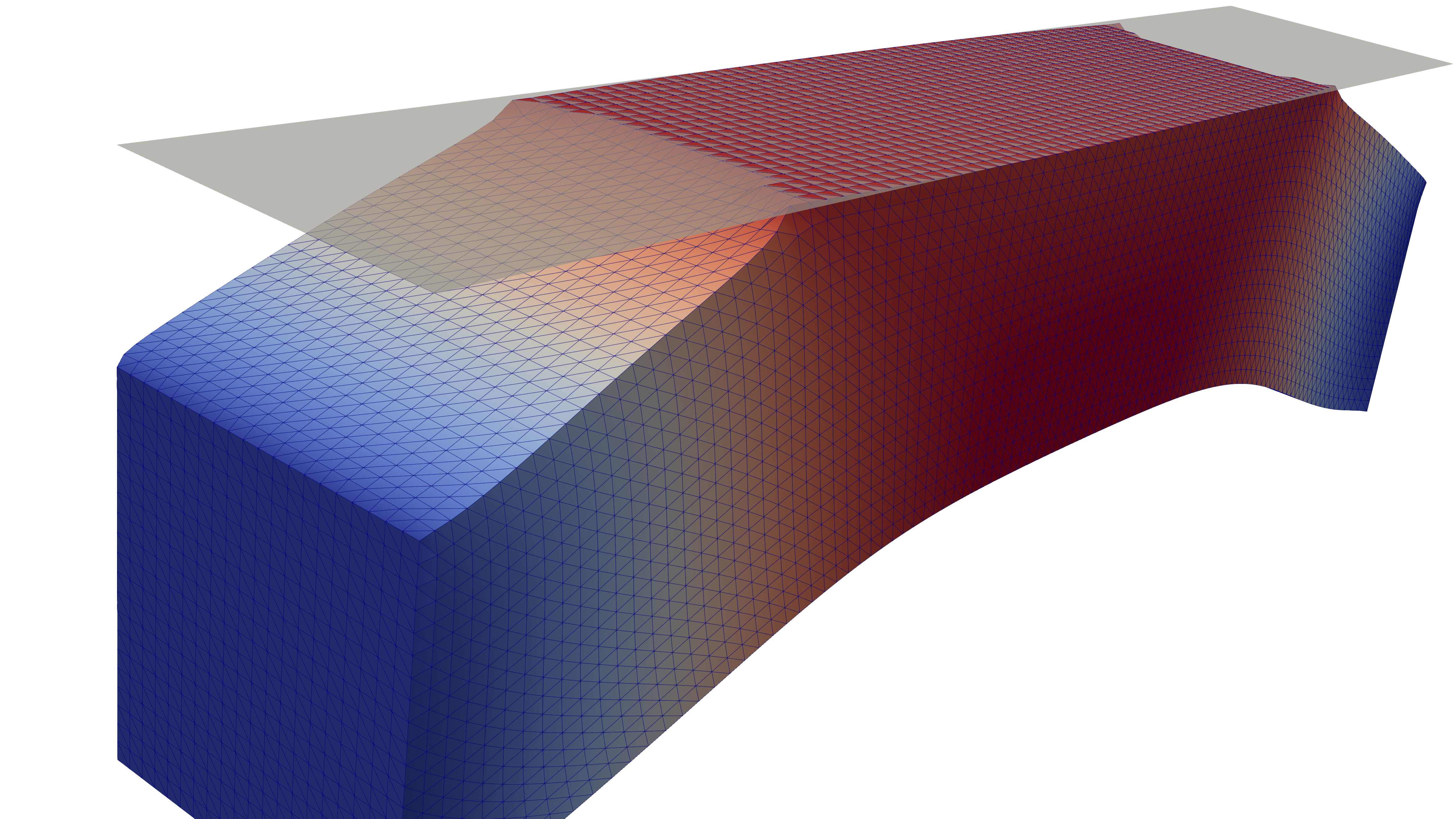}}\;\;
\subfloat[It.~4: 752 active dofs]{\includegraphics[width=0.3\textwidth]{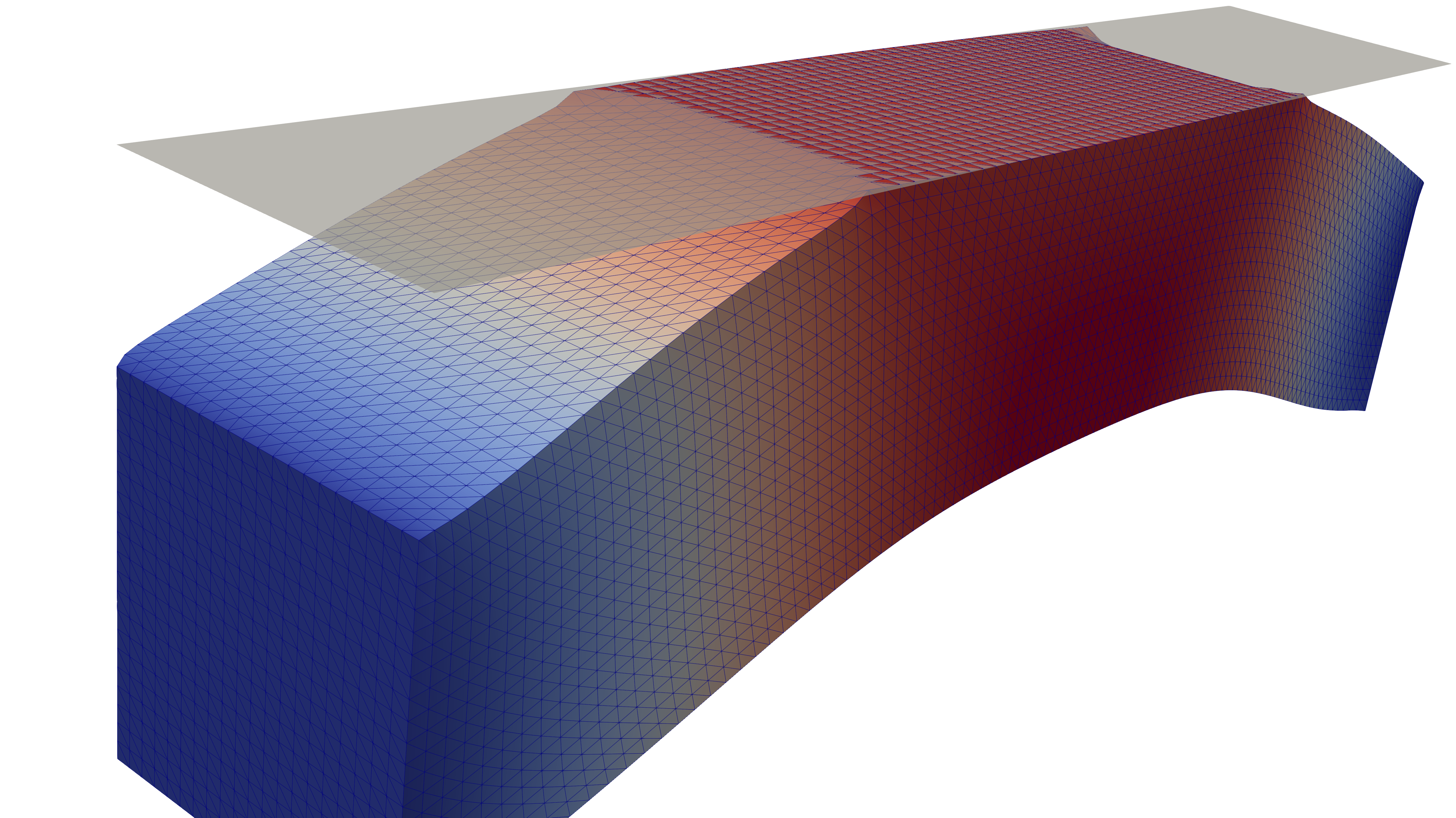}}
\caption{The second to fourth HIK iterates for the two- and three-dimensional Signorini problem of \cref{prob:signorini} on a $100\times20$ mesh (4158 dofs) and $100\times20\times20$ mesh (130,977 dofs), respectively. We report the number of active set dofs at each iteration. We observe that, unlike the obstacle problem, many active set dofs can switch to the inactive set in one step.}
\label{fig:signorini-iterations}
\end{figure}

\section{A global convergence rate}
\label{sec:convergence}
Beside the local superlinear convergence of HIK due to its equivalence to SSN, in \cite{hintermuller2002} global convergence results are provided depending on properties of the matrix $A$ (such as $A$ being an M-matrix) or $A$ satisfying specific sign conditions. These global convergence results, however, do not offer any global rate. The goal of this section is now to derive, under conditions, such a global convergence rate for the HIK iterates in \cref{alg:hik} on a fixed mesh. For the ease of notation, we drop the subscript $_h$ in conjunction with coefficient vectors and (in)active sets.

We begin with a useful, technical lemma. 
\begin{lemma}
\label{lem:linear}
Let $(\vectt{u}^*, \vectt{\lambda}^*)$ denote the solution to \cref{pdas:1}, with $B=I$, $c=1$, and $(\vectt{u}^k, \vectt{\lambda}^k)$ the $k^{\text{th}}$ iterate of \cref{alg:hik}. Suppose that $A$ is symmetric positive-definite and $k \geq 1$.  Then 
\begin{enumerate}[label={\rm(}\roman*{\rm)}]
\item $A\vectt{u}^k + \vectt{\lambda}^k - \vectt{b} = \vectt{0}$, \label{lem:linear:i}
\item $A_{\mathcal{I}^k \mathcal{A}^k} (\vectt{u}^k-\vectt{u}^*)_{\mathcal{A}^k} + A_{\mathcal{I}^k \mathcal{I}^k} (\vectt{u}^k-\vectt{u}^*)_{\mathcal{I}^k} +  (\vectt{\lambda}^k-\vectt{\lambda}^*)_{\mathcal{I}^k} = \vectt{0}$,  \label{lem:linear:ii}
\item $\| \vectt{u}^k - \vectt{u}^*\|_A = \| \vectt{\lambda}^k - \vectt{\lambda}^*\|_{A^{-1}}$, \label{lem:linear:iii}
\end{enumerate}
If we further assume that $\vectt{u}^k \leq \vectt{\varphi}$, then
\begin{enumerate}[label=(\roman*)]
\setcounter{enumi}{3}
\item $\vectt{u}^k_{\mathcal{A}^k} = \vectt{\varphi}_{\mathcal{A}^k}$.  \label{lem:linear:iv}
\end{enumerate}
\end{lemma}
\begin{proof}
\cref{lem:linear}\labelcref{lem:linear:i} follows from the discussion in \Cref{sec:pdas}. Since $A\vectt{u}^* + \vectt{\lambda}^* - \vectt{b} = \vectt{0}$, from \cref{lem:linear}\labelcref{lem:linear:i}, we have that
\begin{align}
A(\vectt{u}^k - \vectt{u}^*) + (\vectt{\lambda}^k - \vectt{\lambda}^*) = \vectt{0}.
\label{eq:linear:1}
\end{align}
Thus \cref{lem:linear}\labelcref{lem:linear:ii} follows by considering the rows corresponding to the inactive set. Similarly, from \cref{eq:linear:1}, \cref{lem:linear}\labelcref{lem:linear:iii} is deduced via
 \begin{align*}
& \| \vectt{u}^k - \vectt{u}^* \|^2_A = (\vectt{u}^k - \vectt{u}^*)^\top A (\vectt{u}^k - \vectt{u}^*) \\
&\quad = - (\vectt{u}^k - \vectt{u}^*)^\top  (\vectt{\lambda}^k - \vectt{\lambda}^*) = (\vectt{\lambda}^k - \vectt{\lambda}^*)^\top A^{-1}  (\vectt{\lambda}^k - \vectt{\lambda}^*) =  \| \vectt{\lambda}^k - \vectt{\lambda}^*\|^2_{A^{-1}}.
\end{align*}
By assumption $\vectt{u}^k \leq \vectt{\varphi}$. Hence, from \cref{eq:static}, if $\vectt{u}^k_i < \vectt{\varphi}$, then $i \in \mathcal{I}^{k-1}$ and thus $\vectt{\lambda}^k_i = 0$. But then by definition of $\mathcal{A}^k$, $i \not\in \mathcal{A}^k$. Hence we must have $\vectt{u}^k_{\mathcal{A}^k} = \vectt{\varphi}_{\mathcal{A}^k}$ and \cref{lem:linear}\labelcref{lem:linear:iv} follows.
\end{proof}

\begin{remark}[Primal iterate feasibility]\label{rem_hik}
In \cite[Th.~3.2]{hintermuller2002} it was shown that $\vectt{u}^k \leq \vectt{\varphi}$ for $k \geq 2$ and $\vectt{u}^*\leq\vectt{u}^{k+1}\leq\vectt{u}^k$ for $k\geq 1$(see also \cite[Th.~1]{hoppe1987}) when $A$ is a (non-singular) $M$-matrix. It also holds for small perturbations of an $M$-matrix whenever $k \geq k_0$ for some $k_0 \in \mathbb{N}$ \cite[Th.~3.3]{hintermuller2003}.
\end{remark}

In the next theorem, we deduce the exact rate of convergence of \cref{alg:hik}.
\begin{theorem}
\label{th:convergence-rate}
Let $(\vectt{u}^*, \vectt{\lambda}^*)$ denote the solution to \cref{pdas:1}, with $B=I$, $c=1$, and $(\vectt{u}^k, \vectt{\lambda}^k)$ the $k^{\text{th}}$ iterate of \cref{alg:hik}. Suppose that $A$ is a symmetric positive-definite matrix. Consider an HIK iterate $k \geq 1$ such that $\mathcal{A}^* \subseteq \mathcal{A}^k$ and, consequently, $\mathcal{I}^k \subseteq \mathcal{I}^*$. Moreover, suppose that $\vectt{u}^k \leq \vectt{\varphi}$. Then
\begin{align}
\label{eq:convergence-recurrence}
\| \vectt{u}^{k+1} - \vectt{u}^* \|^2_A - \| \vectt{u}^{k} - \vectt{u}^* \|^2_A + \| \vectt{\lambda}^k_{\mathcal{I}^k} \|^2_{(A_{\mathcal{I}^k,\mathcal{I}^k})^{-1}} = 0
\end{align}
and, if $\vectt{u}^k \neq \vectt{u}^*$, then
\begin{align}
\frac{\| \vectt{u}^{k+1} - \vectt{u}^* \|_A}{{\| \vectt{u}^{k} - \vectt{u}^* \|_A}}
=
\frac{\| \vectt{\lambda}^{k+1} - \vectt{\lambda}^* \|_{A^{-1}}}{{\| \vectt{\lambda}^{k} - \vectt{\lambda}^* \|_{A^{-1}}}} 
= \rho(\vectt{\lambda}^k)
\coloneqq \left(1-\frac{\| \vectt{\lambda}^k_{\mathcal{I}^k} \|^2_{(A_{\mathcal{I}^k,\mathcal{I}^k})^{-1}}}{\| \vectt{\lambda}^{k} - \vectt{\lambda}^* \|_{A^{-1}}^2}\right)^{1/2}.
\label{eq:contraction}
\end{align} 
\end{theorem}
\begin{proof}
From \cref{eq:HIKstep} we have that
\begin{align*}
&\begin{pmatrix}
\vectt{u}^{k+1} - \vectt{u}^*\\
\vectt{\lambda}^{k+1} - \vectt{\lambda}^*
\end{pmatrix}
= 
\begin{pmatrix}
\vectt{u}^{k} - \vectt{u}^*\\
\vectt{\lambda}^{k} - \vectt{\lambda}^*
\end{pmatrix}
-
\begin{pmatrix}
A & I\\
-E^{\mathcal{A}^k} & E^{\mathcal{I}^k}
\end{pmatrix}^{-1}
\begin{pmatrix}
A\vectt{u}^k + \vectt{\lambda}^k - \vectt{b}\\
\vectt{g}^k
\end{pmatrix}\\
&\quad =\begin{pmatrix}
A & I\\
-E^{\mathcal{A}^k} & E^{\mathcal{I}^k}
\end{pmatrix}^{-1}
\left(
\begin{pmatrix}
A & I\\
-E^{\mathcal{A}^k} & E^{\mathcal{I}^k}
\end{pmatrix}
\begin{pmatrix}
\vectt{u}^{k} - \vectt{u}^*\\
\vectt{\lambda}^{k} - \vectt{\lambda}^*
\end{pmatrix}
- 
\begin{pmatrix}
A\vectt{u}^k + \vectt{\lambda}^k - \vectt{b}\\
\vectt{g}^k
\end{pmatrix}
\right)\\
& \quad =
\begin{pmatrix}
A & I\\
-E^{\mathcal{A}^k} & E^{\mathcal{I}^k}
\end{pmatrix}^{-1}
\begin{pmatrix}
\vectt{0}\\
(\vectt{u}^* - \vectt{\varphi})_{\mathcal{A}^k}  -  \vectt{\lambda}^*_{\mathcal{I}^k} 
\end{pmatrix}
=
\begin{pmatrix}
A & I\\
-E^{\mathcal{A}^k} & E^{\mathcal{I}^k}
\end{pmatrix}^{-1}
\begin{pmatrix}
\vectt{0}\\
- E^{\mathcal{A}^k}(\vectt{u}^k - \vectt{u}^*)_{\mathcal{A}^k} 
\end{pmatrix},
\end{align*}
where we have leveraged \cref{lem:linear}\labelcref{lem:linear:i} and use $A\vectt{u}^* + \vectt{\lambda}^* = \vectt{b}$ for the penultimate equality. The final equality follows from \cref{lem:linear}\labelcref{lem:linear:iv} and since $\mathcal{I}^k \subseteq \mathcal{I}^*$, then $\vectt{\lambda}^*_{\mathcal{I}^k}  = \vectt{0}$.

By defining the Schur complement $S^k = E^{\mathcal{I}^k} + E^{\mathcal{A}^k} A^{-1}$, a Schur complement factorization reveals that
\begin{align*}
\begin{pmatrix}
\vectt{u}^{k+1} - \vectt{u}^*\\
\vectt{\lambda}^{k+1} - \vectt{\lambda}^*
\end{pmatrix}
&= 
\begin{pmatrix}
A^{-1} (S^k)^{-1}E^{\mathcal{A}^k}(\vectt{u}^{k} - \vectt{u}^*)\\
-(S^k)^{-1}E^{\mathcal{A}^k}(\vectt{u}^{k} - \vectt{u}^*)
\end{pmatrix}.
\end{align*}
By examining the first row, we observe that
\begin{align*}
\vectt{u}^{k+1} - \vectt{u}^* =  ( E^{\mathcal{I}^k}  A + E^{\mathcal{A}^k})^{-1} E^{\mathcal{A}^k} ( \vectt{u}^{k} - \vectt{u}^*).
\end{align*}
Now splitting (and re-arranging) the inactive and active indices, we note that 
\begin{align}
\begin{split}
&( E^{\mathcal{I}^k}  A + E^{\mathcal{A}^k})^{-1} E^{\mathcal{A}^k} = 
\begin{pmatrix}
A_{\mathcal{I}^k \mathcal{I}^k} & A_{\mathcal{I}^k \mathcal{A}^k}\\
\vectt{0} & I_{\mathcal{A}^k \mathcal{A}^k}
\end{pmatrix}^{-1}
\begin{pmatrix}
\vectt{0}  &\vectt{0}  \\
 \vectt{0} & I_{\mathcal{A}^k \mathcal{A}^k}
\end{pmatrix}\\
&\qquad=
\begin{pmatrix}
A_{\mathcal{I}^k \mathcal{I}^k}^{-1} & -A_{\mathcal{I}^k \mathcal{I}^k}^{-1} A_{\mathcal{I}^k \mathcal{A}^k}\\
\vectt{0} & I_{\mathcal{A}^k \mathcal{A}^k}
\end{pmatrix}
\begin{pmatrix}
\vectt{0}  &\vectt{0}  \\
 \vectt{0} & I_{\mathcal{A}^k \mathcal{A}^k}
\end{pmatrix}
=
\begin{pmatrix}
\vectt{0} & -A_{\mathcal{I}^k \mathcal{I}^k}^{-1} A_{\mathcal{I}^k \mathcal{A}^k}\\
\vectt{0} & I_{\mathcal{A}^k \mathcal{A}^k}
\end{pmatrix}.
\end{split}
\label{eq:conv:1}
\end{align}
Hence, we get
\begin{align}
(( E^{\mathcal{I}^k}  A + E^{\mathcal{A}^k})^{-1} E^{\mathcal{A}^k} )^\top A (( E^{\mathcal{I}^k}  A + E^{\mathcal{A}^k})^{-1} E^{\mathcal{A}^k} )
= 
\begin{pmatrix}
\vectt{0} & \vectt{0}\\
\vectt{0} & Q^k
\end{pmatrix},
\end{align}
where $Q^k$ is the Schur complement $Q^k \coloneqq A_{\mathcal{A}^k \mathcal{A}^k} - A_{\mathcal{A}^k \mathcal{I}^k} (A_{\mathcal{I}^k \mathcal{I}^k} )^{-1} A_{\mathcal{I}^k \mathcal{A}^k}$. Thus
\begin{align}
\|\vectt{u}^{k+1} - \vectt{u}^*\|^2_A = \|(\vectt{u}^{k} - \vectt{u}^*)_{\mathcal{A}^k}\|^2_{Q^k}.
\end{align}
Setting $\vectt{p}^k \coloneqq \vectt{u}^{k} - \vectt{u}^*$ and leveraging \cref{lem:linear}\labelcref{lem:linear:ii}, we observe that
\begin{align*}
&\|\vectt{u}^{k+1} - \vectt{u}^*\|^2_A - \|\vectt{u}^{k} - \vectt{u}^*\|^2_A \\
&\quad = (\vectt{p}^k_{\mathcal{A}^k})^\top Q^k \vectt{p}^k_{\mathcal{A}^k} 
- \left(
(\vectt{p}^k_{\mathcal{A}^k})^\top A_{\mathcal{A}^k \mathcal{A}^k}\vectt{p}^k_{\mathcal{A}^k}
+ 2 (\vectt{p}^k_{\mathcal{A}^k})^\top A_{\mathcal{A}^k \mathcal{I}^k}\vectt{p}^k_{\mathcal{I}^k}
+ (\vectt{p}^k_{\mathcal{I}^k})^\top A_{\mathcal{I}^k \mathcal{I}^k}\vectt{p}^k_{\mathcal{I}^k}
\right)\\
& \quad = 
-(\vectt{p}^k_{\mathcal{A}^k})^\top A_{\mathcal{A}^k \mathcal{I}^k} (A_{\mathcal{I}^k \mathcal{I}^k} )^{-1} A_{\mathcal{I}^k \mathcal{A}^k}\vectt{p}^k_{\mathcal{A}^k} - (\vectt{p}^k_{\mathcal{I}^k})^\top A_{\mathcal{I}^k \mathcal{I}^k}\vectt{p}^k_{\mathcal{I}^k}
- 2 (\vectt{p}^k_{\mathcal{A}^k})^\top A_{\mathcal{A}^k \mathcal{I}^k}\vectt{p}^k_{\mathcal{I}^k}\\
& \quad = - (\vectt{\lambda}^k_{\mathcal{I}^k})^\top (A_{\mathcal{I}^k \mathcal{I}^k})^{-1}\vectt{\lambda}^k_{\mathcal{I}^k} = -\|\vectt{\lambda}^k_{\mathcal{I}^k}\|^2_{ (A_{\mathcal{I}^k \mathcal{I}^k})^{-1}}.
\end{align*}
By rearranging the equation, dividing through by $\|\vectt{u}^{k} - \vectt{u}^*\|^2_A$, and two final applications of \cref{lem:linear}\labelcref{lem:linear:iii}, we arrive at the result. 
\end{proof}
The convergence rate in \cref{th:convergence-rate} reveals that the convergence of \cref{alg:hik} is dictated by the size of the Lagrange multiplier iterate in the current inactive set. Moreover, for any $i \in \mathcal{I}^k \cap \mathcal{I}^{k-1}$ it holds that $\vectt{\lambda}^k_i = 0$ and for any $j \in \mathcal{I}^k \cap \mathcal{A}^{k-1}$ we have $\vectt{u}^k_j = \vectt{\varphi}_j$ and thus $\vectt{\lambda}^k_j \leq 0$. Hence $\vectt{\lambda}^k_{\mathcal{I}^k} \leq \vectt{0}$. In other words, $\| \vectt{\lambda}^k_{\mathcal{I}^k} \|^2_{(A_{\mathcal{I}^k,\mathcal{I}^k})^{-1}}$ is a measure of the dual feasibility violation of the Lagrange multiplier iterate. Moreover, if $\| \vectt{\lambda}^k_{\mathcal{I}^k} \|^2=0$, then $\vectt{\lambda}^k=\vectt{\lambda}^*$ and $\vectt{u}^k=\vectt{u}^*$.
We also note that in view of \cref{rem_hik} the assumptions of \cref{th:convergence-rate} are satisfied for $k\geq 2$ when $A$ is a non-singular symmetric M-matrix (i.e., a Stieltjes matrix). The latter is the case, e.g., for the obstacle problem in our setting.

In finite dimensions, equivalence to an SSN or the aforementioned $M$-matrix property of $A$, respectively, ensure that $\rho(\vectt{\lambda}^k) \to 0$ as $k \to \infty$. We may deduce the mesh-dependent convergence behaviour of \cref{alg:hik}, as $h\to 0$, by examining the infinite-dimensional Lagrange multiplier. The numerical evidence in \cref{fig:convergence} suggests that, for the obstacle problem, $\rho(\vectt{\lambda}^k) \to 1$ as $h \to 0$ for any $k$. Whereas, in the thin obstacle and Signorini problems,  $\rho(\vectt{\lambda}^k)$ is uniformly bounded below one. We further analyse this in \Cref{sec:ssn}.%

\section[Peeling]{Sticky active sets and the layer-by-layer peeling effect}
\label{sec:peeling}

This section is dedicated to explaining the layer-by-layer peeling effect observed for obstacle problems during the deactivation phase cf.~\cref{fig:peeling}. At the end of the section, we remark why the settings of the thin obstacle and Signorini problems do not cause layer-by-layer peeling. We fix a mesh and drop the subscript $_h$ in this section for brevity. Moreover, throughout this section we assume that $f\geq c_f>0$ and $\varphi_h = \varphi \equiv \bar \varphi > 0$.

Recall that a P1-FEM discretization is employed. Hence, all dofs $i \in \{1\colto N\}$ are associated with a node $x_i \in \mathbb{R}^d$, $d \in \{1,2,3\}$, such that the FEM basis functions satisfy $\phi_i(x_j) = \delta_{ij}$. Further, let $\omega(x_i)$ denote the star patch around a node of the triangulation $x_i \in \mathbb{R}^d$, i.e.
\begin{align}
\omega(x_i) \coloneqq \cup_{\{K \in \mathcal{T}_h\, : \, x_i \in \bar K\}} K,
\label{eq:star-patch}
\end{align}

\begin{definition}[Strict interior of $\mathcal{A}^k$]
We say that the dof $i \in \op{int} \mathcal{A}^k$ 
if, for all $x_j \in \overline{\omega(x_i)}$, it holds that $j  \in \mathcal{A}^k$. 
\end{definition}

Our next goal is to prove the following two results which are concerned with a detailed analysis of the (discrete) active/inactive set behavior of HIK iterations. 
\begin{theorem}[Sticky active sets]
\label{th:peeling}
Suppose that \cref{alg:hik} is applied to a P1-FEM discretization of \cref{prob:obstacle} resulting in \cref{pdas:1} with $B=I$ and $c=1$. 
At the $k^{\text{th}}$ HIK iteration, suppose further that $i \in \operatorname{int} \mathcal{A}^k$. Then $i \in \mathcal{A}^{k+1}$. 
\end{theorem}
The proof is given at the end of this section. \cref{th:peeling} directly implies the following corollary.

\begin{corollary}[Peeling effect]
\label{cor:peeling}
Under the conditions of \cref{th:peeling}, if a dof $i$ associated with a node $x_i$ satisfies $i \in \mathcal{A}^k \cap \mathcal{I}^{k+1}$ then there must exist a node $x_j \in \overline{\omega(x_i)}$ such that $j \in \mathcal{I}^k$. In other words, a dof can only switch from the active set to the inactive set if it neighbours a dof in the inactive set at the $k^{\text{th}}$ HIK iteration.
\end{corollary}
\cref{th:peeling} and, consequently \cref{cor:peeling}, reveal that if a dof is in the active set and all its neighbouring dofs are also in the active set then, at the next HIK iteration, the dof must remain in the active set. Conversely, only dofs on the boundary of the current active set are able to deactivate. We visualize this in \cref{fig:peeling2}.
\begin{figure}[h!]
\centering
\includegraphics[width =0.25 \textwidth]{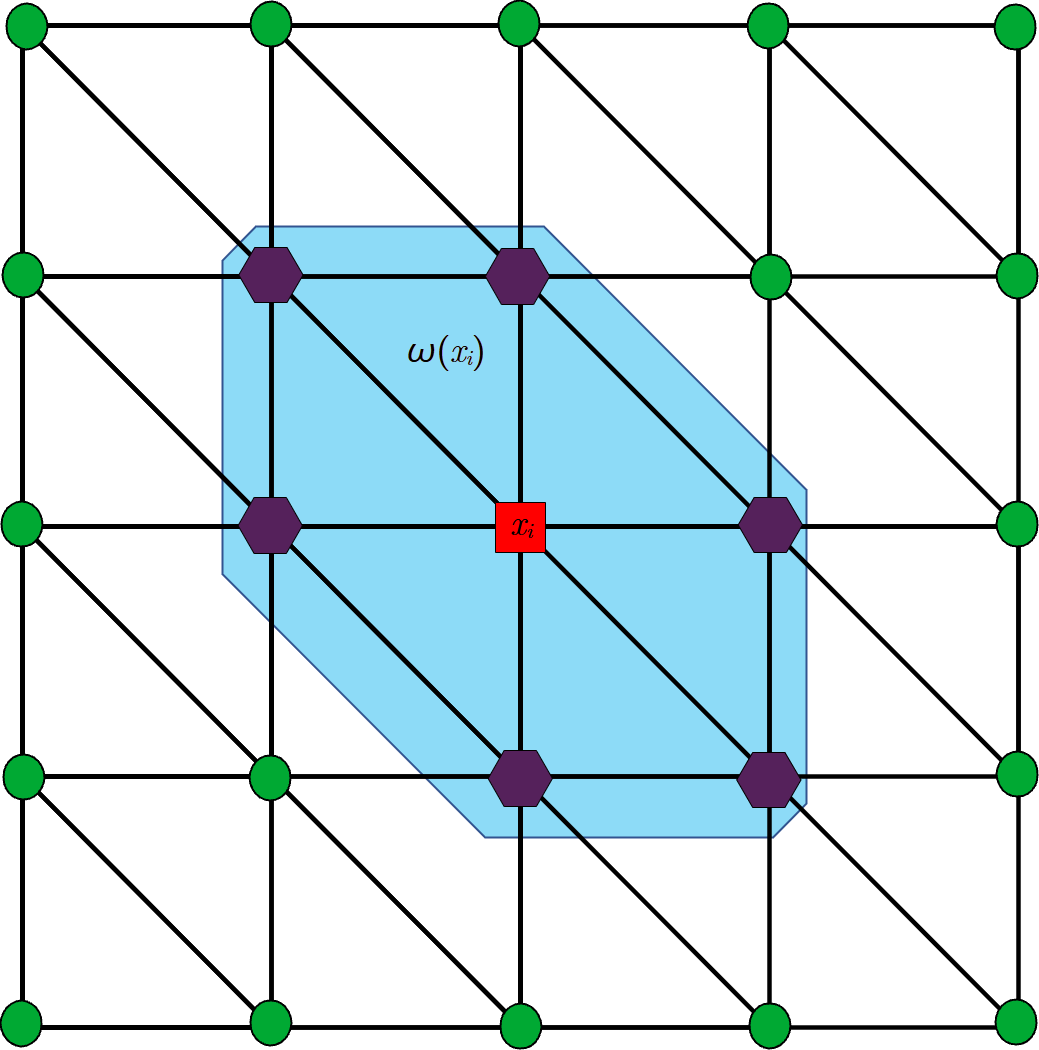}
\caption{The red square and purple hexagons represent active set dofs, $j \in \mathcal{A}^k$. The green dots are inactive set dofs $j \in \mathcal{I}^k$. The blue patch around the node $x_i$ represents its star patch $\omega(x_i)$ defined in \cref{eq:star-patch}. \cref{th:peeling} reveals that the dof $i$, associated with $x_i$, must remain in the active set at the next HIK iteration as all its neighbours in the star patch are also active set dofs. Conversely, the hexagonal dofs may (but not necessarily) deactivate as they neighbour inactive set dofs.}
\label{fig:peeling2}
\end{figure}

If the active set prediction is connected, then, in one dimension, \cref{cor:peeling} %
reveals that at most two dofs can deactivate (peel away from the obstacle) per HIK iteration.

In order to prove the result in \cref{th:peeling}, we first provide a technical lemma.
\begin{lemma}
\label{lem:switching}
Suppose that the conditions in \cref{th:peeling} hold. At the $k^{\text{th}}$ HIK iteration consider a dof $i \in \op{int} \mathcal{A}^k$. Then, for any $\vectt{v} \in \mathbb{R}^N$,
\begin{align}
[A_{\{1 \colto N\}, \mathcal{I}^k}  \vectt{v}_{\mathcal{I}^k}]_i = 0. \label{eq:switch1}
\end{align}
 Moreover suppose that $\mathcal{S}^k  = \{i  :  i \in \mathcal{A}^k \cap \mathcal{I}^{k+1}\} \neq \emptyset$. Then
\begin{align}
A_{\mathcal{S}^k, \{1 \colto N\}} \vectt{\delta}^k \geq - [A\vectt{u}^k - \vectt{b}]_{\mathcal{S}^k}. \label{eq:switch2}
\end{align}
\end{lemma}
\begin{proof}
We first show \cref{eq:switch1} and note that
\begin{align}
 [A_{\{1 \colto N\}, \mathcal{I}^k}  \vectt{v}_{\mathcal{I}^k}]_i =  \sum_{\ell \in \mathcal{I}^k} \int_\Omega \nabla \phi_i \cdot (\vectt{v}_\ell \nabla \phi_\ell) \, \dx. \label{eq:s3}
\end{align}
But since $i \in \op{int} \mathcal{A}^k$, $\op{supp}(\nabla \phi_i) \cap \op{supp}(\nabla \phi_\ell) \neq \emptyset$ only for some $\ell \in \mathcal{A}^k$. Hence \cref{eq:s3} evaluates to zero. The first row of \cref{eq:HIKstep} implies that
\begin{align}
A_{\mathcal{S}^k, \{1 \colto N\}} \vectt{\delta}^k + \vectt{\varepsilon}^k_{\mathcal{S}^k} = -(A\vectt{u}^k + \vectt{\lambda}^k - \vectt{b})_{\mathcal{S}^k}.
\end{align}
Moving $\vectt{\lambda}^k$ to the left, we discover that
\begin{align}
A_{\mathcal{S}^k, \{1 \colto N\}} \vectt{\delta}^k + \vectt{\lambda}^{k+1}_{\mathcal{S}^k} = -(A\vectt{u}^k  - \vectt{b})_{\mathcal{S}^k}. \label{eq:s5}
\end{align}
Since $\mathcal{S}^k \subseteq \mathcal{I}^{k+1}$, line 8 in \cref{alg:hik} implies that
\begin{align}
\vectt{\lambda}^{k+1}_{\mathcal{S}^k} + \vectt{u}^{k+1}_{\mathcal{S}^k} - \vectt{\varphi}_{\mathcal{S}^k} \leq \vectt{0}. \label{eq:s4}
\end{align}
But since $\mathcal{S}^k \subseteq \mathcal{A}^{k}$, then line 6 of \cref{alg:hik} implies that $\vectt{u}^{k+1}_{\mathcal{S}^k} - \vectt{\varphi}_{\mathcal{S}^k} = \vectt{0}$ which in turn reduces \cref{eq:s4} to 
\begin{align}
\vectt{\lambda}^{k+1}_{\mathcal{S}^k} \leq \vectt{0}. \label{eq:s6}
\end{align}
Substituting \cref{eq:s6} into \cref{eq:s5} we recover \cref{eq:switch2}.
\end{proof}

We now provide the proof of \cref{th:peeling}.
\begin{proof}[Proof of \cref{th:peeling}]
The result will be shown via a proof by contradiction. Suppose that there exists a dof $i \in \op{int} \mathcal{A}^k$ such that $i \in \mathcal{I}^{k+1}$, i.e., 
\begin{align}
\hat{\mathcal{S}}^k \coloneqq \{ i : i \in \op{int} \mathcal{A}^k \cap \mathcal{I}^{k+1} \} \neq \emptyset.
\end{align}
Since $\hat{\mathcal{S}}^k  \subseteq \mathcal{S}^k$ then \cref{eq:switch2} in \cref{lem:switching} implies that
$
A_{\hat{\mathcal{S}}^k, \{1 \colto N\}} \vectt{\delta}^k \geq -(A\vectt{u}^k  - \vectt{b})_{\hat{\mathcal{S}}^k}.
$
Differentiating the terms according to active and inactive sets, and noting that $\vectt{\delta}^k_{\mathcal{A}^k} = (\vectt{\varphi} - \vectt{u}^k)_{\mathcal{A}^k}$, we find that
\begin{align}\label{6.10}
A_{\hat{\mathcal{S}}^k, \mathcal{A}^k} (\vectt{\varphi} - \vectt{u}^k)_{\mathcal{A}^k} +  A_{\hat{\mathcal{S}}^k, \mathcal{I}^k} \vectt{\delta}^k_{\mathcal{I}^k} \geq - A_{\hat{\mathcal{S}}^k, \mathcal{A}^k}  \vectt{u}^k_{\mathcal{A}^k} - A_{\hat{\mathcal{S}}^k, \mathcal{I}^k}  \vectt{u}^k_{\mathcal{I}^k} + \vectt{b}_{\hat{\mathcal{S}}^k}.
\end{align}
By assumption $\varphi_h = \varphi \equiv \bar \varphi > 0$. Hence $\nabla \varphi_h = 0$ and, therefore, $A \vectt{\varphi} = \vectt{0}$ which implies that $A_{\hat{\mathcal{S}}^k, \mathcal{A}^k} \vectt{\varphi}_{\mathcal{A}^k} = \vectt{0}$. By \cref{eq:switch1} in \cref{lem:switching} the second terms in the left and right-hand side of \eqref{6.10}, respectively, equals the zero vector. Finally by moving the first term on the right-hand side to the left we conclude that
$
\vectt{b}_{\hat{\mathcal{S}}^k} \leq \vectt{0}.
$
However, by assumption $f \geq c_f > 0$ a.e.~in $\Omega$ and therefore $\vectt{b}_i = \int_\Omega f \phi_i \, \dx > 0$ for all $i \in \{1 \colto N\}$. But this contradicts $\vectt{b}_{\hat{\mathcal{S}}^k} \leq \vectt{0}$ and we conclude that $\hat{\mathcal{S}}^k = \emptyset$. 
\end{proof}

\begin{remark}[No peeling for thin obstacle and Signorini problems]
\label{rem:no-peeling}
In \Cref{sec:experiments}, we observed that layer-by-layer peeling does not occur when \cref{alg:hik} is applied to the thin obstacle and Signorini problems. This is because active set dofs on the codimension one obstacle always neighbour an ``inactive" set dof in the bulk. %
\end{remark}

\begin{remark}[Interpretation as $h \to 0$]
In the limit as $h \to 0$, an interpretation of \cref{th:peeling} is that the active set can only change by sets of codimension one. In \cref{prob:obstacle}, this means the active set can only change by sets of measure zero (yet possibly of positive capacity), and thus the solver stagnates. Whereas for \cref{prob:thin,prob:signorini}, where the active set has codimension one, progress is possible.
\end{remark}

\section{HIK in infinite dimensions}
\label{sec:ssn}

In this section, we attempt to generalize the HIK iteration to infinite dimensions for the obstacle and thin obstacle problems of \cref{prob:obstacle,prob:thin} when $d=2$. We recover convergence for \cref{prob:thin} but it is ultimately ill-posed for \cref{prob:obstacle}.

We remark that we are unable to construct an SSN in infinite dimensions, in the sense of \Cref{sec:pdas}, for \cref{prob:obstacle,prob:thin,prob:signorini}. Unlike the finite-dimensional setting \cref{pdas:1}, the equations \cref{eq:ob4,eq:signorini-nonsmooth} do not possess a bounded and invertible Newton derivative. The operator $(\cdot)_+ : L^q(\Omega) \to L^p(\Omega)$ is semismooth on $L^q(\Omega)$ provided that $1 \leq p < q \leq \infty$ \cite[Prop.~4.1]{hintermuller2002} where the norm gap  $p<q$ is essential cf.~\cite[Ex.~5.11]{ulbrich2002semismooth} and \cite[Sec.~4]{hintermuller2002}. However, any attempt at formulating \cref{eq:ob4,eq:signorini-nonsmooth} so that $(\cdot)_+ : L^q(\Omega) \to L^p(\Omega)$ for $p<q$, results in a Newton derivative that is not invertible. Hence, a Newton-like linearization of \cref{eq:ob4,eq:signorini-nonsmooth} would not lead to updates resulting in local superlinear convergence.

On the other hand, the HIK update as stated in \cref{eq:pdas-update} cleanly extends to infinite dimensions. Let $(-\Delta) \in \mathcal{L}(H^1_0(\Omega),H^{-1}(\Omega))$, $u^{k} \in H^1_0(\Omega)$, and $\lambda^k \in H^{-1}(\Omega)$. Suppose that $\varphi$, $u^k$ and $\lambda^k$ are sufficiently regular to be interpreted pointwise a.e.~on $\supp(\varphi)$ (e.g.~$\lambda \in L^p(\supp(\varphi))$ for some $p \geq 1$), and define
$\mathcal{A}^k = \{\lambda^k + u^k - \varphi > 0 \; \text{for a.e.~}  x \in \supp(\varphi)\}$ and $\mathcal{I}^k = \op{int}(\Omega \backslash \mathcal{A}^k)$. The HIK iterate is given by, find $(u^{k+1}, \lambda^{k+1}) \in H^1_0(\Omega) \times H^{-1}(\Omega)$ satisfying \cite[Sec.~4]{hintermuller2002}:
\begin{align}
(-\Delta) u^{k+1} + \lambda^{k+1} = f \;\; \text{in} \; H^{-1}(\Omega), \quad u^{k+1} = \varphi \; \text{on} \; \mathcal{A}^k, \quad \lambda^{k+1} = 0 \; \text{on} \; \mathcal{I}^k. \label{eq:inf:1}
\end{align}
For the (thin) obstacle problem, the iterate \cref{eq:inf:1} is equivalent to solving,
\begin{align}
(\nabla u^{k+1}, \nabla v) = (f, v) \; \forall  v \in H^1_0(\mathcal{I}^k), \;\; u|_{\partial \Omega} = 0, \;\; u|_{\mathcal{A}^k} = \varphi_{\mathcal{A}^k},
\label{eq:inf:2}
\end{align}
and setting $\langle \lambda^{k+1}, w \rangle = (f,w) -(\nabla u^{k+1},\nabla w)$ for all $w \in H^1_0(\Omega)$.

\subsection{HIK for \cref{prob:thin}}
The numerical evidence in \Cref{sec:experiments} suggests that the infinite-dimensional HIK solver converges when applied to the thin obstacle problem of \cref{prob:thin} but loses local superlinear convergence. 

\subsubsection*{Well-posedness} In this subsection we show that the HIK iteration applied to \cref{prob:thin} is well-posed when $d=2$.  For this purpose fix $u^0 = \lambda^0 = 0$, and $\Gamma = \supp(\varphi) = \{(x_1,x_2) \in \Omega : x_1 = 1/2\}$.  A calculation reveals that $u^1 = (-\Delta)^{-1} f$ with $u^1|_{\partial \Omega}  = 0$, $\lambda^1 = 0$ and $\mathcal{A}^1 = \{ u^1 > \varphi\equiv 1 \; \text{a.e.~on} \; \Gamma\}$. Note that $\mathcal{A}^1$ is a one-dimensional object.

The next HIK step is nontrivial. We observe that $\mathcal{I}^1$ is not a Lipschitz domain but corresponds to a unit square domain with a crack compactly embedded in $\Gamma$. Nevertheless, thanks to the Lax--Milgram theorem \cite[Th.~2.2.1.2]{grisvard2011}, the following problem
\begin{align}
(\nabla u^2, \nabla v) = (f,v) \; \forall v \in H^1_0(\mathcal{I}^1),\;\; \text{subject to} \; u^2|_{\partial \Omega} = 0, \;\; u^2|_{\mathcal{A}^1} = \varphi|_{\mathcal{A}^1} = 1,
\end{align}
has a unique solution $u^2 \in H^1(\Omega)$ which is visualized in \cref{fig:jumps}. Elliptic regularity guarantees that $u^2$ is smooth in the interior of $\mathcal{I}^1$, i.e.~away from $\partial \Omega$ and $\mathcal{A}^1$ \cite[Ch.~6, Th.~3]{Evans2010}. In fact, since we are considering a polygonal (albeit non-Lipschitz) two-dimensional domain, we may characterize the behaviour of $u^2$ further via a decomposition. Let $x_a, x_b \in \Gamma \subset \mathbb{R}^2$ denote the endpoints (crack tips) of $\mathcal{A}^1$. Then \cite[Th.~4.4.3.7]{grisvard2011},
\begin{align*}
u^2 = u_r  + u_a + u_b
\end{align*}
where $u_r \in H^2(\Omega)$, and for $j \in \{a,b\}$, $u_j = r_j^{1/2} \sin(\theta_j/2) \eta_j$ where $\eta_j$ are smooth functions and $(r_j, \theta_j) = (|x-x_j|, \operatorname{arctan}((x_2-[x_j]_2)/(x_1-[x_j]_1)))$ are the polar coordinates centred at $x_j \in \mathbb{R}^2$. Hence, the loss of regularity in $u^2$ is due to square-root singularities concentrated at the crack tips. A calculation reveals that $r_j^{1/2} \nabla^2 u_j \in L^2(\Omega)^{2 \times 2}$, $j \in \{a,b\}$. Hence, by H\"older's inequality, for any $1 \leq p < 2$,
\begin{align*}
\|\nabla^2 u_j  \|_{L^p(\Omega)} \leq \|r_j^{1/2} \nabla^2 u_j\|_{L^2(\Omega)} \| r_j^{-1/2} \|_{L^{2p/(2-p)}(\Omega)}.
\end{align*}
The right-hand side is bounded provided $2p/(2-p) < 4$ and thus if $1  \leq p < 4/3$. Hence, we conclude that $u^2 \in W^{2,p}(\Omega)$ for any $p \in [1,4/3)$.

Now by splitting $\Omega$ into $\Omega_- = \{ x_1 < 1/2\}$ and $\Omega_+ = \{ x_1 > 1/2\}$, we recover two open convex and bounded domains. Hence, since $u^2_\pm \in W^{2,p}(\Omega_\pm)$, then $\nabla u^2_\pm \in W^{1,p}(\Omega_\pm)^2$. Thus the boundary trace $\nabla u_\pm|_{\Gamma}$ is well-defined on either side of $\Gamma$ and, therefore, the jump in the gradient across $\Gamma$, denoted $\lsb \nabla u^2 \rsb$, is well-defined on $\Gamma$. Here we use $\lsb \nabla v \rsb \coloneqq \nabla v^- \cdot n^- + \nabla v^+ \cdot n^+$, and $n^\pm$ is the normal direction on either side of $\Gamma$. Moreover, $\lsb \nabla u^2 \rsb \in W^{1-1/p,p}(\Gamma) \subset L^{q}(\Gamma)$ for any $q \in [1,2p/(3-p))$ (i.e.~$1 \leq q < 8/5$).

Now, by an integration by parts, we see that for all $v \in H^1_0(\Omega)$
\begin{align}
\begin{split}
\langle \lambda^2, v\rangle = - \int_{\Gamma} \lsb \nabla u^2  \rsb v \, \ds = - \int_{\mathcal{A}^1} \lsb \nabla u^2  \rsb v \, \ds, 
\end{split}
\label{eq:lambda:2}
\end{align}
as $u^2$ has no jump across $\Gamma \backslash \mathcal{A}^1$. Hence, we may identify $\lambda^2 = -\lsb \nabla u^2  \rsb$ and conclude that $\lambda^2 \in L^{q}(\Gamma)$, $q \in [1,8/5)$, which may be interpreted pointwise a.e.~on $\Gamma$.
Hence, we define the next active set $\mathcal{A}^2 = \{ {\lambda}^2 + u^2|_{\Gamma} - \varphi > 0 \; \text{a.e.~in} \; \Gamma\} = \{ \lsb \nabla u^2 \rsb < 0 \; \text{a.e.~in} \; \Gamma\}$. This analysis holds true for any iterate $k \geq 2$, and thus the HIK iterates satisfy $(u^k, \lambda^k) \in W^{2,p}(\Omega) \times L^{q}(\Gamma)$, for any $p \in [1,4/3)$, $q \in [1,8/5)$ and, therefore, $\mathcal{A}^k = \{ \lsb \nabla u^k \rsb < 0 \; \text{a.e.~in} \; \Gamma\}$ is well-defined.

\begin{remark}[Well-posedness for Signorini]
When HIK is applied to the Signorini problem of \cref{prob:signorini} with $d=2$, the loss in regularity for the second iterate $u^2$ will be due to a switch from a natural to a Dirichlet boundary condition compactly embedded on the top edge of the domain where $\{x_2 = 1\}$. The singularity at this switch is precisely the same singularity at the crack tips of \cref{prob:thin} \cite[Ex.~5.5.4]{Brenner2008}, \cite[Th.~4.4.3.7]{grisvard2011}. We conjecture the Lagrange multiplier $\lambda^k$ lives in $L^q(\Gamma_C)$ for $q \in [1,8/5)$ for every iterate $k \geq 2$ resulting in a well-posed solver for \cref{prob:signorini}.
\end{remark}

\subsubsection*{Convergence as $h \to 0$} To show that progress is made with each HIK iteration for \cref{prob:thin}, we first show that, unless we are at the solution $\lambda^*$, $\lambda^k$ must have regions where it is negative. 
For the remainder of this subsection, we shall assume that $\Gamma$ is represented exactly by a straight chain of shared edges in all the meshes $\{\mathcal{T}_h\}_h$.
\begin{proposition}
\label{prop:4}
Suppose that $u^k \leq \varphi$ a.e.~on $\Gamma$. If $(u^k, \lambda^k) \neq (u^*,\lambda^*)$, then there exists a subset $E^k \subset \Gamma$ with nonzero one-dimensional Lebesgue measure such that $\lambda^k < 0$ a.e.~on $E^k$.
\end{proposition}
\begin{proof}
Suppose the converse holds. Then $(u^k, \lambda^k)$ satisfies stationarity $(-\Delta) u^k + \lambda^k = f$ in $H^{-1}(\Omega)$, primal feasibility $u^k \leq \varphi$ a.e.~on $\Gamma$ and dual feasibility $\lambda^k \geq 0$ a.e.~on $\Gamma$. Our last step is to check complementarity. We have that $\lambda^k = 0$ a.e.~on $\mathcal{I}^{k-1}$. Hence, on $\mathcal{I}^{k-1} \cap \mathcal{A}^k$, $u^k > \varphi$ a.e.~which contradicts the assumptions and thus $\mathcal{I}^{k-1} \cap \mathcal{A}^k = \emptyset$. Whereas on  $\mathcal{I}^{k-1} \cap \mathcal{I}^k$, we necessarily have that $u^k \leq \varphi$ a.e. Similarly, we have that $u^k = \varphi$ a.e.~on $\mathcal{A}^{k-1}$. Per the definitions, we therefore have that $\lambda^k > 0$ a.e.~on $\mathcal{A}^{k-1} \cap \mathcal{A}^k$ and $\lambda^k = 0$ a.e.~on $\mathcal{A}^{k-1} \cap \mathcal{I}^k$. Combining these four cases, we conclude that $\lambda^k (u^k - \varphi) = 0$ a.e.~on $\Gamma$. Hence, complementarity holds. Thus, $(u^k, \lambda^k)$ satisfies the first-order optimality conditions of \cref{eq:setup:scalar} and, therefore, is the solution. But $(u^k, \lambda^k) \neq (u^*,\lambda^*)$, a contradiction.
\end{proof}
By the definition of $\mathcal{A}^k$, such a set $E^k \subset \mathcal{I}^k$ and $E^k \cap \mathcal{A}^k = \emptyset$. Hence, $\lambda^k$ is partially supported in $\mathcal{I}^k$ and
\begin{align*}
\| \lambda^k \|_{H^{-1}(\mathcal{I}^k)} = \sup_{v \in H^1_0(\mathcal{I}^k)} \frac{|\langle \lambda^k, v\rangle|}{\| \nabla v\|_{L^2(\Omega)}} \geq c^k > 0 \;\; \text{for some constant $c^k$.}
\end{align*}

Recall, from \cref{th:convergence-rate}, that the convergence of HIK as $h \to 0$ is determined by the behaviour of $\| (\vectt{\lambda}^k_h)_{\mathcal{I}^k} \|_{(A_{\mathcal{I}^k,\mathcal{I}^k})^{-1}}$. We first show that $\| (\vectt{\lambda}_h^k)_{\mathcal{I}^k} \|_{(A_{\mathcal{I}^k_h,\mathcal{I}^k_h})^{-1}}$ can be rewritten in a more familiar notation.

\begin{lemma}
\label{lem:dual-norm}
Define the functional $\lambda^{k}_h$ via its coefficient vector as $\langle \lambda^{k}_h, v_h \rangle_{U_h^*, U_h} = (\vectt{\lambda}^{k}_h)^\top \vectt{v}_h$. Let $U_h(\mathcal{I}^k_h) \coloneqq \{ v_h \in U_h : [\vectt{v}_h]_i = 0 \; \forall \; i \in \mathcal{A}^k_h\}$. Then
\begin{align*}
\|\vectt{\lambda}^{k}_h\|_{A^{-1}} &= \| \lambda^k_h\|_{U_h^*} \coloneqq \sup_{v_h \in U_h} \frac{|\langle \lambda^k_h, v_h \rangle|}{\| \nabla v_h\|_{L^2(\Omega)}},\\
\| [\vectt{\lambda}^k_h]_{\mathcal{I}^k_h} \|_{(A_{\mathcal{I}^k_h,\mathcal{I}^k_h})^{-1}} &= \| \lambda^k_h\|_{U_h^*(\mathcal{I}^k_h)} \coloneqq \sup_{v_h \in U_h(\mathcal{I}^k_h)} \frac{|\langle \lambda^k_h, v_h \rangle|}{\| \nabla v_h\|_{L^2(\Omega)}}.
\end{align*}
\end{lemma}
\begin{proof}
The result follows thanks to the Riesz representation theorem.
\end{proof}
\cref{lem:dual-norm} suggests that $\| [\vectt{\lambda}^k_h]_{\mathcal{I}^k_h} \|_{(A_{\mathcal{I}^k_h,\mathcal{I}^k_h})^{-1}} \to \| \lambda^k \|_{H^{-1}(\mathcal{I}^k)}$ as $h \to 0$. We now prove this assertion via two lemmas and a proposition.

\begin{lemma}
\label{lem:1}
The following bound holds:
\begin{align}
\left| \| \lambda^k_h\|_{U^*_h(\mathcal{I}^k_h)} - \|\lambda^k\|_{U^*_h(\mathcal{I}^k_h)}\right| \leq \| \nabla(u^k-u^k_h)\|_{L^2(\Omega)}.
\end{align}
\end{lemma}
\begin{proof}
By definition
\begin{align}
\left|\| \lambda^k_h \|_{U_h^*(\mathcal{I}^k_h)} - \| \lambda^k \|_{U_h^*(\mathcal{I}^k_h)} \right|
\leq \sup_{v_h \in U_h(\mathcal{I}^k_h)} \frac{|\langle \lambda^k_h-\lambda^k, v_h \rangle|}{\| \nabla v_h\|_{L^2(\Omega)}} \leq \|\lambda^k_h - \lambda^k \|_{U_h^*}.
\label{eq:hlim1}
\end{align}
It follows from \cref{lem:linear}\labelcref{lem:linear:i} and \cref{eq:inf:1} that, for all $v_h \in U_h$,
\begin{align*}
\langle \lambda^k - \lambda^k_h, v_h \rangle = (\nabla(u^k_h - u^k), \nabla v_h)_{L^2(\Omega)} \leq \|\nabla (u^k- u^k_h)\|_{L^2(\Omega)} \|\nabla v_h\|_{L^2(\Omega)}. 
\end{align*}
Hence,
\begin{align}
\| \lambda^k - \lambda^k_h \|_{U_h^*} \leq \|\nabla (u^k-u^k_h)\|_{L^2(\Omega)}.
\label{eq:hlim2}
\end{align}
By combining \cref{eq:hlim1,eq:hlim2}, the assertion follows. 
\end{proof}

As there is no potential for confusion, for the remainder of this section, we overload the notation ${\mathcal{I}_h^k} \coloneqq \bigcup_{v_h \in U_h(\mathcal{I}_h^k)} \operatorname{supp}(v_h)$ and $\mathcal{A}^k_h = \Omega \backslash \mathcal{I}^k_h$, i.e.~for \cref{prob:thin}, $\mathcal{A}^k_h = \{ \lsb \nabla u^k_h \rsb < 0 \; \text{on} \; \Gamma \}$ and $\mathcal{I}^k_h = \Omega \backslash \mathcal{A}^k_h$.
\begin{lemma}
\label{lem:2}
Suppose that $u^k_h \to u^k$ strongly in $H^1(\Omega)$ as $h \to 0$. Then, \newline $|\mathcal{A}^k \backslash \mathcal{A}^k_h| \to 0$.
\end{lemma}
\begin{proof}
We will prove the stronger result that  $\lsb \nabla u^k_h \rsb \to \lsb \nabla u^k \rsb$ a.e.~on $\Gamma$. Recall the splitting of $\Omega$ into $\Omega_- = \{ x_1 < 1/2\}$ and $\Omega_+ = \{ x_1 > 1/2\}$. Consider the restriction $\nabla u^k_\pm$ of $\nabla u^k$ to $\Omega_\pm$. Then, for any $p<4/3$,
\begin{align*}
\| \nabla (u^k_\pm - (u^k_h)_\pm)\|_{L^p(\Gamma)} \leq \| \nabla (u^k_\pm - (I_h u^k)_\pm)\|_{L^p(\Gamma)} + \| \nabla( (I_h u^k)_\pm - u^k_\pm)\|_{L^p(\Gamma)},
\end{align*}
where $I_h$ is a suitable linear interpolant. Denote the first term on the right-hand side by (A) and the second by (B). Now, dropping the subscript $_\pm$ for clarity:
\begin{align*}
\text{(A)} &= \sum_{e \in \Gamma}  \| \nabla (u^k - I_h u^k)\|_{L^p(e)} \\
& \leq C \sum_{K \in \mathcal{T}_h} \left( h^{-1/p}  \| \nabla (u^k - I_h u^k)\|_{L^p(K)} + h^{1-1/p} | \nabla u^k |_{W^{1,p}(K)} \right)\\
& \leq C h^{1-1/p} | \nabla u^k |_{W^{1,p}(\Omega)},
\end{align*}
where the first inequality follows from a multiplicative trace inequality \cite[Rem.~12.19]{ern2021:i} and the second inequality follows by a standard interpolant approximation result \cite[Ch.~4.4]{Brenner2008}. Hence (A)$\to 0$ as $h \to 0$. Moreover,
\begin{align*}
\text{(B)} \leq C  h^{-1/p}  \| \nabla (I_h u^k - u^k_h)\|_{L^p(\Omega)} \leq  C h^{1-1/p} | \nabla u^k |_{W^{1,p}(\Omega)}
\end{align*}
where the first inequality follows via a discrete trace inequality \cite[Lem.~12.8]{ern2021:i} and the second by $W^{1,p}$-estimates  for the Poisson problem on polyhedral domains \cite[Th.~8.5.3]{Brenner2008}. Hence, (B)$\to 0$. Therefore, on either side of $\Gamma$, we have that the trace of $\nabla u^k_h$ converges strongly to $\nabla u^k$ in $L^p(\Gamma)$, and we conclude that $\lsb \nabla u^k_h \rsb$ converges pointwise a.e.~to $\lsb \nabla u^k \rsb$.

Recall that $\mathcal{A}^k = \{ \lsb \nabla u^k \rsb < 0 \; \text{on} \; \Gamma\}$ and $\mathcal{A}^k_h = \{ \lsb \nabla u^k_h \rsb < 0 \; \text{on} \; \Gamma\}$. Suppose now that $|\mathcal{A}^k \backslash \mathcal{A}^k_h| \not\to 0$ as $h\downarrow 0$. Then there exists a subsequence $\{h'\}$ with $h'\downarrow 0$ and $\eta>0$ such that $|\mathcal{A}^k\setminus\mathcal{A}^k_{h'}|\geq\eta$ for all $h'>0$.
Let $\mathcal{D}_{h'}^k \coloneqq \mathcal{A}^k\setminus\mathcal{A}^k_{h'}$, and define, for $\epsilon>0$,
$$
\mathcal{D}_{h'}^{k,\epsilon}\coloneqq \{s\in\mathcal{D}_{h'}^k:\lsb\nabla u^k\rsb(s)\leq-\epsilon\},\quad\text{and}\quad
\mathcal{E}_{h'}^{k,\epsilon}\coloneqq\{s\in\mathcal{D}_{h'}^k:-\epsilon<\lsb \nabla u^k\rsb(s)<0\}.
$$
Then, it holds that $\mathcal{D}_{h'}^{k,\epsilon}\cap \mathcal{E}_{h'}^{k,\epsilon}=\emptyset$ and $\mathcal{D}_{h'}^{k,\epsilon}\cup \mathcal{E}_{h'}^{k,\epsilon}= \mathcal{D}_{h'}^{k}$. Hence, we have that (here with $|\cdot|$ the $(d-1)$-dimensional Lebesgue measure)
\begin{equation}\label{MH.setconvergence}
0<\eta\leq|\mathcal{D}_{h'}^k|=|\mathcal{D}_{h'}^{k,\epsilon}|+|\mathcal{E}_{h'}^{k,\epsilon}|.
\end{equation}
Regarding the first term of the sum of the right hand side in \eqref{MH.setconvergence} we observe that	
\begin{equation}\label{MH.D}
|\mathcal{D}_{h'}^{k,\epsilon}|=\int_{\mathcal{D}_{h'}^{k,\epsilon}}1\,\ds\leq\epsilon^{-1}\int_{\mathcal{D}_{h'}^{k,\epsilon}}\left|\lsb\nabla u^k\rsb\right|\, \ds\leq\epsilon^{-1}\int_{\mathcal{D}_{h'}^{k,\epsilon}}\left|\lsb\nabla u^k\rsb - \lsb\nabla u_{h'}^k\rsb\right|\, \ds\to 0 
\end{equation}
as $h'\downarrow 0$. Note that we use here $\lsb\nabla u_{h'}^k\rsb\geq 0$ on $\mathcal{D}_{h'}^k$, $\mathcal{D}_{h'}^k\subseteq\Gamma$, and $\lsb\nabla u_{h'}^k\rsb\to\lsb\nabla u^k\rsb$ in $L^p(
\Gamma)$ for $1\leq p <4/3$.

Concerning ${\mathcal{E}_{h'}^{k,\epsilon}}$ we have
\begin{equation}\label{MH.D2}
{\mathcal{E}_{h'}^{k,\epsilon}}\subset {\mathcal{E}^{k,\epsilon}}\coloneqq\{s\in\Gamma:-\epsilon<\lsb\nabla u^k\rsb(s)<0\},
\end{equation}
and it holds that
${\mathcal{E}^{k,\epsilon_2}}\subseteq {\mathcal{E}^{k,\epsilon_1}}$ for $0<\epsilon_2<\epsilon_1$, and $\bigcap_{\epsilon>0}{\mathcal{E}^{k,\epsilon}}=\emptyset$. Hence, we have that
$
\lim_{\epsilon\downarrow 0}|\mathcal{E}^{k,\epsilon}|=0.
$
In other words, for $\eta>0$ there exists $\epsilon(\eta)>0$ such that
\begin{equation}\label{MH.D4}
|\mathcal{E}^{k,\epsilon}|\leq\frac{\eta}{4},\quad\text{and hence}\quad |\mathcal{E}^{k,\epsilon}_{h'}|\leq\frac{\eta}{4},\quad\text{for all }0<\epsilon < \epsilon(\eta)\text{ and }h'>0.
\end{equation}
Let $\epsilon\in(0,\epsilon(\eta))$ be fixed. Then, due to \eqref{MH.D}, there exists an $h'(\eta,\epsilon)>0$ such that
\begin{equation}\label{MH.D5}
|\mathcal{D}^{k,\epsilon}_{h'}|\leq \frac{\eta}{4},\quad\text{for all }0<h'<h'(\eta,\epsilon).
\end{equation}
When combining \eqref{MH.D4} and \eqref{MH.D5}, then \eqref{MH.setconvergence} yields
$$
0<\eta\leq|\mathcal{D}_{h'}^{k}|\leq\frac{\eta}{2},\quad\text{for all}\quad 0<\epsilon<\epsilon(\eta)\text{ and } 0<h'<h'(\eta,\epsilon),
$$
which is a contradiction. Hence, we must have that
$
\lim_{h\downarrow 0}|\mathcal{D}_{h}^k|=\lim_{h\downarrow 0}|\mathcal{A}^k\setminus\mathcal{A}_h^k|=0,
$
which ends the proof.
\end{proof}

\begin{proposition}
\label{prop:3}
Suppose that $u^k_h \to u^k$ strongly in $H^1(\Omega)$ as $h \to 0$. Then, $\| \lambda^k_h\|_{U^*_h(\mathcal{I}^k_h)} \to \| \lambda^k \|_{H^{-1}(\mathcal{I}^k)}$.
\end{proposition}
\begin{proof}
From \cref{lem:1}, we may equivalently show that $\| \lambda^k\|_{U^*_h(\mathcal{I}^k_h)} \to \| \lambda^k \|_{H^{-1}(\mathcal{I}^k)}$ as $h \to 0$. Let $w \in H^1_0(\mathcal{I}^k)$ and $w_h \in U_h(\mathcal{I}^k_h)$ denote the solutions to the variational problems:
\begin{align*}
(\nabla w,\nabla v) = \langle \lambda^k, v \rangle \; \forall v \in H^1_0(\mathcal{I}^k) \quad \text{and} \quad
(\nabla w_h,\nabla v_h) = \langle \lambda^k, v_h \rangle \; \forall  v_h \in U_h(\mathcal{I}^k_h).
\end{align*}
We may extend any function $v \in U_h(\mathcal{I}^k_h)$ or $v \in H^1_0(\mathcal{I}^k)$ by zero to all of $\Omega$. Thus $U_h(\mathcal{I}^k_h),  H^1_0(\mathcal{I}^k) \subset H^1_0(\Omega)$ but $U_h(\mathcal{I}^k_h) \not\subset H^1_0(\mathcal{I}^k)$. This is a typical situation when committing a variational crime \cite[Ch.~10]{Brenner2008}. Applying Strang's lemma \cite[Lem.~10.1.7]{Brenner2008}:
\begin{align}
\begin{split}
\label{eq:strang1}
&\| \nabla(w - w_h)\|_{L^2(\Omega)} \\
&\indent \leq \inf_{v_h \in U_h(\mathcal{I}^k_h)} \|\nabla(w - v_h)\|_{L^2(\Omega)} + \sup_{y_h \in U_h(\mathcal{I}^k_h)} \frac{|(\nabla (w - w_h), \nabla y_h)|}{\|\nabla y_h\|_{L^2(\Omega)}}.
\end{split}
\end{align}
Now, $\mathcal{I}^k$ and $\mathcal{I}^k_h$ only differ from $\Omega$ on $\Gamma$ and $y_h = 0$ on $\mathcal{A}^k_h$. Thus
\begin{align*}
(\nabla (w - w_h), \nabla y_h) = \int_\Omega ((-\Delta) w) y_h - f y_h\, \dx + \int_{\mathcal{A}^k \backslash \mathcal{A}^k_h} \lsb \nabla w \rsb y_h \ds.
\end{align*}
The first term equals zero and the second term satisfies, for a constant $C>0$, $5/4 < q < 8/5$, $15/13 < p < 4/3$,
\begin{align*}
\left| \int_{\mathcal{A}^k \backslash \mathcal{A}^k_h} \lsb \nabla w \rsb y_h \ds \right| 
&\leq \| \lsb \nabla w \rsb \|_{L^q(\Gamma)}   \| y_h \|_{L^{\frac{q}{q-1}}(\mathcal{A}^k \backslash \mathcal{A}^k_h)}\\
& \leq  C\|w \|_{W^{2,p}(\Omega)}   \| y_h \|_{L^4(\Gamma)} \cdot  |\mathcal{A}^k \backslash \mathcal{A}^k_h|^{\frac{4q-5}{4q}}\\
& \leq  C\|w \|_{W^{2,p}(\Omega)}   \|  \nabla y_h \|_{L^2(\Omega)} \cdot  |\mathcal{A}^k \backslash \mathcal{A}^k_h|^{\frac{4q-5}{4q}}.
\end{align*}
An application of H\"older's inequality gives the first inequality. The second inequality follows by leveraging the trace estimates discussed earlier in this section for the first term. H\"older's inequality is again used for the second term. Note that $\| w \|_{W^{2,p}(\Omega)}$ is finite for $p<4/3$ as discussed above. Applying a Sobolev inequality $\| y_h \|_{L^4(\Gamma)} \leq C \| y_h \|_{H^{1/2}(\Gamma)}$ followed by a trace estimate and then the Poincar\'e inequality gives us the third inequality.

Hence, since $| \mathcal{A}^k \backslash \mathcal{A}^k_h| \to 0$ as $h\to 0$ (\cref{lem:2}), we have that the second term in \cref{eq:strang1} converges to zero. Again, by applying \cref{lem:2}, one shows the first term also converges to zero and therefore, $\| \nabla w_h \|_{L^2(\Omega)} \to \| \nabla w\|_{L^2(\Omega)}$. Moreover, by the Riesz representation theorem, $\| \nabla w \|_{L^2(\Omega)} = \| \lambda^k \|_{H^{-1}(\mathcal{I}^k)}$ and $\| \nabla w_h \|_{L^2(\Omega)} = \| \lambda^k\|_{U^*_h(\mathcal{I}^k_h)}$. The result follows.
\end{proof}
\cref{prop:3} tells us that the convergence of HIK as $h \to 0$ is determined by the behaviour of $\lambda^k$ in $\mathcal{I}^k$. 

We now prove the following global infinite-dimensional convergence theorem.
\begin{theorem}
\label{th:convergence-rate-inf}
Consider an HIK iterate $k \geq 1$ for \cref{eq:setup:scalar} such that $\mathcal{A}^* \subseteq \mathcal{A}^k$ and, consequently, $\mathcal{I}^k \subseteq \mathcal{I}^*$. Moreover, suppose that $(u^k,\lambda^k) \neq (u^*,\lambda^*)$ and $u^k \leq \varphi$ a.e.~on $\Gamma$. Then
\begin{align}
\| \nabla (u^{k+1} - u^*) \|^2_{L^2(\Omega)} - \| \nabla( u^{k} - u^*) \|^2_{L^2(\Omega)} + \| {\lambda}^k \|^2_{H^{-1}(\mathcal{I}^k)} =0.
\label{eq:contraction:inf:1}
\end{align}
Therefore, HIK converges at the following guaranteed sublinear rate:
\begin{align}
\frac{\| \nabla (u^{k+1} - u^*) \|_{L^2(\Omega)} }{{\| \nabla (u^{k} - u^*) \|_{L^2(\Omega)} }}
=
\frac{\| {\lambda}^{k+1} - {\lambda}^* \|_{H^{-1}(\Omega)}}{{\| {\lambda}^{k} - {\lambda}^* \|_{H^{-1}(\Omega)}}} 
= \left(1-\frac{\| {\lambda}^k \|^2_{H^{-1}(\mathcal{I}^k)}}{\| {\lambda}^{k} - {\lambda}^* \|_{H^{-1}(\Omega)}^2}\right)^{1/2}.
\label{eq:contraction:inf:2}
\end{align}
\end{theorem}
\begin{proof}
Note that $u^*_h \to u^*$ strongly in $H^1(\Omega)$ \cite[Th.~2]{falk1974}. Via a proof-by-induction argument, suppose that $u^k_h \to u^k$ strongly in $H^1(\Omega)$. Then,
\begin{align*}
\| \vectt{u}_h^{k} - \vectt{u}_h^* \|_A  = \| \nabla (u^{k}_h - u^*_h) \|_{L^2(\Omega)} \to \| \nabla (u^{k} - u^*) \|_{L^2(\Omega)}.
\end{align*}
By \cref{lem:2,prop:3}, it follows that $\| (\vectt{\lambda}^k_h)_{\mathcal{I}^k} \|_{(A_{\mathcal{I}^k,\mathcal{I}^k})^{-1}} = \| \lambda^k_h\|_{U^*_h(\mathcal{I}^k_h)} \to \| \lambda^k \|_{H^{-1}(\mathcal{I}^k)}$ as $h \to 0$. A similar argument,  via Strang's lemma, as in the proof of \cref{prop:3} reveals that $\| \nabla u^{k+1}_h \|_{L^2(\Omega)} \to \| \nabla u^{k+1}\|_{L^2(\Omega)}$ and thus $\| \vectt{u}_h^{k} - \vectt{u}_h^* \|_A \to \| \nabla (u^{k+1} - u^*) \|_{L^2(\Omega)}$. Applying these limits to \cref{eq:convergence-recurrence} in \cref{th:convergence-rate} leads us to \cref{eq:contraction:inf:1}.

At $k=1$, $u^1$ is the solution of the Poisson problem on a convex domain and thus $u^1_h \to u^1$ strongly in $H^1(\Omega)$ by a standard FEM convergence result. Hence, by induction, \cref{eq:contraction:inf:1} holds for all $k \geq 1$.

Dividing \cref{eq:contraction:inf:1} by $\|\nabla(u^k-u^*)\|_{L^2(\Omega)}$, noting $\|\nabla(u^k-u^*)\|_{L^2(\Omega)} = \| {\lambda}^{k+1} - {\lambda}^* \|_{H^{-1}(\Omega)}$, and rearranging the equation gives \cref{eq:contraction:inf:2}. \cref{prop:4} guarantees that $\| \lambda^k\|_{H^{-1}(\mathcal{I}^k)}> 0$. Moreover,  $\| \lambda^k\|_{H^{-1}(\mathcal{I}^k)} < \| \lambda^k- \lambda^*\|_{H^{-1}(\Omega)}$. Together this ensures sublinear convergence.
\end{proof}
Unless a solution has been found in iteration $k$, \cref{th:convergence-rate-inf} tells us that for sufficiently small $h>0$, $\| [\vectt{\lambda}^k_h]_{\mathcal{I}^k_h} \|_{(A_{\mathcal{I}^k_h,\mathcal{I}^k_h})^{-1}}$ is uniformly (in $h$) bounded below by a positive constant. A uniform bound is the first step to proving mesh-independent convergence.

\begin{remark}
\cref{fig:convergence} suggests that the convergence is in fact linear. This would imply that $\| \lambda^k\|_{H^{-1}(\mathcal{I}^k)} / \| \lambda^k- \lambda^*\|_{H^{-1}(\Omega)} \geq c > 0$ for some constant $c\in (0,1)$ independent of $k$. Proving such a result remains an open problem.
\end{remark}

\begin{figure}[t]
\centering
\includegraphics[width = \textwidth]{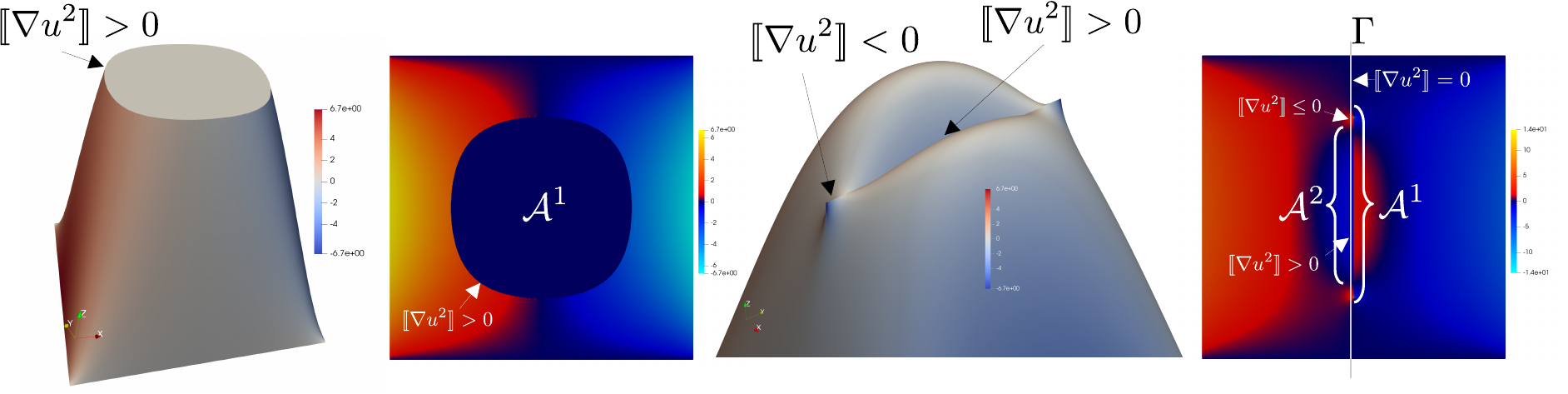}
\caption{A contour and surface plot of the second HIK iterate $u^2$ in the obstacle problem of \cref{prob:obstacle} (left and left middle) and the thin obstacle problem of \cref{prob:thin} (right middle and right), respectively. The colormaps are associated with the value of $\lsb u^2 \rsb$. The active set $\mathcal{A}^1$ and the sign of the jump $\lsb \nabla u^2 \rsb$ are labelled.}
\label{fig:jumps}
\end{figure}

\subsection{HIK for \cref{prob:obstacle}}
In this subsection, we show that the HIK solver diverges after the second iteration when applied to the obstacle problem of \cref{prob:obstacle}.

\subsubsection*{Divergence}
The first HIK iterate is well-posed since $u^0 = \lambda^0 = 0$, $\mathcal{A}^0 = \emptyset$ and thus $u^1 = (-\Delta)^{-1} f$ in $\Omega$ with $u^1 = 0$ on $\partial \Omega$ and $\lambda^1 = 0$. Next $\mathcal{A}^1 = \{ \lambda^1 + u^1 - \varphi > 0\} =  \{u^1 > \varphi\}$ and $\mathcal{I}^1 = \op{int}(\Omega \backslash \mathcal{A}^1)$. As before, elliptic regularity ensures that $u^1$ is smooth in the interior of $\Omega$ \cite[Ch.~6, Th.~3]{Evans2010}. Moreover, the boundary $\partial \mathcal{A}^1$ is the level set $\{u^1=1\}$ such that $\nabla u^1 \neq 0$ on $\partial \mathcal{A}^1$. Thus $\mathcal{A}^1$, is an open connected subdomain of $\Omega$ with a smooth boundary and $\mathcal{I}^1$ is an open, connected bounded Lipschitz domain.
A calculation reveals that
\begin{align}
u^2(x) = \begin{cases}
1 & \text{if} \; x \in \mathcal{A}^1,\\
u^1(x) & \text{otherwise}.
\end{cases}
\end{align}
Since $u^2$ is piecewise smooth in $\Omega \backslash \mathcal{A}^1$ and $\mathcal{A}^1$, the jump in the normal direction $\lsb \nabla u^2 \rsb$ across $\partial \mathcal{A}^1$ is well-defined. In fact $\lsb \nabla u^2 \rsb$ is uniformly bounded on $\partial \mathcal{A}^1$ and positive. We visualize $u^2$ and $\lsb \nabla u^2 \rsb$ in \cref{fig:jumps}.  Now, for all $v \in H^1_0(\Omega)$, an integration by parts reveals that
\begin{align}
\begin{split}
\langle \lambda^2, v\rangle
&= \int_{\mathcal{A}^1} f v \, \dx - \int_{\partial \mathcal{A}^1} \lsb \nabla u^2  \rsb v \, \ds.
\end{split}
\label{eq:lambda:1}
\end{align}
The Lagrange multiplier iterate $\lambda^2$ is supported on $\overline{\mathcal{A}^1}$ and can be rewritten as $\lambda^2 = f|_{\mathcal{A}^1} - \lsb \nabla u^2 \rsb \delta_{\partial \mathcal{A}^1}$ where $\delta_{\partial \mathcal{A}^1}$ is the Dirac delta measure of codimension one. Hence, in stark contrast to the thin obstacle problem, it is here that a problem arises. Indeed, $\lambda^2$ cannot be interpreted pointwise a.e.~on $\Omega$ and thus the set $\mathcal{A}^2 = \{ \lambda^2 + u^2 - \varphi > 0 \; \text{a.e.~in} \; \Omega\}$ is ill-defined. Dropping the singular component, and defining $\mathcal{A}^2 = \{  f|_{\mathcal{A}^1} + u^2 - \varphi > 0 \; \text{a.e.~in} \; \Omega\}$, results in $\mathcal{A}^2 = \mathcal{A}^1$. Changes in the active set are driven by regions where the Lagrange multiplier iterate is negative. The ``negative" component of $\lambda^2$ is confined to a set of measure zero (yet of positive capacity) resulting in no measurable change (in the sense of the full Lebesgue measure). We conclude that it is unclear how to construct a well-posed HIK iterate, when $k \geq 2$, for the infinite-dimensional obstacle problem.

\subsubsection*{Deterioration as $h \to 0$}

From \cref{th:convergence-rate}, the behaviour of HIK at the second iterate as $h \to 0$ is determined by $\| \lambda^2_h \|_{U^*_h(\mathcal{I}^2_h)}$. \cref{th:peeling} tells us that $\mathcal{I}^2_h$ can only be a perturbation of $h$ away from $\mathcal{I}^1_h$. Hence, assuming that $\mathcal{I}^1_h$ converges to $\mathcal{I}^1$ in a suitable sense, we may expect $\mathcal{I}^2_h$ to also converge to $\mathcal{I}^1$. 

Examining the infinite-dimensional Lagrange multiplier iterate:
\begin{align*}
\| \lambda^2 \|_{H^{-1}(\mathcal{I}^1)} = \sup_{v \in H^1_0(\mathcal{I}^1)} \frac{\left|\int_{\mathcal{A}^1} f v \, \dx - \int_{\partial \mathcal{A}^1 \cap \partial \mathcal{I}^1}	\lsb \nabla u^2 \rsb v \, \ds \right|}{\| \nabla v \|_{L^2(\Omega)}} = 0,
\end{align*}
where $v$ is extended by zero from $\mathcal{I}^1$ to $\Omega$. The second equality follows since $v = 0$ in $\mathcal{A}^1$ and its trace is zero on $\partial \mathcal{I}^1$.  Assuming that $\| \lambda^2_h \|_{U^*_h(\mathcal{I}^2_h)} \to \| \lambda^2 \|_{H^{-1}(\mathcal{I}^1)} = 0$ we conclude that the convergence rate $\rho$ in \cref{th:convergence-rate} deteriorates to one.

From \cref{lem:1}, $\| \lambda^2_h \|_{U^*_h(\mathcal{I}^2_h)}$ is bounded by $\| \nabla (u^2-u^2_h) \|_{L^2(\Omega)}$ and $\| \lambda^2\|_{U^*_h(\mathcal{I}^2_h)}$. Note that $u^2$ has a jump in its gradient across $\partial \mathcal{A}^1$ and thus its regularity is capped, $u^2 \in H^s(\Omega)$ with $s < 3/2$ (since a function $v$ cannot jump if $v \in H^{1/2}(\Omega)$). Assuming that the decay of $\| \nabla (u^2-u^2_h) \|_{L^2(\Omega)}$ is the bottleneck, then we conjecture that $\| \lambda^2_h \|_{U^*_h(\mathcal{I}^2_h)} \to 0$  at a rate of $O(h^{1/2-\epsilon})$ for any $\epsilon > 0$. This conjecture is supported by the convergence plot in \cref{fig:convergence}.

\section{Conclusions}
\label{sec:conclusions}

Primal-dual active set strategies, co-developed with semismooth Newton methods, are one of the cornerstones to fast local convergence for problems with pointwise inequality constraints. Yet, despite their close relationship with nonsmooth Newton methods, their validity on the infinite-dimensional level is extraordinarily subtle.
As demonstrated in the numerical experiments of \Cref{sec:experiments} (\cref{tab:mesh-dependence}) the convergence behaviour degrades when applied to discretized obstacle and Signorini problems as the mesh size $h$ decreases. In turn this leads to growing mesh-dependent iteration counts. For obstacle-type problems, this growth is exponential, asymptotically doubling with each uniform refinement. In contrast the growth appears linear for Signorini-type problems.

In \Cref{sec:peeling}, we prove that, for the obstacle problem, degrees of freedom can only peel away (or deactivate) from the obstacle one layer of dofs at a time leading to exponential iteration count growth. Strikingly, peeling does not occur for Signorini problems because the active set is a codimension one object that always neighbours inactive degrees of freedom.

Along the way to the infinite-dimensional perspective in \Cref{sec:ssn}, we showed, in \cref{th:convergence-rate}, that the convergence rate of the HIK primal-dual active set strategy is dictated by the dual feasibility violation of the current Lagrange multiplier iterate in the interior of the inactive set. For obstacle-type problems, the Lagrange multiplier iterate vanishes in this region as $h\to 0$. Hence the convergence rate deteriorates to one and the HIK solver diverges at the second iterate. This occurs since the infinite-dimensional Lagrange multiplier iterate has no pointwise a.e.~interpretation and the next active set is ill-defined. Whereas for Signorini problems and obstacle problems with a codimension-one obstacle, we showed that the Lagrange multiplier iterate is more regular. This ensures subsequent active sets are well-defined and the solver does not stagnate even on the infinite-dimensional level. Subsequently in \cref{th:convergence-rate-inf}, we prove that HIK converges at a sublinear rate. However, its equivalence to a semismooth Newton method is lost and the region of local superlinear convergence becomes vanishingly smaller as the mesh size goes to zero.

The analysis considered here opens the possibility for developing solvers that combine the favourable properties of a PDAS with mesh-independent iteration counts. For instance, understanding the peeling effect may help guide the precise refinement of grids during the application of a full approximation scheme \cite{bueler2024}. Alternatively, the convergence rate results in \cref{th:convergence-rate,th:convergence-rate-inf} may help design an adaptive regularization that maximizes the size of a regularized Lagrange multiplier on the inactive set to accelerate convergence. The broader lesson is that algorithms should be analysed at the level of the original problem, not only after discretization, since otherwise unexpected degradation in convergence may occur.

\section*{Acknowledgements}
IP and MH acknowledge the funding support by the Deut\-sche Forschungsgemeinschaft (DFG, German Research Foundation) under Germany's Excellence Strategy -- The Berlin Mathematics Research Center MATH+ (EXC-2046/1, project ID: 390685689), and MH acknowledges in addition subproject AA5-2 “Robust Multilevel Training of Artificial Neural Networks.”).  IP would like to thank Timon Gutleb for helpful comments on the initial draft. %

\appendix
\section{Further numerical experiments}
\label{sec:further-experiments}

We supplement our investigations by a numerical study of different scenarios: (i) Lifting of the multiplier into the discrete primal space; (ii) applying a grid sequencing scheme by starting on a coarse mesh and interpolating the solution onto the next finer mesh as a starting point (i.e., we implement a {\it full approximation scheme} following terminology from multigrid methods); (iii) solving problems with biactivities, i.e., problems where on a relevant subset of $\{u^*=\varphi\}$ it holds $\lambda^*=0$; (iv) optimal control problems with pointwise constraints on the control (known for applicability of SSN in function space and mesh independence) or pointwise state constraints (SSN in function space not applicable).
We record the respective iteration counts in \cref{tab:further}. 

\subsection{Primal-dual correction}
\label{sec:further-experiments:primal-dual}
There appears to be a primal-dual mismatch in \cref{alg:hik} where the discretized Lagrange multiplier is an element of the dual space in \cref{pdas:1a}, but seemingly in the primal space in \cref{pdas:1b}. One may correct this by considering a P1-FEM discretization for the Lagrange multiplier. After following a Newton linearization (and not performing a mass lumping), one finds they must solve
\begin{align}
\begin{pmatrix}
A & M \\
-E^{\mathcal{A}^k_h} & E^{\mathcal{I}^k_h}
\end{pmatrix}
\begin{pmatrix}
\vectt{\delta}^k\\
\vectt{\varepsilon}^k
\end{pmatrix}
=
-
\begin{pmatrix}
A \vectt{u}^k + M\vectt{\lambda}^k - \vectt{b}\\
\vectt{g}^k
\end{pmatrix}
\label{eq:primal-dual-step}
\end{align}
at each iteration. The active $\mathcal{A}^k_h$ and inactive $\mathcal{I}^k_h$ retain the same definitions. As the P1-FEM mass matrix $M$ is not diagonal, we cannot as easily decompose the solve and, at each iteration, we invert the full matrix in \cref{eq:primal-dual-step} via an LU factorization. Hence, we do not refer to this solver as a PDAS, but nevertheless it is an SSN for the discretized problem. We apply this solver to \cref{prob:obstacle}. Unfortunately, as reported in \cref{tab:further}, the iteration count still roughly doubles with each uniform refinement. 

\subsection{Grid-sequencing}
\label{sec:further-experiments:gridsequencing}
By choosing an initial guess close to the solution, one can mitigate the effects of mesh dependence as fewer dofs must peel away from the obstacle during the deactivation phase. Here we investigate the effectiveness of grid-sequencing when applied to \cref{prob:obstacle} with $d=2$. The initial guess on each mesh is constructed by interpolating the solution on the previous mesh. On the coarsest mesh, we pick a zero initial guess. The HIK linear systems are solved via a Cholesky factorization. \cref{tab:further} reveals that we observe no growth in the iteration count with each uniform refinement. However, note that access to the grid-sequenced initial guess relies on solving the obstacle problem on all the parent meshes. Hence, comparing iteration counts on a mesh-by-mesh basis is problematic.

\subsection{Obstacle problem with biactivity}
\label{sec:further-experiments:biactive}

The solutions of some obstacle problems feature regions $ B \subseteq \Omega$ of \emph{biactivity} where, simultaneously, $\lambda(x) = 0$ and $u(x)=\varphi(x)$ for a.e.~$x \in B$. Physically, this is interpreted as the membrane $u$ exerting zero force on the obstacle $\varphi$ at that point. Solvers sometimes struggle to converge to biactive solutions. We borrow, and slightly modify, two biactive obstacle problems from \cite[Sec.~4.8]{keith2023}. 

We fix the square domain $\Omega = (-1,1)^2$ and have a zero obstacle $\varphi \equiv 0$. Let $\mathbb{I}_B$ denote the indicator function for some $B \subseteq \Omega$. In our first problem, we pick the Dirichlet boundary condition $u(x_1,x_2)|_{\partial \Omega} = g(x_1,x_2)$ and forcing term $f$ given by:
\begin{align}
g(x_1,x_2) = - \mathbb{I}_{\{x_1 \geq 0\}} \cdot x_1^4,
\quad
f(x_1,x_2)= \mathbb{I}_{\{x_1 \geq 0\}} \cdot (12x_1^2).
\label{eq:biactive:1}
\end{align}
This results in the solution $u(x_1,x_2) = g(x_1,x_2)$ and $\lambda \equiv 0$. 

The second obstacle problem we consider has the solution 
\begin{align}
u(x_1,x_2) = -\mathbb{I}_{\{x_1^2+x_2^2<1/4\}} \cdot (1-4x_1^2 - 4x_2^2)^4, \quad
\lambda(x_1,x_2) = \mathbb{I}_{\{x_1^2+x_2^2>3/4\}}.
\end{align}
This solution is induced by a zero Dirichlet boundary condition and the forcing term:
\begin{align}
\label{eq:biactive:2}
f(x_1,x_2) = -\Delta u(x_1,x_2) + \mathbb{I}_{\{x_1^2+x_2^2 >3/4\}}.
\end{align}
We record the HIK iteration growth in \cref{tab:further}. We observe similar growth as for the other obstacle problem examples.

\subsection{Optimal control}
\label{sec:further-experiments:optimal-control}

Let $\Omega = (0,1)^2$, $\beta = 10^{-2}$, and $u_d(x_1,x_2) \equiv 5$. Our first example, in optimal control, is to compute $(u,z) \in H^1_0(\Omega) \times L^2(\Omega)$ that, for all $q \in H^1_0(\Omega)$, minimizes
\begin{align}
\min_{u,z} \frac{1}{2} \| u - u_d \|^2_{L^2(\Omega)} + \frac{\beta}{2} \|z \|^2_{L^2(\Omega)} \;\; \text{subject to} \; \; (\nabla u, \nabla q) = (z,q),
\label{eq:oc1}
\end{align}
together with either the state constraint $u \leq 1$ or the control constraint $z \leq 1$ a.e.~in $\Omega$. The second example is to compute $(u,z) \in H^1_0(\Omega) \times H^1_0(\Omega)$ that, for all $q \in H^1_0(\Omega)$, minimizes
\begin{align}
\min_{u,z} \frac{1}{2} \| u - u_d \|^2_{L^2(\Omega)} + \frac{\beta}{2} \| \nabla z \|^2_{L^2(\Omega)} \;\; \text{subject to} \; \; (\nabla u, \nabla q) = (z,q),
\label{eq:oc2}
\end{align}
and the ``control" constraint $z \leq 1$ a.e.~in $\Omega$. For our last optimal control example, we consider, for all $q \in H^1_0(\Omega)$, 
\begin{align}
\min_{u,z} \frac{1}{2} \| u - u_d \|^2_{H^1(\Omega)} + \frac{\beta}{2} \|z \|^2_{L^2(\Omega)} \;\; \text{subject to} \; \; (\nabla u, \nabla q) = (z,q),
\label{eq:oc3}
\end{align}
with the state constraint $u \leq  1/5$ a.e.~in $\Omega$. We discretize $u$, $z$, and the equality constraint Lagrange multiplier $p$ with P1-FEM. Note that the stiffness matrix $A$ in \cref{alg:hik} is replaced by the iteration-independent Hessian associated with the Lagrangian of the optimal control problems. The HIK linear systems are solved via an LU factorization. We provide the HIK iteration counts in \cref{tab:further}.

The control constrained optimal control problem \cref{eq:oc1} leads to mesh-independent iteration counts. This is unsurprising as HIK applied to such problems is an SSN even on the infinite-dimensional level \cite[Sec.~4]{hintermuller2002}. All the other formulations feature exponentially growing iteration counts.

\begin{table}[t!]
\renewcommand{\arraystretch}{1.1}
\centering
\small
\begin{tabular}{|l|c|c|c|c|c|c|c|}
\hhline{-|-|-|-|-|-|-|-|}
\rowcolor{lightgray!10} \multicolumn{1}{|l|}{Obstacle problem ($n \times n$)} & $2^4$ & $2^5$ & $2^6$  &$2^7$ &$2^8$ &$2^9$ & $2^{10}$\\ 
 \hline
Mass matrix (\Cref{sec:further-experiments:primal-dual})& 8  & 10 &  15& 26& 45 & 85 & 166  \\
Grid-sequencing (\Cref{sec:further-experiments:gridsequencing})&   5 &  2 &  2 & 3& 3 & 3 & 4  \\
Biactive \cref{eq:biactive:1} (\Cref{sec:further-experiments:biactive}) & 7 & 15& 31& 63& 127& 255 & 510\\
Biactive \cref{eq:biactive:2} (\Cref{sec:further-experiments:biactive}) &  7& 10& 15& 26& 48& 93&181\\
\hline
\rowcolor{lightgray!10} \multicolumn{1}{|l|}{Optimal control ($n \times n$)} & $5$ & $10$ & $20$  &$40$ &$80$ &$160$ & $320$\\ 
 \hline
\cref{eq:oc1} \& $z \leq 1$ & 3  &4  &  4&3 &4 & 4&4  \\
\cref{eq:oc1} \& $u \leq 1$ & 2 &  7& 10 &25  &45 & 90 &  186\\
\cref{eq:oc2} \& $z \leq 1$ & 3  &3  &  6&10 &16 & 31& 59  \\
\cref{eq:oc3} \& $u \leq 1/5$ & 2  &  6& 10 & 20 & 37& 73& 142 \\
\hline
\end{tabular}
\caption{HIK iteration counts for the examples introduced in \Cref{sec:further-experiments}. The numbers in the greyed regions are the mesh resolution $n$.} 
\label{tab:further}
\end{table}

\bibliographystyle{siamplain}
\bibliography{references}

\begin{thebibliography}{10}

\bibitem{Adams2003}
{\sc R.~A. Adams and J.~J. Fournier}, {\em Sobolev spaces}, Elsevier,
  second~ed., 2003.

\bibitem{alphonse2025}
{\sc A.~Alphonse, P.~Dvurechensky, I.~P.~A. Papadopoulos, and C.~Sirotenko},
  {\em {LeAP-SSN: A semismooth Newton method with global convergence rates}},
  2025, \url{https://arxiv.org/abs/2508.16468}.

\bibitem{paraview}
{\sc U.~Ayachit}, {\em The {P}ara{V}iew guide: A Parallel Visualization
  Application}, Kitware, Inc., 2015.

\bibitem{petsc}
{\sc S.~Balay, S.~Abhyankar, M.~F. Adams, S.~Benson, J.~Brown, P.~Brune,
  et~al.}, {\em {PETSc/TAO} users manual}, Tech. Report ANL-21/39 - Revision
  3.21, Argonne National Laboratory, 2024,
  \url{https://doi.org/10.2172/2205494}.

\bibitem{banz2015}
{\sc L.~Banz and A.~Schr{\"o}der}, {\em Biorthogonal basis functions in
  $hp$-adaptive {FEM} for elliptic obstacle problems}, Computers \& Mathematics
  with Applications, 70 (2015), pp.~1721--1742,
  \url{https://doi.org/10.1016/j.camwa.2015.07.010}.

\bibitem{benson2006}
{\sc S.~J. Benson and T.~S. Munson}, {\em Flexible complementarity solvers for
  large-scale applications}, Optimization Methods and Software, 21 (2006),
  pp.~155--168, \url{https://doi.org/10.1080/10556780500065382}.

\bibitem{bergounioux1999}
{\sc M.~Bergounioux, K.~Ito, and K.~Kunisch}, {\em Primal-dual strategy for
  constrained optimal control problems}, SIAM Journal on Control and
  Optimization, 37 (1999), pp.~1176--1194,
  \url{https://doi.org/10.1137/S0363012997328609}.

\bibitem{bergounioux2002}
{\sc M.~Bergounioux and K.~Kunisch}, {\em Primal-dual strategy for
  state-constrained optimal control problems}, Computational Optimization and
  Applications, 22 (2002), pp.~193--224,
  \url{https://doi.org/10.1023/A:1015489608037}.

\bibitem{Bezanson2017}
{\sc J.~Bezanson, A.~Edelman, S.~Karpinski, and V.~B. Shah}, {\em Julia: {A}
  fresh approach to numerical computing}, SIAM Review, 59 (2017), pp.~65--98,
  \url{https://doi.org/10.1137/141000671}.

\bibitem{brandt1983}
{\sc A.~Brandt and C.~W. Cryer}, {\em Multigrid algorithms for the solution of
  linear complementarity problems arising from free boundary problems}, SIAM
  Journal on Scientific and Statistical Computing, 4 (1983), pp.~655--684,
  \url{https://doi.org/10.1137/0904046}.

\bibitem{Brenner2008}
{\sc S.~C. Brenner and L.~R. Scott}, {\em {The Mathematical Theory of Finite
  Element Methods}}, vol.~15 of Texts in Applied Mathematics, Springer New
  York, New York, NY, 3~ed., 2008,
  \url{https://doi.org/10.1007/978-0-387-75934-0}.

\bibitem{brezis1971}
{\sc H.~Br{\'e}zis}, {\em Nouveaux th{\'e}or{\`e}mes de r{\'e}gularit{\'e} pour
  les probl{\`e}mes unilat{\'e}raux}, Les rencontres
  physiciens-math{\'e}maticiens de Strasbourg-RCP25, 12 (1971), pp.~1--14.

\bibitem{brokate2022}
{\sc M.~Brokate and M.~Ulbrich}, {\em Newton differentiability of convex
  functions in normed spaces and of a class of operators}, SIAM Journal on
  Optimization, 32 (2022), pp.~1265--1287,
  \url{https://doi.org/10.1137/21M1449531}.

\bibitem{bueler2024}
{\sc E.~Bueler and P.~E. Farrell}, {\em A full approximation scheme multilevel
  method for nonlinear variational inequalities}, SIAM Journal on Scientific
  Computing, 46 (2024), pp.~A2421--A2444,
  \url{https://doi.org/10.1137/23M1594200}.

\bibitem{chen1996}
{\sc C.~Chen and O.~L. Mangasarian}, {\em A class of smoothing functions for
  nonlinear and mixed complementarity problems}, Computational Optimization and
  Applications, 5 (1996), pp.~97--138,
  \url{https://doi.org/10.1007/BF00249052}.

\bibitem{de1996}
{\sc T.~De~Luca, F.~Facchinei, and C.~Kanzow}, {\em A semismooth equation
  approach to the solution of nonlinear complementarity problems}, Mathematical
  programming, 75 (1996), pp.~407--439,
  \url{https://doi.org/10.1007/BF02592192}.

\bibitem{dirkse1995}
{\sc S.~P. Dirkse and M.~C. Ferris}, {\em The {PATH} solver: a non-monotone
  stabilization scheme for mixed complementarity problems}, Optimization
  Methods and Software, 5 (1995), pp.~123--156,
  \url{https://doi.org/10.1080/10556789508805606}.

\bibitem{dirkse1997}
{\sc S.~P. Dirkse and M.~C. Ferris}, {\em Crash techniques for large-scale
  complementarity problems}, Complementarity and Variational Problems: State of
  the Art, 92 (1997), p.~40.

\bibitem{dokken2025}
{\sc J.~S. Dokken, P.~E. Farrell, B.~Keith, I.~P. Papadopoulos, and T.~M.
  Surowiec}, {\em The latent variable proximal point algorithm for variational
  problems with inequality constraints}, Computer Methods in Applied Mechanics
  and Engineering, 445 (2025), p.~118181,
  \url{https://doi.org/10.1016/j.cma.2025.118181}.

\bibitem{ern2021:i}
{\sc A.~Ern and J.-L. Guermond}, {\em Finite Elements {I}: Approximation and
  Interpolation}, Springer, 2021,
  \url{https://doi.org/10.1007/978-3-030-56341-7}.

\bibitem{Evans2010}
{\sc L.~C. Evans}, {\em {Partial Differential Equations}}, American
  Mathematical Society, 2~ed., 2010.

\bibitem{falk1974}
{\sc R.~S. Falk}, {\em Error estimates for the approximation of a class of
  variational inequalities}, Mathematics of Computation, 28 (1974),
  pp.~963--971, \url{https://doi.org/10.1090/S0025-5718-1974-0391502-8}.

\bibitem{farrell2020}
{\sc P.~E. Farrell, M.~Croci, and T.~M. Surowiec}, {\em Deflation for
  semismooth equations}, Optimization Methods and Software, 35 (2020),
  pp.~1248--1271, \url{https://doi.org/10.1080/10556788.2019.1613655}.

\bibitem{fichera1964}
{\sc G.~Fichera}, {\em Problemi elastostatici con vincoli unilaterali: il
  problema di {S}ignorini con ambigue condizioni al contorno}, Accademia
  nazionale dei Lincei, 1964.

\bibitem{graser2009}
{\sc C.~Gr{\"a}ser and R.~Kornhuber}, {\em Multigrid methods for obstacle
  problems}, Journal of Computational Mathematics,  (2009), pp.~1--44.

\bibitem{grisvard2011}
{\sc P.~Grisvard}, {\em Elliptic problems in nonsmooth domains}, SIAM, 2011,
  \url{https://doi.org/10.1137/1.9781611972030}.

\bibitem{hackbusch1983}
{\sc W.~Hackbusch and H.~D. Mittelmann}, {\em On multi-grid methods for
  variational inequalities}, Numerische Mathematik, 42 (1983), pp.~65--76,
  \url{https://doi.org/10.1007/BF01400918}.

\bibitem{harker1990}
{\sc P.~T. Harker and B.~Xiao}, {\em {Newton's method for the nonlinear
  complementarity problem: A B-differentiable equation approach}}, Mathematical
  Programming, 48 (1990), pp.~339--357,
  \url{https://doi.org/10.1007/BF01582262}.

\bibitem{hintermuller2002}
{\sc M.~Hinterm{\"u}ller, K.~Ito, and K.~Kunisch}, {\em The primal-dual active
  set strategy as a semismooth {N}ewton method}, SIAM Journal on Optimization,
  13 (2002), pp.~865--888, \url{https://doi.org/10.1137/S1052623401383558}.

\bibitem{hintermuller2003}
{\sc M.~Hinterm{\"u}ller, V.~A. Kovtunenko, and K.~Kunisch}, {\em Semismooth
  {N}ewton methods for a class of unilaterally constrained variational
  problems}, Universit{\"a}t Graz/Technische Universit{\"a}t Graz. SFB
  F003-Optimierung und Kontrolle,  (2003).

\bibitem{hintermuller2006feasible}
{\sc M.~Hinterm{\"u}ller and K.~Kunisch}, {\em Feasible and noninterior
  path-following in constrained minimization with low multiplier regularity},
  SIAM Journal on Control and Optimization, 45 (2006), pp.~1198--1221,
  \url{https://doi.org/10.1137/050637480}.

\bibitem{hintermuller2004mesh}
{\sc M.~Hinterm{\"u}ller and M.~Ulbrich}, {\em A mesh-independence result for
  semismooth {N}ewton methods}, Mathematical Programming, 101 (2004),
  pp.~151--184, \url{https://doi.org/10.1007/s10107-004-0540-9}.

\bibitem{hoppe1987}
{\sc R.~H. Hoppe}, {\em Multigrid algorithms for variational inequalities},
  SIAM Journal on Numerical Analysis, 24 (1987), pp.~1046--1065,
  \url{https://doi.org/10.1137/0724069}.

\bibitem{hoppe1994}
{\sc R.~H. Hoppe and R.~Kornhuber}, {\em Adaptive multilevel methods for
  obstacle problems}, SIAM Journal on Numerical Analysis, 31 (1994),
  pp.~301--323, \url{https://doi.org/10.1137/0731016}.

\bibitem{hueber2005}
{\sc S.~H{\"u}eber and B.~I. Wohlmuth}, {\em A primal--dual active set strategy
  for non-linear multibody contact problems}, Computer Methods in Applied
  Mechanics and Engineering, 194 (2005), pp.~3147--3166,
  \url{https://doi.org/10.1016/j.cma.2004.08.006}.

\bibitem{ito2003}
{\sc K.~Ito and K.~Kunisch}, {\em Semi--smooth {N}ewton methods for variational
  inequalities of the first kind}, ESAIM: Mathematical Modelling and Numerical
  Analysis, 37 (2003), pp.~41--62, \url{https://doi.org/10.1051/m2an:2003021}.

\bibitem{ito2008}
{\sc K.~Ito and K.~Kunisch}, {\em Lagrange multiplier approach to variational
  problems and applications}, SIAM, 2008,
  \url{https://doi.org/10.1137/1.9780898718614}.

\bibitem{keith2023}
{\sc B.~Keith and T.~M. Surowiec}, {\em Proximal {G}alerkin: {A}
  structure-preserving finite element method for pointwise bound constraints},
  Foundations of Computational Mathematics,  (2024), pp.~1--97,
  \url{https://doi.org/10.1007/s10208-024-09681-8}.

\bibitem{kinderlehrer2000}
{\sc D.~Kinderlehrer and G.~Stampacchia}, {\em An introduction to variational
  inequalities and their applications}, SIAM, 2000,
  \url{https://doi.org/10.1137/1.9780898719451}.

\bibitem{kirby2024}
{\sc R.~C. Kirby and D.~Shapero}, {\em High-order bounds-satisfying
  approximation of partial differential equations via finite element
  variational inequalities}, Numerische Mathematik,  (2024), pp.~1--21,
  \url{https://doi.org/10.1007/s00211-024-01405-y}.

\bibitem{kornhuber1994}
{\sc R.~Kornhuber}, {\em Monotone multigrid methods for elliptic variational
  inequalities {I}}, Numerische Mathematik, 69 (1994), pp.~167--184,
  \url{https://doi.org/10.1007/BF03325426}.

\bibitem{kornhuber1996}
{\sc R.~Kornhuber}, {\em Monotone multigrid methods for elliptic variational
  inequalities {II}}, Numerische Mathematik, 72 (1996), pp.~481--499,
  \url{https://doi.org/10.1007/s002110050178}.

\bibitem{munson2001}
{\sc T.~S. Munson, F.~Facchinei, M.~C. Ferris, A.~Fischer, and C.~Kanzow}, {\em
  The semismooth algorithm for large scale complementarity problems}, INFORMS
  Journal on Computing, 13 (2001), pp.~294--311,
  \url{https://doi.org/10.1287/ijoc.13.4.294.9734}.

\bibitem{papadopoulos2021}
{\sc I.~P. Papadopoulos, P.~E. Farrell, and T.~M. Surowiec}, {\em Computing
  multiple solutions of topology optimization problems}, SIAM Journal on
  Scientific Computing, 43 (2021), pp.~A1555--A1582,
  \url{https://doi.org/10.1137/20M1326209}.

\bibitem{obstacle.jl}
{\sc I.~P.~A. Papadopoulos}, {\em {PrimalDualActiveSet.jl}}, 2025,
  \url{https://github.com/ioannisPApapadopoulos/PrimalDualActiveSet.jl}.

\bibitem{zenodo}
{\sc I.~P.~A. Papadopoulos}, {\em {PrimalDualActiveSet.jl, v0.1.0 (Zenodo)}},
  2026, \url{https://doi.org/10.5281/zenodo.21382741}.

\bibitem{schoberl1998}
{\sc J.~Sch{\"o}berl}, {\em Solving the {S}ignorini problem on the basis of
  domain decomposition techniques}, Computing, 60 (1998), pp.~323--344,
  \url{https://doi.org/10.1007/BF02684379}.

\bibitem{schwedes2017mesh}
{\sc T.~Schwedes, D.~A. Ham, S.~W. Funke, and M.~D. Piggott}, {\em Mesh
  dependence in {PDE}-constrained optimisation}, Springer, 2017,
  \url{https://doi.org/10.1007/978-3-319-59483-5}.

\bibitem{stampacchia1964}
{\sc G.~Stampacchia}, {\em Formes bilineaires coercitives sur les ensembles
  convexes}, Comptes Rendus Hebdomadaires Des Seances De L Academie Des
  Sciences, 258 (1964), p.~4413.

\bibitem{ulbrich2002semismooth}
{\sc M.~Ulbrich}, {\em Semismooth {N}ewton methods for operator equations in
  function spaces}, SIAM Journal on Optimization, 13 (2002), pp.~805--841,
  \url{https://doi.org/10.1137/S1052623400371569}.

\end{thebibliography}

\end{document}